\documentclass[11pt]{amsart}
\usepackage{amsmath,amssymb,mathrsfs,euscript,color,mathabx,stmaryrd,mathabx}
\usepackage[colorlinks]{hyperref}
\usepackage[shortlabels]{enumitem}
\usepackage{pdfsync}
\usepackage[backgroundcolor=green!40, linecolor=green!40]{todonotes}
\usepackage[utf8]{inputenc}                                    
\usepackage[T1]{fontenc}

\input{xypic}
\usepackage{tikz-cd}

\usepackage[margin=1.2in]{geometry}

\newcounter{makeconstant}
\newenvironment{makeconstant}%
{\refstepcounter{makeconstant}}%
{}

\newcounter{nonumber}

\theoremstyle{plain}
\newtheorem{theorem}[equation]{Theorem}
\newtheorem{lemma}[equation]{Lemma}
\newtheorem{corollary}[equation]{Corollary}
\newtheorem{proposition}[equation]{Proposition}

\newtheorem{conjecture}[equation]{Conjecture}

\theoremstyle{definition}

\newtheorem{definition}[equation]{Definition}

\newtheorem{remark}[equation]{Remark}

\newtheorem*{remark*}{Remark}

\numberwithin{equation}{section} 

\def\A{{\rm A}}
\def\B{{\rm B}}

\def\D{{\rm D}}
\def\E{{\rm E}}
\def\F{{\rm F}}
\def\G{{\rm G}}
\def\H{{\mathrm{H}}}
\def\I{{\rm I}}
\def\J{{\rm J}}
\def\K{{\rm K}}
\def\L{{\rm L}}
\def\M{{\rm M}}
\def\N{{\rm N}}

\def\P{{\rm P}}
\def\Q{{\rm Q}}
\def\R{{\rm R}}

\def\T{{\rm T}}
\def\U{{\rm U}}
\def\V{{\rm V}}
\def\W{{\rm W}}
\def\X{{\rm X}}

\def\Cc{\mathscr{C}}

\def\kk{\mathfrak{k}}

\def\so{\mathsf{o}}

\def\o{\mathfrak{o}}
\def\p{\mathfrak{p}}

\def\t{\theta}

\def\({\left(}
\def\){\right)}
\def\>{\geqslant}
\def\<{\leqslant}

\def\Spec{\operatorname{Spec}}
\def\Hom{\operatorname{Hom}}
\def\End{\operatorname{End}}
\def\Aut{\operatorname{Aut}}

\def\GL{\operatorname{GL}}

\def\ker{\operatorname{ker}}

\def\id{\operatorname{id}}

\def\Ind{\operatorname{Ind}}

\def\ind{\operatorname{ind}}

\def\Irr{\operatorname{Irr}}

\def\dim{\operatorname{dim}}

\def\st{\operatorname{st}}

\def\Sp{\operatorname{Sp}}

\def\SL{\operatorname{SL}}

\def\tG{{\widetilde{\G}}}

\definecolor{dorange}{RGB}{255,140,0}
\definecolor{dgreen}{RGB}{0,205,10}

\def\ignore#1{\relax}

\def\bdots{\mathinner{\mkern1mu\raise1pt\hbox{.}\mkern2mu\raise4pt\hbox{.}
           \mkern2mu\raise7pt\vbox{\kern7pt\hbox{.}}\mkern1mu}}

\def\presuper#1#2%
  {\mathop{}%
   \mathopen{\vphantom{#2}}^{#1}%
   \kern-\scriptspace%
   #2}
\def\presub#1#2%
  {\mathop{}%
   \mathopen{\vphantom{#2}}_{#1}%
   \kern-\scriptspace%
   #2}

\def\ft{{\mathfrak{t}}} 

\def\LG{{\tensor*[^L]\G{^\so}}}

\def\LL{\mathrm{LL}}
\def\Cusp{\mathrm{Cusp}}

\def\Ql{\overline{\mathbb{Q}}_\ell}

\newcommand{\bext}{\mathrm{beta}}

\usepackage{bbm}

\def\ft{\mathfrak{t}}

\def\kk{\mathrm{k}}

\def\Rep{\mathrm{Rep}}
\def\Fl{\overline{\mathbb{F}}_\ell}

\def\LG{\presuper{L}\G}
\def\Zl{\overline{\mathbb{Z}}_\ell} 
\def\Gf{\mathtt{G}}

\def\Ql{\overline{\mathbb{Q}}_\ell}
\def\cW{\mathscr{W}}
\def\presuper#1#2%
  {\mathop{}%
   \mathopen{\vphantom{#2}}^{#1}%
   \kern-\scriptspace%
   #2}
   
 \def\LL{\text{LL}}  
   
   \theoremstyle{definition}

\def\Div{\mathtt{Div}}
\def\Herm{\mathtt{Herm}}

\usepackage{subfiles}

\title[Block decompositions for $p$-adic classical groups]{Block decompositions for $p$-adic classical groups and their inner forms}
\date{\today}
\author[D.~Helm]{David Helm}
\address{David Helm, Department of Mathematics, Imperial College, London, SW7 2AZ, United Kingdom.}
\email{d.helm@imperial.ac.uk }
\author[R.~Kurinczuk]{Robert Kurinczuk}
\address{Robert Kurinczuk, School of Mathematics and Statistics, University of Sheffield, Sheffield, S3 7RH, United Kingdom.}
\email{robkurinczuk@gmail.com}
\author[D.~Skodlerack]{Daniel Skodlerack}
\address{Daniel Skodlerack, Institute of Mathematical Sciences, ShanghaiTech University,
201210, Pudong New District, Shanghai, China}
\email{dskodlerack@shanghaitech.edu.cn}
\author[S.~Stevens]{Shaun Stevens}
\address{Shaun Stevens, School of Mathematics, University of East Anglia, Norwich, NR4 7TJ, United~Kingdom.}
\email{shaun.stevens@uea.ac.uk}

\begin{document}
\begin{abstract}
For an inner form~$\G$ of a general linear group or classical group over a non-archimedean local field of odd residue characteristic, we decompose the category of smooth representations on~$\mathbb{Z}[1/p,\mu_{p^{\infty}}]$-modules by endo-parameter.  We prove that parabolic induction preserves these decompositions, and hence that it preserves endo-parameters.  Moreover, we show that the decomposition by endo-parameter is the~$\overline{\mathbb{Z}}[1/p]$-block decomposition; and, for $\R$ an integral domain and~$\mathbb{Z}[1/p,\mu_{p^{\infty}}]$-algebra, introduce a graph whose connected components parameterize the~$\R$-blocks, in particular including the cases~$\R=\overline{\mathbb{Z}}_{\ell}$ and~$\R=\overline{\mathbb{F}}_\ell$ for~$\ell\neq p$.  From our description, we deduce that the~$\Zl$-blocks and~$\Fl$-blocks of~$\G$ are in natural bijection, as had long been expected.  Our methods also apply to the trivial endo-parameter (i.e., the depth zero subcategory) of any connected reductive~$p$-adic group, providing an alternative approach to results of Dat and Lanard in depth zero. Finally, under a technical assumption (known for inner forms of general linear groups) we reduce the~$\R$-block decomposition of~$\G$ to depth zero.
\end{abstract}
\maketitle

\setcounter{tocdepth}{1}
\tableofcontents

\section{Introduction}

\subsection{} 
In \cite{MR771671}, Bernstein shows that the abelian category~$\Rep_{\mathbb{C}}(\G)$ of smooth representations on complex vector spaces of a $p$-adic connected reductive group~$\G$ decomposes as a direct product~$\Rep_{\mathbb{C}}(\G)=\prod \Rep_{\mathbb{C}}(\mathfrak{s})$, of full abelian indecomposable subcategories  $\Rep_{\mathbb{C}}(\mathfrak{s})$ indexed by inertial classes of supercuspidal representations of Levi subgroups of $\G$.  If we replace $\mathbb{C}$ with an arbitrary algebraically closed field of characteristic~$\ell\neq p$ and consider~$\ell$-modular representations of~$\G$, or more ambitiously consider the category~$\Rep_{\R}(\G)$ of representations of~$\G$ over an arbitrary integral domain~$\R$ in which $p$ is invertible, such a precise decomposition remains unknown.  

This is the $p$-adic analogue of Brauer's theory of blocks in the modular representation theory of finite reductive groups, which has found numerous applications in finite group theory, but can be difficult to compute as soon as the order of the finite group is not invertible in~$\R$.

\subsection{}\label{para1.2}
Recently, there has been renewed interest in computing block decompositions of~$\Rep_{\R}(\G)$, the category of smooth representations on~$\R$-modules for arbitrary~$\mathbb{Z}[1/p]$-algebras~$\R$, partly motivated by conjectured categorical forms of the local Langlands correspondence and local Langlands in families, see for example \cite{FarguesScholze} and \cite{DHKM2}.

In the following cases, complete block decompositions are known:
 \begin{enumerate}
 \item\label{GLnknowndecomps} $\Rep_{\Fl}(\GL_n(\F))$ due to Vign\'eras \cite{VignerasSelecta}; and~$\Rep_{\W(\Fl)}(\GL_n(\F))$ due to the first author \cite{HelmForum}.
 \item\label{GLmDknowndecomps} $\Rep_{\Fl}(\GL_m(\D))$, where~$\D$ is an~$\F$-central division algebra, due to S\'echerre and the fourth author \cite{SecherreStevens}. 
 \item $\Rep_{\overline{\mathbb{Z}}[1/\mathrm{N}_\G]}(\G)$ -- the \emph{banal case} -- due to Dat, Moss and the first and second authors \cite{DHKM2}, where~$\mathrm{N}_{\G}$ is the product of all primes which divide the pro-order of a compact open subgroup of~$\G$.
 \end{enumerate}
However all of these results rely on the \emph{uniqueness of supercuspidal support} of an irreducible representation of a ($p$-adic or finite) group on an~$\Fl$-vector space in the various special cases.  Along similar lines, Cui shows uniqueness of supercuspidal support for~$\SL_n(\F)$ \cite{Cui1}, and gives a fine decomposition of~$\Rep_{\Fl}(\SL_n(\F))$, which on the supercuspidal subcategory~$\Rep_{\Fl}(\SL_n(\F))^{sc}$ she refines to the block decomposition \cite{Cui2}, and recently in the tame case \cite{Cuitameblocks} to the block decomposition.  

For symplectic $p$-adic groups, however, uniqueness of (mod~$\ell$) supercuspidal support fails in general: a counterexample for~$\Sp_8(\mathbb{F}_q)$ when~$\ell \mid q^2+1$ was given by Dudas \cite{DudasSC}, which was lifted to a counterexample for~$p$-adic~$\Sp_8(\F)$ by Dat in \cite{DatDudas}.  Conversely, for a $p$-adic classical group, in the special case where one has uniqueness of supercuspidal support in associated finite classical groups, to follow the approach taken in \cite{KurinczukU21} to try to lift this uniqueness to the $p$-adic classical group one needs to start from the result that ``endo-parameters are compatible with parabolic induction'', a consequence of the first main theorem of this paper.

\subsection{}
Partial decompositions of the category have been given in other cases:  Let~$\R$ be a~$\mathbb{Z}[\mu_{p^{\infty}},1/p]$-algebra and~$\G$ be a connected reductive~$p$-adic group.  In \cite[I 5.8]{Vig96} and \cite[Appendix]{Dat09}, Vign\'eras and Dat prove that there is a decomposition of $\Rep_\R(\G)$ as a direct product of full subcategories~\[\Rep_{\R}(\G)=\prod_{r\in\mathbb{Q}} \Rep_{\R}(\G)_r,\] according to the \emph{depth} of a representation.

The depth zero factor has been further studied:  Using Deligne--Lusztig theory, Lanard has produced fine decompositions of the full abelian subcategory~$\Rep_{\Zl}(\G)_0$ of depth zero representations \cite{Lanard1,Lanard2,LanardUnipotent} on~$\Zl$-modules, and in many cases Dat and Lanard \cite{DatLanard} have shown that~$\Rep_{\overline{\mathbb{Z}}[1/p]}(\G)_0$ is a~$\overline{\mathbb{Z}}[1/p]$-block.  As one might expect and our results show, the case of positive depth is more complex; for example, the depth $r$ factor~$\Rep_{\overline{\mathbb{Z}}[1/p]}(\G)_r$ is far from a~$\overline{\mathbb{Z}}[1/p]$-block.  However, granted the finer decomposition by ``endo-parameter'' (see Theorem~\ref{firstmaintheorem} below) we find for inner forms of general linear and classical groups, one might hope for a reduction to depth zero showing an arbitrary $\overline{\mathbb{Z}}[1/p]$-block parameterized by an endo-parameter is equivalent (as categories) to a depth zero block in a related group -- see Conjecture \ref{equivconjecture} for a precise conjecture which follows for inner forms of general linear groups by work of Chinello \cite{Chinello}.  This conjecture should be compatible with a ``reduction to the tame case'' via the local Langlands programme, and makes explicit on the representation theory side conjectures of Dat \cite{DatFunctoriality}.

\subsection{}
We now describe more precisely the contents of this paper.
Let~$\R$ be a~$\mathbb{Z}[\mu_{p^\infty},1/p]$-algebra, and~$\G$ an inner form of a~$p$-adic general linear group of a classical $p$-adic group (symplectic, special orthogonal, or unitary) with $p\neq 2$.  Note that, although we write ``$p$-adic'', we also allow the underlying nonarchimedean field to have positive characteristic as well as characteristic zero.

 From the constructions of \cite{KSS} and \cite{SkodInnerFormII}, we have a fine invariant of an irreducible $\R$-representation called its \emph{endo-parameter}; in positive depth this is a sophisticated refinement of the depth of an irreducible~$\R$-representation.  Our first main result then computes the $\overline{\mathbb{Z}}[1/p]$-block decomposition for inner forms of classical groups:

\begin{theorem}[{Theorem \ref{endosplit}, Corollary \ref{EPCor}}]\label{firstmaintheorem}
Let~$\R$ be a~$\mathbb{Z}[\mu_{p^\infty},1/p]$-algebra, and~$\G$ be an inner form of a $p$-adic general linear group (for any $p$), or of a $p$-adic classical group with $p$ not 2.  We have a decomposition of categories
\[\Rep_\R(\G)=\prod_{\ft}\Rep_\R(\ft)\] 
where the product is taken over all endo-parameters for $\G$ and~$\Rep_\R(\ft)$ denotes the full subcategory of representations all of whose irreducible subquotients have endo-parameter~$\ft$.    Moreover:
\begin{enumerate}
\item \label{firstmaintheoremi}If~$\R\subseteq \overline{\mathbb{Z}}[1/p]$ then this is the~$\R$-block decomposition;
\item The compactly induced representation~$\P(\ft)\otimes\R$ (of Definition \ref{Progen}) is a finitely generated projective generator of~$\Rep_\R(\ft)$;
\item Parabolic induction and restriction are compatible with these decompositions.
\end{enumerate}
\end{theorem}

\subsection{}
Our proof of the decomposition~$\Rep_\R(\G)=\prod_{\ft}\Rep_\R(\ft)$ is loosely based on the proof of the decomposition by depth, and requires us to establish certain rigidity properties of semisimple characters (cf.,~\cite{technicalpaper}) and to extend intertwining computations and the construction of Heisenberg representations associated to semisimple characters to the setting of arbitrary~$\mathbb{Z}[\mu_{p^\infty},1/p]$-algebras (Section \ref{Heisenberg}).  
%
%

\subsection{}
Given the decomposition by endo-parameter, valid for any~$\mathbb{Z}[\mu_{p^\infty},1/p]$-algebra, the next natural question is how to refine it to produce the block decomposition in various cases of interest.  Our method here is also used to show that the decomposition by endo-parameter is the $\overline{\mathbb{Z}}[1/p]$-block decomposition, i.e., Theorem \ref{firstmaintheorem} \ref{firstmaintheoremi}.  

For an integral domain~$\R/\mathbb{Z}[\mu_{p^\infty},1/p]$, to each endo-factor~$\Rep_\R(\ft)$ using the representation theory of finite general linear and classical groups over~$\R$ (or associated reductive finite groups in depth zero) we associate a graph we call the (\emph{fine})~\emph{$(\ft,\R)$-graph}.  The vertices of this graph are given by certain finitely generated projective representations in~$\Rep_{\R}(\ft)$, and the connected components of the graph parametrize the blocks of~$\Rep_{\R}(\mathfrak{t})$. 
  For example, our method gives a process to compute the block decomposition of~$\Rep_{\Zl}(\G)$ for a prime~$\ell\neq p$, and the statement that~$\Rep_{\overline{\mathbb{Z}}[1/p]}(\ft)$ is a~$\overline{\mathbb{Z}}[1/p]$-block is equivalent to the~$(\ft,\overline{\mathbb{Z}}[1/p])$-graph being connected.  More precisely, we show:

\begin{theorem}[{Theorem \ref{maintheoremblocks}}]
Suppose that either~$\G$ is an inner form of a $p$-adic general linear group (for any $p$), or of a $p$-adic classical group with $p$ not 2, or that~$\G$ is any reductive~$p$-adic group and~$\mathfrak{t}$ is trivial (i.e., the depth zero case). 
\begin{enumerate}
\item The~blocks of~$\Rep_{\R}(\mathfrak{t})$ are in natural bijection with the connected components of the $(\ft,\R)$-graph.
\item The projective module defined as the direct sum over the vertices in a connected component of the $(\ft,\R)$-graph is a finitely generated projective generator of an~$\R$-block, and running over the connected components defines the decomposition of~$\P(\ft)\otimes\R$ into~$\R$-blocks.
\end{enumerate}
\end{theorem}

To define the~$(\mathfrak t,\R)$-graph, we decompose our explicitly constructed finitely generated projective generator for~$\Rep_{\R}(\mathfrak{t})$ into summands which constitute the vertices in the $(\ft,\R)$-graph and draw an edge between two such summands~$\Pi_1,\Pi_2$ if there is a non-zero morphism~$\Pi_1\rightarrow \Pi_2$ (or equivalently there is a non-zero morphism~$\Pi_2\rightarrow\Pi_1$).  

To compute when vertices are in the same connected component our approach is via type theory over~$\mathbb{C}$.  The first aim of the theory of types over~$\mathbb{C}$ is to construct a compactly induced progenerator~$\ind_{\J}^{\G}(\lambda)$ for a Bernstein factor~$\Rep_{\mathbb{C}}(\mathfrak{s})$ where $(\J,\lambda)$ is a pair -- called an~$\mathfrak{s}$-\emph{type} -- consisting of a compact open subgroup~$\J$ of~$\G$ and an irreducible smooth representation~$\lambda$ of~$\J$.  This is accomplished for classical groups in \cite{MiSt}, inner forms of general linear groups in \cite{SecherreStevensVI}, and inner forms of classical groups in \cite{SkodlerackYe}.  Our progenerator~$\P(\ft)$ (Definition \ref{Progen}) is constructed as a finite direct sum~$\bigoplus \ind_{\J^1_i}^{\G}(\eta_i)$ where~$\J^1_i$ is a finite index pro-$p$ subgroup of~$\J_i$ which is a~``$\J$-group'' used in the construction of types, and~$\eta_i$ is a particular irreducible (Heisenberg) representation of~$\J_i^1$.  The vertices in our graph are the representations~$\ind_{\J_i}^{\G}(P)$ where~$P$ is an indecomposable summand of~$\ind_{\J_i^1}^{\J_i}(\eta_i)$.   Suppose~for simplicity here that~$\R$ is contained in~$\mathbb{C}$, then we show that there is a non-zero morphism~$\ind_{\J_i}^{\G}(P)\rightarrow ~\ind_{\J_j}^{\G}(P')$ if and only if~$\ind_{\J_i}^{\G}(P)\otimes \mathbb{C}$ and~$\ind_{\J_i'}^{\G}(P')\otimes \mathbb{C}$ have nonzero components in a common Bernstein block.  This gives a recipe to compute the~$\R$-block decomposition from the~$\R$-block decomposition of the finite reductive groups~$\J_i/\J_i^1$ and from the understanding of when two $\mathfrak{s}$-types parameterize the same Bernstein component over~$\mathbb{C}$.  (In fact, as we choose our~$\mathfrak{s}$-types following the recipes of \cite{MiSt,SkodlerackYe} starting with the~$\eta_i$, and as the~$\eta_i$ are chosen related to each other, this question is weaker than knowing when two arbitrary~$\mathfrak{s}$-types parameterize the same Bernstein component.).  In the special case of inner forms of general linear groups, our results allow one to give different proofs of the known decompositions of \ref{GLnknowndecomps} and \ref{GLmDknowndecomps} of Paragraph~\ref{para1.2}, as well as in these cases computing the block decomposition over other integral domains.

\subsection{}
Under a technical hypothesis, on choosing \emph{compatible extensions} of the~$\eta_i$ to~$\J_i$ with strong intertwining properties, in Theorem \ref{reductiontodepthzeromaintheorem} we reduce the block decomposition to computing the block decomposition in depth zero of a related group.  

This technical hypothesis follows from standard arguments for inner forms of general linear groups (the details will appear in \cite{SkodlerackYe}).  For inner forms of classical groups we also expect this technical hypothesis to be addressed in broad generality in the work in progress of \cite{SkodlerackYe}.

\subsection{}
The second aim of the theory of types over~$\mathbb{C}$, having constructed a type~$(\J,\lambda)$ for a Bernstein block is to compute the associated Hecke algebra~$\mathcal{H}(\J,\lambda)=\End_{\mathbb{C}[\G]}(\ind_{\J}^{\G}(\lambda))$ and its module category~$\mathcal{H}(\J,\lambda)$-mod.  By standard category theory, the Bernstein block is then isomorphic to the category of (right) modules over~$\mathcal{H}(\J,\lambda)$ and one has ``completely'' described the block if one understands the Hecke algebra.  In the setting of~$\R$-representations over general~$\mathbb{Z}[1/p]$-algebras, one can pursue an analogous strategy, however the projective generators of blocks are in general much larger and, other than particularly simple cases, we expect it best to first pursue a strategy of reduction to depth zero (i.e., establish Conjecture \ref{equivconjecture}), and then further reduction to the unipotent block setting as predicted by \cite{DatFunctoriality}. 

Understanding the unipotent block is quite a different question. For general linear groups, Vign\'eras~\cite{VigSchur} constructs a \emph{Schur algebra} whose modules describe the subcategory of the unipotent block consisting of representations annihilated by a certain ideal of the global Hecke algebra, and moreover proves that some power of this ideal annihilates the whole unipotent block. Going beyond this seems difficult, though a first step is made at a derived level by Berry~\cite{BerryGL2} in the case of~$\GL_2(\F)$ and~$R=\bar F_\ell$, in the non-banal case of odd~$\ell$ dividing~$q_{\F}+1$.

\subsection{}
We finally turn to the interpretation of our results on the block decomposition over~$\overline{\mathbb{Z}}[1/p]$ in terms of Langlands parameters, where we list some direct consequences of our results and of conjectures in the area.  In this section we only consider the \emph{stable block} decomposition, i.e., the finest coarsening of the block decomposition where $L$-equivalent~$\mathbb{C}$-representations are in the same component.

Using the ramification theorem for~$\Sp_{2n}(\F)$ of the fourth author, Blondel, and Henniart, we can rephrase the stable block decomposition in terms of restriction to wild inertia, see Corollary \ref{corgaloisdecomp}.  Part of the motivation for this work is to approach conjectures relating categories of sheaves on moduli spaces of Langlands parameters to representations of $p$-adic groups, first decomposing both sides and studying the potentially simpler task of matching these decompositions and approaching the conjectures block by block.  In Section \ref{LLIFsection}, in the semisimple setting we turn this around, assume a conjectural \emph{local Langlands in families}, and deduce properties on endo-parameters from this; we hope to approach these properties directly in future work.

\subsection{}
One major open question remaining is to what extent the results here could be generalised beyond the groups considered here. What one would like is a theory of endo-parameters for an arbitrary reductive group~$\G$ which on the one hand gives a decomposition of~$\Rep_{\R}(\G)$ for any~$\overline{\mathbb{Z}}[1/p]$-algebra~$\R$, and on the other is compatible with restriction to wild inertia via the Langlands correspondence (where one retains the monodromy information). In the tame case, using the constructions of Yu~\cite{Yu} and Fintzen~\cite{FintzenYu,FintzenTypes}, Dat has at least most of this in~\cite{DatET} in terms of wild Langlands parameters (see especially Theorem~2.7.1), but in general it is completely open. It would also be desirable for the theory of endo-parameters to be functorial in the group, as one sees between symplectic and general linear groups in~\cite{BHS}. Thus there is much still to be done.

\subsection*{Structure of the paper}
In Section \ref{secPreliminaries}, we set notation, and develop some basic results on representations of~$p$-adic groups over~$\mathbb{Z}[1/p]$-algebras.  Section \ref{secEPs} recalls the theory of endo-parameters of \cite{KSS,SkodInnerFormII} in a special case (which is sufficient for this paper and most applications) and the main results of \cite{technicalpaper}.   
Section \ref{Heisenberg} develops the theory of semisimple characters and Heisenberg representations integrally.  Section \ref{secESing} establishes our basic decomposition of~$\Rep_{\mathbb{Z}[\mu_{p^\infty},1/p]}(\G)$ by endo-parameter.  Section \ref{secbetaextsandtypes} recalls the construction of beta extensions of Heisenberg representations over~an algebraically closed field~$\K$ of characteristic different to~$p$, and the construction of types for Bernstein blocks when~$\K$ has characteristic zero.  In Section \ref{secblocks}, we prove our results on block decompositions.  Section \ref{secLPs} is a non-technical section which explains connections between our decompositions and the local Langlands programme. 

\subsection{History of the ideas in the paper}
This work has been a long time coming.  The original decomposition of Theorem \ref{endosplit}, with compatibility with parabolic induction and restriction, and the reduction to depth zero expectation were announced before 2016; and for example, these results are mentioned as work in progress of the authors in Lanard's 2017 arXiv-ed introduction of \cite{Lanard1} and in the 2016 arXiv version of \cite{KSS}.  A very nice result of Dat and Lanard \cite{DatLanard}, which they use to show the depth zero~$\overline{\mathbb{Z}}[1/p]$-subcategory is a block in many cases, then gave us the key finite group theory ingredient needed to establish that this decomposition is the~$\overline{\mathbb{Z}}[1/p]$-block decomposition, completing Theorem \ref{firstmaintheorem}.  The idea that the decomposition by endo-parameter could be the~$\overline{\mathbb{Z}}[1/p]$-block decomposition (provided one can establish the relevant finite group theory question) follows by reinterpreting a result on endo-parameters of \cite{MR4000000}, as mentioned in the introduction to \cite{DHKM1}.

Independently to our work, for tame groups, Dat has announced in \cite{DOberwolfach,DatIHES} block decompositions over~$\overline{\mathbb{Z}}[1/p]$ with a Langlands-style parametrization of the factors and a clear pattern for reduction to depth zero (we thank Dat for sharing details of his current work with us in 2019 \cite{DatET}).  As mentioned above, our work was mainly inspired by earlier work of Dat and Vign\'eras on the decomposition by depth, Dat's exhaustion of the category by semisimple characters \cite{Dat09}, together with standard expectations of reduction to depth zero in type theory (for recent examples of reduction to depth zero ideas see \cite{Chinello} and \cite[Section 8]{RKSS}).  The exception to this are parts of Section \ref{secLPs}.  In Section \ref{LLIFsection} we use \cite{DHKM2} to deduce some expected properties of endo-parameters, with part of the point to give an explanation from this point of view of the ``Ramification Theorems'' of Bushnell--Henniart and of \cite{BHS} (together with understanding the conjecture in \cite{KSS}) which show we can also index some of our decompositions by Galois-theoretic data.  When we input any decompositions of moduli spaces on the Galois side from \cite{DHKM1, cotner2024connected} (we do not pursue these in any detail over other coefficient rings than~$\overline{\mathbb{Z}}[1/p]$), it is worth noting that more recent ideas of Dat have been intertwined into these as explained in the introduction of \cite{DHKM1}.  There is now a preprint of Dat and Fintzen \cite{dat2026parametrizationreductiondepthzero}, completing these ideas and proving a reduction to depth zero in the tame setting.

The first version of this article was posted to arXiv in May 2024.  We then split the original article into two at the end of 2025. This part (which gives the main results and applications) and a technical paper \cite{technicalpaper} which (among other things) fills gaps in the literature and proves there exists a finite set of representatives for the cuspidal-conjugacy classes of $m$-semisimple characters of a given endo-parameter.   The main results presented here remain the same, but the technical paper has been improved and we introduce a notion of endo-support in that paper now (not used in this paper). The exposition of this paper has also been improved.

\subsection{Acknowledgements}
The first author was partially supported by EPSRC New Horizons grant EP/V018744/1. The second author was supported by EPSRC grant EP/V001930/1 and the Heilbronn Institute for Mathematical Research. The third author was supported by a Shanghai 2021 `Science and Technology Innovation Action Plan' Natural Science Foundation Project grant.  The fourth author was supported by EPSRC grants EP/H00534X/1 and EP/V061739/1.  We thank Jean-Fran\c cois Dat, Jessica Fintzen, Johannes Girsch, Thomas Lanard, Vanessa Miemietz, Peter Schneider, and Vincent S\'echerre for useful conversations.

\section{Preliminaries}\label{secPreliminaries}

\subsection{Notation}
For a non-archimedean local skew-field~$\D$  we write~$\o_\D$ for the ring of integers of~$\D$,~$\p_\D$ for the unique maximal ideal of~$\o_\D$, and~$\kk_\D$ for the residue field~$\o_\D/\p_\D$.  

\subsection{Smooth~$\R$-representations}\label{abstractsmoothreps}
Let~$\R$ be a~commutative ring (with identity) and~$\H$ a locally profinite group.  We call a smooth representation of~$\H$ on an~$\R$-module an \emph{$\R$-representation}.  The basic theory of~$\R$-representations is developed in Vign\'eras' book \cite{Vig96}.  In particular,~$\R$-representations of~$\H$ form an abelian category, we denote by~$\Rep_{\R}(\H)$, \cite[I 4.2]{Vig96}; and for~$\R$-representations~$\pi_1,\pi_2$ of~$\H$ we write~$\Hom_{\R[\H]}(\pi_1,\pi_2)$ for the~$\R$-module of morphisms~$\pi_1\rightarrow\pi_2$ in this category.  We call an~$\R$-representation \emph{irreducible} if its only~$\R$-subrepresentations are~$0$ and itself; in particular,~$0$ is \emph{not} considered an irreducible~$\R$-representation.  

In this section we collect a few simple lemmas which allow us to approach certain questions on~$\R$-representations over Dedekind domains locally, and which do not appear in the standard source \cite{Vig96}.

\begin{lemma}
Let~$\pi$ be an irreducible~$\R$-representation of~$\H$ on an~$\R$-module~$\mathscr{V}$, and let~
\[\mathcal{A}=\{r\in\R: rv=0\text{ for all }v\in\mathscr{V}\}\] denote the annihilator of~$\mathscr{V}$ as an~$\R$-module.  Then
\begin{enumerate}
\item The annihilator~$\mathcal{A}$ is a prime ideal of~$\R$.
\item The action of~$\R$ factors through~$\R/\mathcal{A}$ and~$\pi$ is torsion-free and irreducible as an~$\R/\mathcal{A}$-representation.
\item Let~$\K(\mathcal{A})$ denote the field of fractions of~$\R/\mathcal{A}$.  Then~$\pi\otimes\K(\mathcal{A})$ is an irreducible~$\K(\mathcal{A})$-representation which is naturally isomorphic to~$\pi$ as an~$\R/\mathcal{A}$-representation via $v\mapsto v\otimes 1$.
\end{enumerate}
\end{lemma}
 
\begin{proof}
The annihilator~$\mathcal{A}$ is clearly an ideal, and as~$\pi$ is irreducible as an~$\R$-representation of~$\H$,~$\mathcal{A}$ is a prime ideal (if~$r_1 r_2\subset\mathcal{A}$, then as~$ r_2\mathscr{V}$ is an~$\R$-representation of~$\H$ and a subrepresentation of~$\pi$ it is zero, in which case~$r_2$ annihilates~$\mathscr{V}$, or~$\mathscr{V}$ in which case~$r_1$ annihilates~$\mathscr{V}$; and as~$\pi$ is non-zero it is not the whole of~$\R$).   

As~$\pi$ is irreducible,~$\alpha\pi= \pi$ for all~$\alpha\in\R\backslash \mathcal{A}$, hence multiplication by~$\alpha$ is invertible on~$\mathscr{V}$ and we write~$\alpha^{-1}$ for the inverse of~$\alpha$ as an~$\R$-module endomorphism.  Hence~$\pi$ is torsion-free as an~$\R/\mathcal{A}$-representation.  Moreover,~$v\mapsto v\otimes 1$ defines an isomorphism~$\pi\mapsto\pi\otimes \K(\mathcal{A})$ as~$v\otimes \frac{\alpha}{\beta}=\alpha \beta^{-1}(v)\otimes 1$; and~$\pi\otimes \K(\mathcal{A})$ is irreducible as a~$\K(\mathcal{A})$-representation as any proper~$\H$-stable subspace defines a proper~$\H$-stable $\R/\mathcal{A}$-subspace of~$\pi$.
\end{proof}

We make use of this simple lemma in the special case of Dedekind domains where we have:

\begin{corollary}\label{NFcases}
Let~$\R$ be a Dedekind domain with field of fractions~$\K$. Let~$\pi$ be an irreducible~$\R$-representation of~$\H$.  Then, either:
 \begin{enumerate}
\item  $\pi$ is torsion-free as an~$\R$-module, the morphism~$\pi\rightarrow \pi\otimes \K$ given by $v\mapsto v\otimes 1$ is an isomorphism of~$\R$-modules, and~$\pi\otimes\K$ is an irreducible~$\K$-representation of~$\H$; or,
\item there exists a unique non-zero maximal ideal~$\mathfrak{l}\in \mathrm{m}\text{-}\mathrm{Spec}(\R)$ which annihilates~$\pi$, and~$\pi$ is an irreducible~$(\R/\mathfrak{l})$-representation of~$\H$.
\end{enumerate}
\end{corollary}

If~$\R$ is a Dedekind domain, then for any~$\H$ as above there are always irreducible~$\R$-representations fitting into the second case of the Corollary (take the trivial~$\R/\mathfrak{l}$-representation for example (or any irreducible~$\R/\mathfrak{l}$-representation)).  However, there might not exist torsion-free irreducible~$\R$-representations (as one easily sees if~$\H$ is a finite group and~$\R$ is not a field).  The discrete group~$\H=\mathbb{Q}^\times$ acting on~$\mathbb{Q}$ by multiplication, gives an example of a torsion-free irreducible~$\mathbb{Z}$-module, but our interest lies in reductive $p$-adic groups where we have:

\begin{lemma}\label{TFdontappear}
Let~$\G$ be a reductive~$p$-adic group.  Let~$\R$ be a Noetherian integral domain with infinitely many prime ideals such that any nonzero prime ideal is maximal, and field of fractions~$\K$.  Let~$\pi$ an irreducible~$\R$-representation of~$\G$, then there exists a non-zero maximal ideal of~$\R$ annihilating~$\pi$ (i.e.,~$\pi$ has torsion).
\end{lemma}

\begin{proof}
Suppose~$\pi$ is an~$\R$-torsion-free irreducible~$\R$-representation of~$\G$.    Let~$\G^0$ denote the~$\F$-points of the connected component of the underlying algebraic group of~$\G$ -- the connected reductive~$p$-adic group.  Then~$\G/\G^0$ is finite, and as such~$\pi\mid_{\G^0}$ is finitely generated and has an irreducible quotient~$\pi^{0}$.  Let~$\mathscr{V}$ be the underlying~$\R$-module of~$\pi^{0}$, and~$\mathscr{V}/\mathscr{W}$ the quotient defining the underlying~$\R$-module of~$\pi$.  If~$\mathscr{V}/\mathscr{W}$ was not torsion free, then there would exist~$r\in \R$ mapping~$\mathscr{V}$ into a proper stable subspace of~$\mathscr{V}$, which is not possible as~$\pi$ is irreducible.  Hence~$\pi^{0}$ is~$\R$-torsion-free, and we are reduced to the setting where~$\G$ is a connected reductive~$p$-adic group.  Moreover, without loss of generality, we can invert~$p$ in our ring~$\R$.

Choose a compact open pro-$p$ subgroup~$\U$ of~$\G$ such that~$\pi$ is a quotient of~$\ind_{\U}^{\G}(1)$.  Then~$\ind_{\U}^{\G}(1)$ is finitely generated projective and we have an equivalence of categories between the category of smooth~$\R$-representations which are quotients of direct sums of copies of~$\ind_{\U}^{\G}(1)$ and the category of~$\R[\U\backslash \G/\U]$-modules.  In particular, the image of~$\pi$ is an~$\R$-torsion-free simple~$\R[\U\backslash \G/\U]$-module which we denote by~$M$.  

We set~$\mathcal{H}=\R[\U\backslash \G/\U]$.  
By \cite{DHKMfiniteness} (see also \cite[Theorem 1.2]{DHKM2}),~$\mathcal{H}$ is a finitely generated module over its centre~$\mathfrak{Z}$ which is a finitely generated~$\R$-algebra.  Any $\R$-endomorphism of $M$ is either zero or an isomorphism, so this says that non-zero elements of~$\R$ act invertibly on $M$.  On the other hand, $M$ is generated by one element as an~$\mathcal{H}$-module, so is finitely generated as a $\mathfrak{Z}$-module.  This means that~$M$ admits a quotient~$M'$ that is simple as a~$\mathfrak{Z}$-module.  The same argument shows that non-zero elements of~$\R$ induce isomorphisms of $M'$ with $M'$. 

The action of~$\mathfrak{Z}$ on $M'$ factors through the quotient of a maximal ideal $\mathfrak{m}$, so we obtain an injective map~$\R\rightarrow \mathfrak{Z}\rightarrow \mathfrak{Z}/\mathfrak{m}$.  So, as~$\R$ is Jacobson (cf.~\cite[Theorem 10]{EmertonJacobson} and \cite[\href{https://stacks.math.columbia.edu/tag/00G4}{Tag 00G4}]{stacks-project}),~$\mathfrak{Z}/\mathfrak{m}$ is an integral extension of~$\R$.  This is absurd, because~$\R$ is not a field. 
\end{proof}

Suppose~$\R$ is a local ring,~$\G$ a reductive $p$-adic group, and suppose the maximal split central torus of~$\G$ has positive rank.  Then one easily constructs~$\R$-torsion-free unramified characters of~$\G$, from which one can build more torsion free representations of other reductive $p$-adic groups.

Via considering certain lattices in representations of~$\G$ over number fields with rings of integers~$\R$, we will also consider a collection of $\R$-torsion-free (reducible) representations of (locally) compact groups, where we will use the following simple lemma:

\begin{lemma}\label{isomatprimesDDgivesiso}
Let~$\R$ be a Dedekind domain with field of fractions~$\K$, and~$\pi_1,\pi_2$ be~$\R$-representations of~$\H$ which are~torsion free as~$\R$-modules, and such that~$\pi_2$ is finitely generated as an~$\R$-representation. Suppose~$\phi:\pi_1\rightarrow \pi_2$ a morphism of~$\R$-representations of~$\H$ such that
\begin{align*}
\phi\otimes 1:\pi_1\otimes\K&\xrightarrow{\sim} \pi_2\otimes\K\\
\phi\otimes 1:\pi_1\otimes (\R/\mathfrak{m})&\xrightarrow{\sim} \pi_2\otimes (\R/\mathfrak{m});
\end{align*}
are isomorphisms, for all~$\mathfrak{m}\in \text{m-}\Spec(\R)$.  Then~$\phi$ is an isomorphism.
\end{lemma}
\begin{proof}
As~$\pi_i$ are torsion free,~$\pi_i\hookrightarrow \pi_i\otimes \K$, and in particular~$\phi$ is injective as~$\phi\otimes 1:\pi_1\otimes \K\rightarrow \pi_2\otimes\K$ is injective.   Let~$C_{\phi}$ denote the cokernel of~$\phi$.  If it is non-zero then it has an irreducible quotient~$\Pi$ as~$\pi_2$ is finitely generated.  By Lemma \ref{NFcases}, there exists~$k\in\{\K,\R/\mathfrak{m}:\mathfrak{m}\in\text{m-}\Spec(\R)\}$, such that~$\Pi\otimes k$ is non-zero, so the cokernel is non-zero after tensoring with~$k$ by right-exactness of tensor.  But right-exactness of tensor again shows that~$C_{\phi}\otimes k=C_{\phi\otimes 1}$ where~$C_{\phi\otimes 1}$ is the cokernel of~$\phi
\otimes 1:\pi_1\otimes k\rightarrow \pi_2\otimes k$, and we know that the latter is zero -- a contradiction. 
\end{proof}

We return to the setting of a general commutative ring~$\R$, where we have:

\begin{lemma}[{\cite[Lemma A.1]{DHKM2}}]\label{Homsandscalarextension}
Suppose that there exists a compact open subgroup of~$\H$ of invertible pro-order in~$\R$. 
Let~$\pi$ be a finitely generated projective~$\R$-representation of~$\H$,~$\R'$ a commutative~$\R$-algebra, and~$\pi'$ an~$\R$-representation of~$\H$.  Then the natural map
\[\Hom_{\R[\H]}(\pi,\pi')\otimes \R'\rightarrow \Hom_{\R'[\H]}(\pi\otimes \R',\pi'\otimes \R')\]
$f\otimes r'\mapsto r'(f\otimes 1)$ defines an isomorphism of~$\R'$-modules.
\end{lemma}

Finally, to construct representations of $p$-adic groups via compact induction from explicitly constructed representations of compact open subgroups; inverting $p$ in our coefficient rings we find that these are finitely generated projective utilizing the following simple lemma:

\begin{lemma}\label{lemmaprojectivesfinitegroups}
Let~$\R$ be a (commutative)~$\mathbb{Z}[1/p]$-algebra.  Let~$\rho$ be an irreducible (smooth) representation of a pro-$p$ compact open subgroup~$\H$ of~$\G$ on a projective~$\R$-module, or if~$\R$ is an integral domain, a (smooth) representation of~$\H$ on a projective~$\R$-module of finite rank as an~$\R$-module.  Then~$\rho$ is a projective~$\R[\H]$-module, and~$\ind_{\H}^{\G}(\rho)$ a finitely generated projective representation of~$\G$.
\end{lemma}
\begin{proof}
As compact induction from an open subgroup is left adjoint to an exact functor (restriction), and preserves finite type, the final statement follows from the previous.

Let~$\K$ be a compact open subgroup of~$\H$ on which~$\rho$ is trivial.  Then $\Hom_{\R[\H]}(\ind_{\K}^{\H}(1_{\R}),\rho)\neq 0$, where~$1_{\R}$ denotes the trivial character of~$\K$ on the free $\R$-module~$\R$.   Choose any non-zero morphism~$\phi$, then as~$\rho$ is a projective~$\R$-module, we can split~$\phi$ as an~$\R$-module homomorphism via a morphism~$s$, and then as~$p$ is invertible in~$\R$ we can average this splitting over~$\H/\K$ to obtain a splitting as~$\R[\H/\K]$-modules:
\[s_{\H}=[\H:\K]^{-1}\sum_{h\in\H/\K}R(h)\circ s\circ\rho(h^{-1}),\]
where~$R(h)$ denotes the right regular action on~$\ind_{\K}^{\H}(1_{\R})$.  Hence~$\ind_{\K}^{\H}(1_{\R})=\rho\oplus\rho'$.  Moreover, as~$\Hom_{\R[\H]}(\ind_{\K}^{\H}(1_{\R}),-)\simeq\Hom_{\R[\H]}(\rho,-)\times\Hom_{\R[\H]}(\rho',-)$ as functors we can reduce to showing~$\ind_{\K}^{\H}(1_{\R})$ is projective, and as compact induction preserves projectivity to~$1_{\R}$ is a projective representation of~$\K$.  Projectivity of the trivial representation of~$\K$ under our hypotheses is \cite[I 4.9 (a)]{Vig96}
\end{proof}

\subsection{Primitive idempotents of the centre and blocks}
Let~$\mathfrak{Z}_{\R}(\G)$ denote the centre of~$\Rep_{\R}(\G)$, it is a commutative unital ring which acts on every smooth~$\R$-representation.  Suppose~$\phi:\R\rightarrow \R'$ is a morphism of commutative~rings, then we have a morphism
\begin{equation}
\label{centresmap}\phi^*:\mathfrak{Z}_{\R}(\G)\rightarrow \mathfrak{Z}_{\R'}(\G)\end{equation}
which is induced from the functor~$\Rep_{\R'}(\G)\rightarrow \Rep_{\R}(\G)$ sending a smooth~$\R'$-representation to the smooth~$\R$-representation with the same action of~$\G$ and the action of~$\R$ from the morphism~$\R\rightarrow\R'$, cf.~\cite[Section 2]{DHKM2}.  

An idempotent~$e\in\mathfrak{Z}_{\R}(\G)$ induces a canonical decomposition of every smooth representation~$\pi=e\pi\oplus (1-e)\pi$,
and a corresponding canonical decomposition of morphisms in the category; in other words defines a decomposition~$\Rep_{\R}(\G)=e\Rep_{\R}(\G)\times (1-e)\Rep_{\R}(\G)$ (and conversely, such a decomposition gives rise to a central idempotent).  Note that, given a morphism of rings~$\R\rightarrow \R'$, via (\ref{centresmap}), a decomposition of~$\Rep_{\R}(\G)$ induces a decomposition of~$\Rep_{\R'}(\G)$.

\begin{lemma}[{\cite[Lemma A.3]{Dat09} and \cite[Lemma 3.2 and Proposition 3.4]{DHKM2}}]~\label{basiccentrelemma}
\begin{enumerate}
\item Suppose~$\phi$ is a monomorphism, then~$\phi^*:\mathfrak{Z}_{\R}(\G)\rightarrow \mathfrak{Z}_{\R'}(\G)$ is a monomorphism.
\item We have a decomposition
\[\mathfrak{Z}_{\R}(\G)=\prod_{r\in D(\G)}\mathfrak{Z}_{\R}(\G)_r\]
where
\begin{itemize}[-]
\item $D(\G)\subseteq \mathbb{Q}^+$ denotes the subset of rational numbers for which there is a smooth~$\R$-representation of~$\G$ of the given depth; and
\item $\mathfrak{Z}_{\R}(\G)_r=e_r\mathfrak{Z}_{\R}(\G)$ denotes the centre of the category of smooth~$\R$-representations of depth~$r$, corresponding to a depth~$r$ idempotent~$e_r\in\mathfrak{Z}_\R(\G)$. 
\end{itemize}  Moreover, for each~$r\in D(\G)$, we have an (explicit) finitely generated projective generator~$P_r$ of~$\Rep_{\R}(\G)_r=e_r\Rep_{\R}(\G)$.
\item \label{tensorcentre} Suppose~$\R$ is Noetherian and~$\R'$ a flat (commutative)~$\R$-algebra.  Then the natural map is an isomorphism
\[\mathfrak{Z}_{\R}(\G)_n\otimes\R'\xrightarrow{\sim}\mathfrak{Z}_{\R'}(\G)_n.\]
\end{enumerate}
\end{lemma}
For~$r\in \mathbb{Q}^+$, we write~$\Rep_{\R}(\G)_{\leqslant x}=\prod_{\substack{r\in D(\G)\\ r\leqslant x}}\Rep_{\R}(\G)_r$.

\begin{definition}
\begin{enumerate}
\item Suppose~$e\in \mathfrak{Z}_{\R}(\G)$ is a primitive idempotent.   Then the subcategory~$e\Rep_{\R}(\G)$ is called an~\emph{$\R$-block} of~$\Rep_{\R}(\G)$.  
\item Suppose~$1\in\mathfrak{Z}_{\R}(\G)$ decomposes~$1=\sum_{i\in\I} e_i$ into a sum of primitive idempotents~$e_i\in \mathfrak{Z}_{\R}(\G)$, then the corresponding decomposition
\[\Rep_{\R}(\G)=\prod_{i\in\I} e_i\Rep_{\R}(\G),\]
is called the~\emph{$\R$-block decomposition}.
\end{enumerate}
\end{definition}

Note that, for a given~$\R$, a priori the block decomposition of~$\Rep_{\R}(\G)$ may not exist.  However, for a Noetherian ring, the depth~$n$ centre is Noetherian:

\begin{theorem}[{\cite{DHKM2,DHKMfiniteness}}]
Let~$\R$ be a Noetherian~$\mathbb{Z}[1/p]$-algebra, and~$r\in D(\G)$.  Then~$\mathfrak{Z}_{\R}(\G)_r$ is a finitely generated~$\R$-algebra.
\end{theorem}

And it follows that, in these cases, the~$\R$-block decomposition exists:

\begin{corollary}\label{finitelymanyblocksofdepth}
Let~$\R$ be a Noetherian~$\mathbb{Z}[1/p]$-algebra.  \begin{enumerate}
\item Then the~$\R$-block decomposition exists.  
\item Moreover,~for~$r\in D(\G)$,~$e_r\Rep_{\R}(\G)$ is a finite product of~$\R$-blocks.\end{enumerate}
\end{corollary}

\subsection{Generic irreducibility and parabolic inductions over algebraically closed fields}
Suppose~$\K$ is an algebraically closed field of characteristic~$\ell\neq p$.  For a reductive~$p$-adic group~$\M=\mathbb{M}(\F)$, we let
\[\M^{\circ}=\{g\in \M:\chi(g)\in \mathfrak{o}_\F^\times\text{ for all $\F$-rational characters of~$\mathbb{M}$}\},\]
an open normal subgroup of~$\M$. 
Note that,~$\M/\M^\circ$ is a free abelian group of rank equal to the rank of a maximal split torus in the centre of~$\M$.  Recall a~$\K$-character of~$\M$ is called \emph{unramified} if it is trivial on~$\M^{\circ}$.  A famous result of Bernstein for~$\mathbb{C}$-representations, we have the following generic irreducibility theorem of parabolic inductions:

\begin{theorem}[Generic irreducibility, {\cite[Corollary 1.5]{DHKMfiniteness}}]\label{genirred}
Let~$P$ be a parabolic subgroup of~$\G$ with Levi factor~$\M$, and~$\sigma$ be an irreducible~$\K$-representation of~$\M$.  Then the parabolically induced representation~$i_{\M,\P}^{\G}(\sigma\chi)$ is irreducible for~$\chi$ in a Zariski-dense open subset of the~$\K$-torus of unramified characters of~$\M$.
\end{theorem}

For~$\sigma$ an irreducible~$\K$-representation of~$\M$, we consider~$\ind_{\M^{\circ}}^{\M}(\sigma)\simeq \sigma\otimes\ind_{\M^{\circ}}^{\M}(1)\simeq \sigma\otimes\K[\M/\M^{\circ}]$ as a~$\K[\M/\M^{\circ}]$-representation of~$\M$ and as a family of~$\K$-representations of~$\M$: A closed~$\K$-point of~$\K[\M/\M^{\circ}]$ corresponds to a~$\K$-algebra morphism~$\phi:\K[\M/\M^{\circ}]\rightarrow \K$, and hence an unramified character~$\chi:\M\rightarrow \K^\times$, and we have 
\[\ind_{\M^{\circ}}^{\M}(\sigma)\otimes_{\K[\M/\M^{\circ}],\phi}\K\simeq \sigma\otimes\chi,\]
and can consider~$\ind_{\M^{\circ}}^{\M}(\sigma)$ as a family of irreducible~$\K$-representations of~$\M$. Then~$i_{\M,\P}^{\G}(\sigma\otimes\K[\M/\M^{\circ}])$ is a~$\K[\M/\M^{\circ}]$-representation, and~\[i_{\M,\P}^{\G}(\sigma\otimes\K[\M/\M^{\circ}])\otimes_{\K[\M/\M^{\circ}],\phi}\K\simeq i_{\M,\P}^{\G}(\sigma\otimes\chi)\] 
which is irreducible for a Zariski-dense open subset~$\K$-points of~$\K[\M/\M^{\circ}]$ by generic irreducibility. 
\begin{proposition}\label{genirredproposition}
An element of~$\mathfrak{Z}_{\K}(\G)$ acts on~$i_{\M,\P}^{\G}(\sigma\otimes\K[\M/\M^{\circ}])$ by multiplication by an element of~$\K[\M/\M^{\circ}]$.
\end{proposition}
\begin{proof}
The proof of \cite[1.17]{MR771671} also applies here, to show that \[\End_{\K[\M/\M^{\circ}][\G]}(i_{\M,\P}^{\G}(\sigma\otimes\K[\M/\M^{\circ}]))\simeq \K[\M/\M^{\circ}]\] as a consequence of generic irreducibility and Schur's lemma (note that, when~$\ell\neq 0$, an unramified twist of~$\sigma$ is defined over~$\Fl$).  Now~$\mathfrak{Z}_{\K}(\G)$ acts on~$i_{\M,\P}^{\G}(\sigma\otimes\K[\M/\M^{\circ}])$ via a central endomorphism in $\End_{\K[\G]}(i_{\M,\P}^{\G}(\sigma\otimes\K[\M/\M^{\circ}]))$, and as this endomorphism is central it commutes with multiplication by elements of~$\K[\M/\M^{\circ}]$ (as they define~$\K[\G]$-endomorphisms) and hence defines an element of~$\End_{\K[\M/\M^{\circ}][\G]}(i_{\M,\P}^{\G}(\sigma\otimes\K[\M/\M^{\circ}]))\simeq \K[\M/\M^{\circ}]$.
\end{proof}

\subsection{The Bernstein centre and finitely generated projectives over an algebraically closed field of characteristic zero}\label{projmoduleschar0}
Suppose for this section that~$\K$ is an algebraically closed field of characteristic zero.  Then the~$\K$-block decomposition and the centre of the category of smooth~$\K$-representations has a relatively simple description, due to Bernstein:  

\begin{theorem}[{Bernstein \cite{MR771671}}]
\begin{enumerate}
\item The~$\K$-block decomposition is given by
 \[\Rep_{\K}(\G)=\prod_{\mathfrak{s}\in\mathfrak{B}_{\K}(\G)}\Rep_{\K}({\mathfrak{s}}),\]
where~$\mathfrak{B}_{\K}(\G)$ denotes the set of inertial classes of supercuspidal supports for~$\G$, and (the representations in)~$\Rep_{\K}({\mathfrak{s}})$ consist of all (smooth)~$\K$-representations all of whose irreducible subquotients have supercuspidal support in~$\mathfrak{s}$.
\item Let~$\mathfrak{s}\in\mathfrak{B}_{\K}(\G)$.  For any choice of representative~$(\M,\sigma)$ of~$\mathfrak{s}$, and any parabolic~$\P$ of~$\G$  with Levi factor~$\M$ the representation~$P_{\P,\sigma}=i_{\M,\P}^{\G}(\ind_{\M^{\circ}}^{\M}(\sigma))$ is a finitely generated projective generator for~$\Rep_{\K}(\G)_{\mathfrak{s}}$.
\item Let~$\mathfrak{s}\in\mathfrak{B}_{\K}(\G)$.  Fixing a finitely generated projective generator~$P_{\P,\sigma}=i_{\M,\P}^{\G}(\ind_{\M^{\circ}}^{\M}(\sigma))$ as above for~$\Rep_{\K}(\mathfrak{s})$, identifies the centre~$\mathfrak{Z}_{\K}(\mathfrak{s})$ of~$\Rep_{\K}({\mathfrak{s}})$ with
\[\K[\M/\M^\circ]^{\H_{\mathfrak{s}}\rtimes \W_{\mathfrak{s}}},\]
where~$\H_{\mathfrak{s}}$ is the finite group of unramified characters of~$\M$ fixing~$\sigma$ up to isomorphism, and~$\W_{\mathfrak{s}}=\{w\in\W:\M^w=\M,~(\M,\sigma^w)\in[\M,\sigma]_{\M}\}$.
\end{enumerate}
\end{theorem}
 In particular, Bernstein's description shows that the centre~$\mathfrak{Z}_{\K}(\mathfrak{s})$ is a (commutative) Noetherian domain.    
 Let~$\mathfrak{s}\in\mathfrak{B}_{\K}(\G)$.  We write~$\eta=\eta_{\mathfrak{s}}$ for the unique generic point of~$\Spec(\mathfrak{Z}_{\K}(\mathfrak{s}))$, and given a representation~$\pi$ in~$\Rep_{\K}(\mathfrak{s})$, we write~$\pi_{\eta}$ for its localization at~$\eta$,~i.e.,~$\pi_{\eta}=\pi\otimes \K(\eta)$, where~$\K(\eta)$ is the field of fractions of~$\mathfrak{Z}_{\K}(\mathfrak{s})$.   To the inertial class~$\mathfrak{s}=[\M,\sigma]$, we can associate the class~$\mathfrak{s}_\M=[\M,\sigma]\in\mathfrak{B}_{\K}(\M)$.  We write~$\widetilde{\eta}$ for the unique generic point of~$\Spec(\mathfrak{Z}_{\K}(\mathfrak{s}_\M))$, and given a representation~$\pi$ in~$\Rep_{\K}(\mathfrak{s}_\M)$, we write~$\pi_{\widetilde{\eta}}$ for its localization at~$\widetilde{\eta}$, i.e.,~$\pi_{\widetilde{\eta}}=\pi\otimes \K(\widetilde{\eta})$, where~$\K(\widetilde{\eta})$ is the field of fractions of~$\mathfrak{Z}_{\K}(\mathfrak{s}_\M)$.

\begin{lemma}\label{mainlemmaKprojs}
Let~$\mathfrak{s}\in\mathfrak{B}_{\K}(\G)$, and~$P$ be a (non-zero) finitely generated projective representation in~$\Rep_{\K}(\mathfrak{s})$.
\begin{enumerate}
\item The modules~$P$ and~$\End_{\K[\G]}(P)$ are~$\mathfrak{Z}_{\K}(\mathfrak{s})$-torsion free.
\item We have~$P_{\eta}\otimes_{\K(\eta)}\K(\widetilde{\eta})\simeq \pi^{\oplus r}$ for some irreducible~$\K(\widetilde{\eta})[\G]$-representation~$\pi$.
\item\label{Nocentralidemps} The endomorphism algebra~$\End_{\K[\G]}(P)$ has no non-trivial central idempotents.
\item \label{PP'hom} If~$P'$ is a (non-zero) finitely generated projective representation in~$\Rep_{\K}(\mathfrak{s})$, then\[\Hom_{\K[\G]}(P,P')\neq 0.\]
\end{enumerate}
\end{lemma}

\begin{proof}
\begin{enumerate}
\item The representation~$P_{\P,\sigma}=i_{\M,\P}^{\G}(\ind_{\M^{\circ}}^{\M}(\sigma))$ is~$\mathfrak{Z}_{\K}(\mathfrak{s})$-torsion free, because~$\ind_{\M^{\circ}}^{\M}(\sigma)$ is a torsion free~$\mathfrak{Z}_{\K}(\mathfrak{s}_\M)$-module and parabolic induction is exact.  Hence~$P$ is~$\mathfrak{Z}_{\K}(\mathfrak{s})$-torsion free, as it is a summand of a direct sum of copies of~$P_{\P,\sigma}$.
\item As localization is exact~$P_{\eta}\hookrightarrow ({P_{\P,\sigma}}^{\oplus m})_{\eta}$, and it suffices to prove the statement for~$P_{\P,\sigma}$.  We have
\begin{align*}
(P_{\P,\sigma})_{\eta}\otimes_{\K(\eta)}\K(\widetilde{\eta})&\simeq i_{\M,\P}^{\G}(\ind_{\M^{\circ}}^{\M}(\sigma))\otimes_{\mathfrak{Z}_{\K}(\mathfrak{s})}\K(\widetilde{\eta})\\
&\simeq i_{\M,\P}^{\G}(\ind_{\M^{\circ}}^{\M}(\sigma)_{\eta}\otimes_{\K(\eta)}\K(\widetilde{\eta}))\\
&\simeq i_{\M,\P}^{\G}(\ind_{\M^{\circ}}^{\M}(\sigma)_{\widetilde{\eta}})\otimes_{\K(\eta)}\K(\widetilde{\eta})
\end{align*}
as~$\mathfrak{Z}_{\K}(\mathfrak{s}_\M)\otimes_{\mathfrak{Z}_{\K}(\mathfrak{s})}\K(\eta)=\K(\widetilde{\eta})$.  Now~$ i_{\M,\P}^{\G}(\ind_{\M^{\circ}}^{\M}(\sigma)_{\widetilde{\eta}})$ is absolutely irreducible as a~$\K(\widetilde{\eta})[\G]$-module by generic irreducibility.  Moreover,
\[i_{\M,\P}^{\G}(\ind_{\M^{\circ}}^{\M}(\sigma)_{\widetilde{\eta}})\otimes_{\K(\eta)}\K(\widetilde{\eta})\simeq i_{\M,\P}^{\G}(\ind_{\M^{\circ}}^{\M}(\sigma)_{\widetilde{\eta}})\otimes_{\K(\widetilde{\eta})}(\K(\widetilde{\eta})\otimes_{\K(\eta)}\K(\widetilde{\eta})).\]
And as $\K(\widetilde{\eta})/ \K(\eta)$ is a finite Galois extension of degree~$d=|\W_{\mathfrak{s}}|$, we have~$\K(\widetilde{\eta})\otimes_{\K(\eta)}\K(\widetilde{\eta})\simeq\K(\widetilde{\eta})^{d}$ and the result follows.
\item By torsion freeness,~$\End_{\K[\G]}(P)$ injects into its localization at the generic point of $\mathrm{Spec}(\mathfrak{Z}_{\K}(\mathfrak{s}))$, which is a central simple algebra by the last part, and thus has no nontrivial central idempotents.
\item Take~$m$ such that~$P$ and~$P'$ are summands of~$(P_{\P,\sigma})^{\oplus m}$, and consider these modules as modules over the Hecke algebra~$\End_{\K[\G]}((P_{\P,\sigma})^{\oplus m})$.  After localizing at the generic point of~$\mathfrak{Z}_{\K}(\mathfrak{s})$ this latter algebra is a central simple algebra, and hence has a single indecomposable module; $P,P'$, (and~$P_{\P,\sigma}$) are simply direct sums of copies of this indecomposable. The result is thus immediate.
\end{enumerate}
\end{proof}

\subsection{A lemma of Vign\'eras}
Let~$\H$ be an arbitrary locally compact totally disconnected group.  We will use the following lemma of \cite{Vig96}:

\begin{lemma}[{\cite[I 8.1]{Vig96}}]\label{Lemma1}
Suppose~$\R$ is an algebraically closed field. Let $\K_i,\K'$ be compact mod-centre open subgroups of $\H$ for $i$ in some finite set $\I$, let $\pi$ be a smooth $\R$-representation of $\H$.  Suppose that $(\tau_i,\cW_i)$ are submodules of $\pi\vert_{\K_i}$, that $\tau'$ is an irreducible subquotient of $\pi\vert_{\K'}$, and that $\pi$ is generated by $\bigoplus_{i\in I}\cW_i$.  Then there exist $i\in \I$, $h\in \H$ such that~$\tau'$ is a subquotient of $\Ind_{\K'\cap \presuper{h}\K_i}^{\K'}\presuper{h}\tau_i$.
\end{lemma}
Note that the proof of ibid.~is written for a single $\K_i$, but the proof works just as well for an arbitrary collection of $\K_i$.  In particular, if the pro-order of the subgroup $\K'$ is invertible in $\R$, then its smooth representations are semisimple and we deduce:
\begin{corollary}\label{CorollarytoVigneras}
Suppose~$\R$ is an algebraically closed field.  Under the hypotheses of Lemma \ref{Lemma1}, suppose that the pro-order of $\K'$ is invertible in $\R$.  Then, for the same $i\in \I$, we have $\I_\H(\tau',\tau_i) \neq 0$.
\end{corollary}
In particular, we will use Corollary \ref{CorollarytoVigneras} when~$\K'$ is a pro-$p$ group, which is another reason why we invert~$p$ in our coefficient rings.

\subsection{Parahoric subgroups in reductive $p$-adic groups}\label{parahorics1}
Let~$\mathcal{F}$ be a facet of the extended Bruhat-Tits building of~$\G$, and~$\G_\mathcal{F}^+$ denote the compact open subgroup of~$\G$ which is the pointwise fixator of $\mathcal{F}$.  It has pro-$p$ unipotent radical~$\G_\mathcal{F}^1$ with 
\[1\rightarrow \G_\mathcal{F}^1\rightarrow \G_\mathcal{F}^+\rightarrow \Gf_\mathcal{F}^+\rightarrow 1\]
with~$\Gf_\mathcal{F}^+$ the~$\mathbb{F}_q$-points of an algebraic group~$\mathcal{G}^+_{\mathcal{F}}$ defined over~$\mathbb{F}_q$.  Letting~$\Gf_{\mathcal{F}}$ denote the~$\mathbb{F}_q$-points of the connected component of the identity of~$\mathcal{G}^+_{\mathcal{F}}$, the \emph{parahoric subgroup}~$\G_{\mathcal{F}}$ associated to~$\mathcal{F}$ is the preimage of~$\Gf_{\mathcal{F}}$ in~$\G_{\mathcal{F}}$.  The quotient~$\G_{\mathcal{F}}^+/\G_{\mathcal{F}}\simeq \Gf_{\mathcal{F}}^+/\Gf_{\mathcal{F}}$ is a finite abelian group.   For~$x$ a point in the extended Bruhat--Tits building of~$\G$, we use analogous notation: $\G_x^+$ denotes the pointwise fixator of~$x$, $\G_x$ denotes the parahoric subgroup associated to~$x$,~$\G_x^1$ denotes the pro-$p$ unipotent radical of~$\G_x^+$, $\Gf_x^+=\G_x^+/\G_x^1$ and~$\Gf_x=\G_x/\G_x^1$.  

Note that, in some of our references the reduced Bruhat-Tits building of~$\G$ is favoured instead of the extended building.  In this case, following Bruhat--Tits as in \cite{KP},~$\G^{\circ}$ is denoted~$\mathbb{G}(\F)^1$, and if for~$x$ in the extended building we let~$\overline{x}$ denote its image in the reduced building, then we have
\[\G_x^+=\mathbb{G}(\F)_{\overline{x}}^1,\quad \G_x=\mathbb{G}(\F)_{\overline{x}}^0.\]
In particular, we deduce from \cite[Lemma 2.2.16, Theorem 4.2.17]{KP} that if~$x$ is a vertex (by vertex, we mean a point in the extended building whose image in the reduced building is a vertex), then~$\N_{\G^{\circ}}(\G_x^+)=\G_x^+$.

The parahoric subgroups containined in~$\G_{\mathcal{F}}$ are in bijection with the parabolic subgroups of~$\Gf_{\mathcal{F}}$ (cf.~\cite[Proposition 5.1.32]{BT2}) : we write~$\G_{\mathcal{F},\Q}$ for the parahoric subgroup which is the inverse image of the parabolic subgroup~$\Q$ of~$\Gf_{\mathcal{F}}$ under the map~$\G_{\mathcal{F}}\rightarrow\Gf_{\mathcal{F}}$.  There is thus a facet~$\mathcal{F}'$ such that~$\G_{\mathcal{F}'}=\G_{\mathcal{F},\Q}$ and~$\Q$ has a Levi decomposition~$\Q=\Gf_{\mathcal{F}'}\ltimes \U$.  

The image~$\Q^+$ of~$\G_{\mathcal{F}'}^+$ in~$\Gf_{\mathcal{F}}^+$ (contains and) normalizes~$\U$ which is the image of~$\G_{\mathcal{F}'}^1$, and~$\Q^+$ has a ``Levi decomposition''~$\Q^+=\Gf_{\mathcal{F}'}^+\ltimes \U$.  We will need a supercuspidal support map on these potentially disconnected reductive finite groups:

\begin{lemma}[Existence of ``supercuspidal support'']\label{fgLemmascsupport}
Let~$\pi^+$ be an irreducible representation of~$\Gf_{\mathcal{F}}^+$ over a sufficiently large field.  Then there exists a facet~$\mathcal{F}'$ corresponding to a parabolic subgroup~$\Q$ of~$\Gf_{\mathcal{F}}$, and an irreducible representation~$\tau^+$ of~$\Gf_{\mathcal{F}'}^+$ such that, 
\begin{enumerate}
\item $\tau^+$ has supercuspidal restriction to~$\Gf_{\mathcal{F}'}$ (i.e.,~the projective cover of~$\tau^+$ is cuspidal);
\item $\pi^+$ is a subquotient of~$\ind_{\Q^+}^{\Gf_{\mathcal{F}}^+}(\tau^+)$.
\end{enumerate} 
\end{lemma}
\begin{proof}
Let~$\pi$ be an irreducible subrepresentation of~$\pi^+\mid_{\Gf_{\mathcal{F}}}$.  Let~$(\L,\tau)$ be in the supercuspidal support of~$\pi$, and~$\Q$ denote a parabolic subgroup of~$\Gf_{\mathcal{F}}$ with Levi factor~$\L$ and unipotent radical~$\U$.  Let~$\G_\mathcal{F}'$ be the parahoric corresponding to~$\Q$.

The representation $\pi^+$ is a (sub)quotient of~$\ind_{\Gf_{\mathcal{F}}}^{\Gf_{\mathcal{F}}^+}(\pi)$, hence a subquotient of
\[\ind_{\Gf_{\mathcal{F}}}^{\Gf_{\mathcal{F}}^+} \circ \ind_{\Q}^{\Gf_{\mathcal{F}}}(\tau)\simeq \ind_{\Q^+}^{\Gf_{\mathcal{F}}^+}\circ \ind_{\Q}^{\Q^+}(\tau). \]
Now~$\Q=\Gf_{\mathcal{F}'}\ltimes \U$ and~$\Q^+=\Gf_{\mathcal{F}'}^+\ltimes \U$, and~$\tau$ is trivial on $\U$, hence~$\pi^+$ is a subquotient of
\[\ind_{\Q^+}^{\Gf_{\mathcal{F}}^+}\circ\ind_{\Gf_{\mathcal{F}'}}^{\Gf_{\mathcal{F}'}^+} \tau,\]
thus there exists an irreducible subquotient~$\tau^+$ of~$\ind_{\Gf_{\mathcal{F}'}}^{\Gf_{\mathcal{F}'}^+} \tau$ such that~$\pi^+$ is a subquotient of~$\ind_{\Q^+}^{\Gf_{\mathcal{F}}^+}(\tau^+)$.  As~$\tau^+$ is a subquotient of~$\ind_{\Gf_{\mathcal{F}'}}^{\Gf_{\mathcal{F}'}^+} \tau$, by Mackey theory~$\tau^+$ restricts to a sum of conjugates of~$\tau$ and hence~$\tau^+$ has supercuspidal restriction.  
\end{proof}

\subsection{Forms of classical groups}\label{subsecClassicalGroups}
We introduce a uniform notation to consider inner forms of general linear and classical groups, we follow the notation of \cite[Section 2.3]{technicalpaper}.  We fix~$\F/\F_\so$, a Galois extension of non-archimedean local fields of degree one or two with residual characteristic~$p$, and~$\varepsilon\in\{\pm,0\}$, such that
\newcommand{\caseGL}{\textbf{(\I)}}
\newcommand{\caseClassical}{\textbf{(\I\I)}}
\begin{itemize}
 \item[\caseGL]~$\F=\F_\so$ if~$\varepsilon=0$ and
 \item[\caseClassical]~$2\not\mid p$ if~$\varepsilon\neq 0$.
\end{itemize}
We write~$\overline{\phantom{ll}}$ for the generator of the of the Galois group of~$\F/\F_\so$. 

We denote by~$\Div(\F/\F_\so,\varepsilon)$ the set of pairs~$(\D,(\overline{\phantom{ll}})_\D)$ consisting of a skew-field~$\D$ of finite degree~$d$ with center~$\F$ and an~$\F_\so$-linear endomorphism~$(\overline{\phantom{ll}})_\D$ of~$\D$ extending~$\overline{\phantom{ll}}$ such that
\begin{itemize}
 \item $(\overline{\phantom{ll}})_\D$ is~$\id_\D$ in Case~\caseGL~and
 \item $(\overline{\phantom{ll}})_\D$ is an orthogonal or unitary anti-involution on~$\D$ in Case~\caseClassical. Note that in this case~$\D$ has at most degree~$2$, and if~$\D$ has degree~$2$ then~$\F=\F_\so$. 
\end{itemize}
We will still write~$\overline{\phantom{ll}}$ for~$(\overline{\phantom{ll}})_\D$ if there is no cause of confusion. 

A~\emph{Hermitian space for~$(\F/\F_\so,\varepsilon)$} is a pair~$(\V,h)$ together with a pair~$(\D,(\overline{\phantom{ll}})_\D)\in\Div(\F/\F_\so,\varepsilon)$ such that~$\V$ is a finite dimensional right~$\D$-vector space and~$h$ is an~$\varepsilon$-hermitian form~$h:\V\times\V\rightarrow \D$ with respect to~$\overline{\phantom{ll}}$, in particular~$h$ is the zero map in Case~\caseGL~and non-degenerate in Case~\caseClassical.
We write~$\Herm(\F/\F_\so,\varepsilon)$ for the set of Hermitian spaces for~$(\F/\F_\so,\varepsilon)$. 

We fix a pair~$(\D,(\overline{\phantom{ll}})_\D)\in\Div(\F/\F_\so,\varepsilon)$ and a Hermitian space~$(\V,h)$ for~$(\F/\F_\so,\varepsilon)$. 

We consider the following subgroups of~$\tG:=\GL_\D(\V)$:
The group
\[\U(\V,h)=\{g\in\GL_\D(\V):h(gv,gw)=h(v,w)\text{ for all }v,w\in\V\}\]
of isometries of~$(\V,h)$.  
and
\[\G:=\left\{\begin{array}{ll}
\U(\V,h)\cap\SL_\F(\V) & \text{if}\ h\neq 0\ \text{and}\ \D=\F=\F_\so\ \text{and}\ \varepsilon=+\\
\U(\V,h)& \text{else}
           \end{array}\right.
\]
 
Therefore~$\G$ will be the set of rational points of an~$\F_\so$-form of a general linear, symplectic, or special orthogonal group. 
Note that in the Case~\caseClassical~if~$\D\neq\F$ then every element of~$\U(\V,h)$ has reduced norm~$1$ over~$\F$. 
 
We let~$\Sigma=\langle \sigma\rangle$ denote an abstract cyclic group of order~$2$, which we will let act on various objects. The element~$\sigma$  acts trivial on~$\tG$ if~$h=0$ and as the inverse of the adjoint anti-involution of~$h$ if~$h$ is non-zero.  In particular,
\[
\U(\V,h)=\tG^\Sigma.
\]
We denote the set of~$\D$-endomorphisms of~$\V$ by~$\A$. The square root of the~$\F$-dimension of~$\A$ is called the~\emph{$\F$-degree} 
(or for short just the degree) of~$\A$, denoted by~$\deg_\F(\A)$.


\section{Endo-parameters}\label{secEPs}
We give a short definition of an endo-parameter for~$\G$.  For more details, we refer to \cite{KSS}, \cite{SkodInnerFormI}, and \cite{technicalpaper}.  
\subsection{Semisimple strata and characters}
Let~$[\Lambda,n,0,\beta]$  be a tuple consisting of
\begin{enumerate}
\item An element $\beta\in \mathrm{Lie}(\G)$ such that~$\E=\F[\beta]=\bigoplus_{i\in \I} \E_i$ is a sum of field extensions~$\E_i/\F$, and we write~$\G_\beta$ for the~$\G$-centralizer of~$\beta$.  The group~$\G_{\beta}$ is the~$\F$-points of a product of restriction of scalars to~$\F$ of inner forms of classical groups and general linear groups, and we write~$\mathbb{G}_\beta$ for this reductive group defined over~$\F$.
\item A self-dual~$\mathfrak{o}_\E$-lattice sequence~$\Lambda$ in~$\V$, which defines a point\[\Lambda\in\mathcal{B}(\mathbb{G}_{\beta},\F)\hookrightarrow \mathcal{B}(\mathbb{G},\F).\]
\item A non-negative integer~$n$ such that~$\beta\in\mathfrak{A}_{-n}(\Lambda)$ where
\[\mathfrak{A}_{i}(\Lambda)=\{a\in\End_{\F}(\V):a\Lambda(k)\subseteq \Lambda(k+i),\text{ for all }k\in\mathbb{Z}\}\] is the filtration of~$\End_{\F}(\V)$ defined by~$\Lambda$.
\end{enumerate}
If this tuple satisfies some technical conditions (cf., \cite[3.1]{St08}, \cite[4.1]{SkodInnerFormI}, \cite[4.1]{SkodInnerFormII}) then it is called a \emph{semisimple stratum for}~$(\G,h)$.  As the integer~$n$ is determined by~$(\Lambda,\beta)$, in this paper we will denote the stratum~$[\Lambda,n,0,\beta]$ by~$[\Lambda,\beta]$ (as we will not need the more general strata~$[\Lambda,n,r,\beta]$ of the above references).  

Write~$\Psi_i(X)\in \F[X]$ denote the minimal polynomial of~$\beta_i$, then we set~$\V^i=\ker(\Psi_i(\beta))$; and we have a decomposition~$\V=\bigoplus_{i\in\I} \V^i$, called the \emph{splitting associated} to~$\beta$.

The negative of the anti-involution defined by $h$ on~$\End_{\F}(\V)$ defines an action of~$\Sigma$ on~$\E$ with~$\sigma(\beta)=-\beta$, and hence defines an action on the indexing set~$\I$ of the decomposition~$\beta=\sum_{i\in\I}\beta_i$.  We write~$\I^\sigma$ for the~$\sigma$-fixed orbits and~$\I_{\sigma}=\I\backslash \I^\sigma$ for its complement.  This defines a Levi subgroup~$\M(\beta)$ of~$\G$ by
\begin{equation*}
\M(\beta)=\begin{cases}
\prod_{i\in\I}\Aut_{\D}(\V_i),&\text{if $h=0$;}\\
\G\cap( (\Aut_{\D}( \sum_{i\in\I^\sigma}(\V_i))\times \prod_{i\in\I\backslash \I_{\sigma}}\Aut_{\D}(\V_i))&\text{otherwise.}
\end{cases}
\end{equation*}
It turns out that for special orthogonal groups, in certain cases, there is a smaller Levi subgroup more adapted for our purposes. Define condition~$(\star)$ for~$(\G,h)$ to be:
\begin{itemize}[$(\star)$]
\item $\G$ is a special orthogonal group, and there exists~$i_0\in\I$ such that~$\beta_{i_0}=0$,~$\dim_{\F}(\V_{i_0})=2$, and~$h\mid_{\V_{i_0}}$ has Witt index one.
\end{itemize}
Then we define~$\M_c(\beta)$ to be~$\M(\beta)$, unless~$(\G,h)$ satisfies~$(\star)$, in which case we set:
\begin{equation*}
\M_c(\beta)=
\G\cap( (\Aut_{\F}( \sum_{i_0\neq i\in\I^\sigma}(\V_i))\times \prod_{i\in\I\backslash \I_{\sigma}}\Aut_{\F}(\V_i)\times \mathrm{SAut}_{\F}(V_{i_0})).
\end{equation*}
where~$\mathrm{SAut}_{\F}$ denotes the~$\F$-automorphisms of determinant one. This subgroup~$\M_c(\beta)$ is the minimal Levi subgroup of~$\G$ which contains the centralizer~$\G_{\beta}$ of~$\beta$.

We call~$[\Lambda,\beta]$ an~\emph{$m$-semisimple stratum} for~$(\G,h)$ if
~$\Lambda$ is a vertex in~$\mathcal{B}(\G_{\beta},\F)$.
\begin{remark}
When~$h$ is non-trivial, strata for~$(\G,h)$ are called a self-dual semisimple stratum in earlier works including \cite{KSS,SkodInnerFormI}.   
\end{remark}

Let~$[\Lambda,\beta]$ be a semisimple stratum for~$\G$.  Associated to this data in \cite[3.1]{St08}, \cite[Definition~5.4]{SkodInnerFormI}, and \cite[\S4]{SkodlerackCuspQuart}, are
\begin{enumerate}
\item A~$\Sigma$-stable pro-$p$ subgroup~$\widetilde{\H}^1(\beta,\Lambda)$ of~$\widetilde{\G}$, and we write~$\H^1(\beta,\Lambda)=\widetilde{\H}^1(\beta,\Lambda)\cap\G$ for the associated pro-$p$ subgroup of~$\G$.
\item A finite set of (complex) characters~$\mathcal{C}_h(\beta,\Lambda)$ of~$\H^1(\beta,\Lambda)$ which we call \emph{semisimple characters} for~$\G$.  
\end{enumerate}

\begin{remark}
\begin{enumerate}
\item The set~$\mathcal{C}_h(\beta,\Lambda)$ is defined as~$\mathcal{C}(\beta,\Lambda)^{\Sigma}$ where~$\mathcal{C}(\beta,\Lambda)$ is the set of (complex) semisimple characters of~$\widetilde{\H}^1(\beta,\Lambda)$ defined in \cite{BK93,St08}, and if $h$ is non-trivial,~$\sigma\in\Sigma$ is acting via the inverse of the adjoint anti-involution on~$\G$. 
\item Semisimple characters have strong intertwining properties, in particular, for~$\theta\in\mathcal{C}_h(\beta,\Lambda)$, we have~$\I_{\G}(\theta)=\H^1(\beta,\Lambda)\G_\beta\H^1(\beta,\Lambda)$.
\end{enumerate}
\end{remark}
 If~$[\Lambda,\beta]$ is an~$m$-semisimple stratum for~$(\G,h)$ we call elements of~$\mathcal{C}_h(\beta,\Lambda)$ $m$-\emph{semisimple characters}.  If~$[\Lambda,\beta]$, $[\Lambda',\beta]$ are semisimple strata for~$\G$, then there is a canonical bijection
\[\tau_{\Lambda,\Lambda',\beta}:\mathcal{C}_h(\beta,\Lambda)\rightarrow \mathcal{C}_h(\beta,\Lambda'),\]
called \emph{transfer}.

Let~$\theta\in\mathcal{C}_h(\beta,\Lambda)$ and~$\theta'\in\mathcal{C}_h(\beta',\Lambda')$ be semisimple character.  If they intertwine in~$\G$, then letting~$\I$ (resp.~$\I'$) denote the indexing set of~$\beta$ (resp.~$\beta'$), then there is a canonical bijection~$\zeta:\I\rightarrow \I'$ we call a \emph{matching} \cite[Theorem 10.1]{SkSt}, \cite[Theorem 8.8]{KSS}, \cite[Theorem 6.3]{SkodInnerFormII}.

\subsection{Endo-parameters}
Let~$\mathcal{C}(\G)=\bigcup \mathcal{C}_h(\beta,\Lambda)$ denote the collection of all semisimple characters for~$\G$.  This collection depends only on the group~$\G$, not on the hermitian form~$h$ used to define~$\G$,  i.e., being a semisimple character for~$\G$ only depends on the group.  However, when we choose data to realize a semisimple character for~$\G$ in an explicit set associated to a stratum, we are (in the classical case) choosing a hermitian form in this data.

%
%
Intertwining in~$\G$ defines an equivalence relation on~$\mathcal{C}(\G)$ by  \cite{KSS}, \cite{SkodInnerFormI}, and the intertwining classes are called \emph{endo-parameters} for~$\G$.  

Let~$\mathfrak{t}$ be an endo-parameter for~$\G$, then the number of~$\G$-conjugacy classes of~$m$-semisimple characters for~$\G$ with endo-parameter~$\mathfrak{t}$ is in general infinite.  This motivates the introduction of a weaker notion than conjugacy  {\cite[Definition 4.8]{technicalpaper}}:

\begin{definition}Let~$\theta$ and~$\theta'$ be intertwining semisimple characters for~$\G$.  Then they are called
~c-$\G$-conjugate (\emph{cuspidally-$\G$-conjugate}) if there exist strata~$[\Lambda,\beta],[\Lambda',\beta']$ for~$(\G,h)$, and~$g\in\G$, such that 
\begin{enumerate}
\item $\theta\in\mathcal{C}_h(\beta,\Lambda)$ and~$\theta'\in\mathcal{C}_h(\beta',\Lambda')$;
\item $g V_i=V_{\zeta(i)}$ for all~$i\in\I$, where~$\zeta:\I\rightarrow \I'$ is the matching given by the intertwining of~$\theta$ and~$\theta'$;
\item $\presuper{g}(\theta_{|\H^1(\beta,\Lambda)\cap \M_c(\beta)})=\theta'_{|\H^1(\beta',\Lambda')\cap\M_c(\beta')}.$
\end{enumerate}
\end{definition}

By \cite[Proposition 4.9]{technicalpaper},~c-$\G$-conjugacy defines an equivalence relations on~$\mathcal{C}(\G)$, and we call the respective classes~\emph{c-$\G$-conjugacy classes}.
%
%
\begin{theorem}[{{\cite[Theorem 4.10]{technicalpaper}}}]\label{finiteness}
The number of c-$\G$-conjugacy classes 
of~$m$-semisimple characters for~$\G$ with a fixed endo-parameter~$\mathfrak{t}$ is finite.\end{theorem}

In fact, we can be much more precise:

\begin{proposition}[{{\cite[Proposition 4.16]{technicalpaper}}}]\label{representatives}
Let $\ft$ be an endo-parameter for~$\G$, and~$\theta\in \Cc_h(\beta,\Lambda)$ a semisimple character of endo-parameter~$\ft$, and recall we have an associated embedding~$\mathcal{B}(\mathbb{G}_\beta,\F)\hookrightarrow \mathcal{B}(\mathbb{G},\F)$.  Choose a set of representatives~$\mathrm{Vert}_\beta$ for the~$\G_\beta$-classes of vertices in a chamber of~$\mathcal{B}(\mathbb{G}_\beta,\F)$.  Then the set
\[\{\tau_{\Lambda,\Upsilon,\beta}(\theta):\Upsilon\in \mathrm{Vert}_{\beta}\}\]
forms a complete system of representatives for the c-$\G$-conjugacy classes of~$m$-semisimple characters of endo-parameter~$\ft$.
\end{proposition}

\subsection{Parabolic induction and restriction of endo-parameters}
Let~$\mathfrak{t}$ be an endo-parameter for~$\G$.  Let~$\theta$ be a semisimple character of endo-parameter~$\mathfrak{t}$, and choose a hermitian form~$h$ and stratum~$[\Lambda,\beta]$ so that~$\theta\in\mathcal{C}_h(\beta,\Lambda)$.  Then the~$\G$-conjugacy class of~$\M(\beta)$ (resp.~$\M_c(\beta)$) depends only on~$\mathfrak{t}$ \cite[Lemma 3.8]{technicalpaper}, and we denote this class by~$[\M(\mathfrak{t})]$ (resp.~$[\M_c(\mathfrak{t})]$). 

Let~$\P$ be a Levi subgroup of~$\G$ and~$\P=\M\ltimes \N$ a Levi decomposition of~$\P$.  We have non-normalized parabolic induction and restriction functors, we denote by~$i_\P^\G:\Rep_{\R}(\M)\rightarrow \Rep_{\R}(\G)$ and~$r^\G_\P:\Rep_\R(\G)\rightarrow \Rep_\R(\M)$ respectively.  Here we define maps of endo-parameters, which we will later show are compatible with parabolic induction and restriction.

We write~$\M=\prod_{j=0}^s \M_j$, with~$\M_0$ an inner form of a classical group (or trivial), and~$\M_j$ general linear groups.
\begin{definition}
An \emph{endo-parameter}~$\ft$ for~$\M$ is a tuple~$(\ft_j)_{j=0}^s$ of endo-parameters~$\ft_j$ for~$\M_j$.
\end{definition}

Let~$\ft=(\ft_j)_{j=0}^s$ be an endo-parameter for~$\M$ and~choose semisimple characters~$\theta_{j}$ for~$M_j$ of endo-parameter~$\ft_{j}$. Then~$\theta_{\M}=\bigotimes_{j=1}^s \theta_{j}$ is a semisimple character for~$\M$, and by~\cite[Proposition 5.1]{MiSt}, \cite[Theorems 6.6 and 6.10]{SkodInnerFormII}, 
we can choose a semisimple character~$\theta$ for~$\G$ with~$\theta_{|\M}=\theta_{\M}$.  We let~$i_{\M}^\G(\ft)$ denote the endo-parameter for~$\G$ defined by~$\theta$.  This is independent of the choice of~$\theta$, as any two choices for~$\theta_{\M}$ intertwine in~$\M$ (by definition) so, since the corresponding semisimple characters~$\theta$ are decomposed with respect to~$(\M,\P)$ for any parabolic~$\P$ with Levi~$\M$, these semisimple characters~$\theta$ also intertwine so give the same endo-parameter~$i_{\M}^\G(\ft)$. Moreover, parabolic induction of endo-parameters is clearly transitive.

Conversely, let~$\ft$ now be an endo-parameter for~$\G$. We define~$r_{\M}^{\G}(\ft)$ by
\[
r_{\M}^{\G}(\ft)=\{\text{endo-parameters }\ft_\M\text{ for }\M:\ft=i_{\M}^\G(\ft_\M)\}.
\]
Set~$\W_\M=\N_\G(\M)/\M$. In general,~$r_{\M}^{\G}(\ft)$ will \emph{not} consist of a single~$\W_\M$-conjugacy class of endo-parameters for~$\M$, but is a (possibly empty) finite set of~$\W_\M$-conjugacy classes. 

There is one case of particular interest, namely when~$\M\in[\M_c(\ft)]$. If~$\theta$ is any semisimple character with endo-parameter~$\ft$, and~$\theta\in\mathcal{C}_h(\beta,\Lambda)$
such that~$\M=\M_c(\beta)$, then the restriction~$\theta_{\M}:=\theta_{\vert \M}$ is a semisimple character for~$\M$, and we write~$\ft_\M$ for its corresponding endo-parameter for~$\M$.  The~$\G$-conjugacy class of this is independent of the choice of~$\theta$ by~\cite[Theorem~6.5,~Corollary~6.15]{SkodInnerFormII}, and clearly~$\ft_\M\in r_{\M}^{\G}(\ft)$. We call the~$\G$-conjugacy class of the pair~$(\M,\ft_{\M})$ the \emph{cuspidal support} of~$\ft$.

\begin{lemma}[{{\cite[Lemma 4.13]{technicalpaper}}}]\label{EPsupportmap}
Let $\ft$ be an endo-parameter for~$\G$ and let $\M\in[\M_c(\ft)]$, and~$(\M,\ft_{\M})$ in the cuspidal support of~$\ft$.  The~c-$\G$-conjugacy classes of m-semisimple characters of endo-parameter $\ft$ are in bijection with the $\M$-conjugacy classes of m-semisimple characters of endo-parameter $\ft_\M$. 
\end{lemma}

\begin{remark}
The bijection is defined by choosing a set of representatives~$\Xi_{\M}$ for~$\ft_{\M}$, and for each~$\theta_{\M}\in\Xi_{\M}$ choosing a semisimple character~$\theta$ for~$\G$ with~$\theta_{|\M}=\theta_{\M}$.
\end{remark}

\section{Heisenberg representations}\label{Heisenberg}

\subsection{Heisenberg representations and semisimple characters over~$\mathbb{Z}[1/p,\mu_{p^{\infty}}]$}\label{HeisenbergZ}
\subsubsection{Heisenberg representations over~$\mathbb{C}$} Let~$[\Lambda,\beta]$ be a semisimple stratum for~$\G$.  Then also attached to this datum, see \cite[3.1]{St08}, \cite[Definition~5.4]{SkodInnerFormI}, and \cite[\S4]{SkodlerackCuspQuart}, are~$\Sigma$-stable compact open subgroups~$\widetilde{\J}^1(\beta,\Lambda)\leqslant \widetilde{\J}(\beta,\Lambda)$ of~$\tG$ containing~$\widetilde{\H}^1(\beta,\Lambda)$, and we write
\[\J^1(\beta,\Lambda)=\widetilde{\J}^1(\beta,\Lambda)\cap\G,\quad \J(\beta,\Lambda)=\widetilde{\J}(\beta,\Lambda)\cap\G,\]
for the associated compact open subgroups of~$\G$.  The group~$\J^1(\beta,\Lambda)$ is pro-$p$, and normal in~$\J(\beta,\Lambda)$ with
\[\J(\beta,\Lambda)/\J^1(\beta,\Lambda)\simeq (\G_{\beta})_{\Lambda}/(\G_{\beta})_{\Lambda}^1,\]
a finite reductive group.  

Let~$\t\in\Cc_h(\beta,\Lambda)$ be a semisimple character.  Then there exists a unique irreducible $\mathbb{C}$-representation~$\eta$ of~$\J^1(\beta,\Lambda)$ which contains~$\t_-$, by \cite[Proposition 3.5]{St08}, {\cite[B.2]{technicalpaper}} and \cite[Proposition 4.3]{SkodlerackCuspQuart}.  These representations are called \emph{Heisenberg~$\mathbb{C}$-representations}, and this definition is extended to algebraically closed fields in \cite{RKSS}.

\subsubsection{Integral semisimple characters}
Let~$\K$ be a compact open (pro-$p$) normal subgroup of~$\H^1(\beta,\Lambda)$ such that all ($\mathbb{C}$-valued) semisimple characters in~$\Cc_h(\beta,\Lambda)$ are trivial on~$\K$.  We fix~$r$ sufficiently large, so that \emph{all} characters of all pro-$p$ subgroups of~$\P(\Lambda)$ trivial on~$\K$ take values in~$\mathbb{Z}[\mu_{p^r}]$.  In particular, for~$\theta\in\Cc_h(\beta,\Lambda)$, we can choose an integral model of~$\theta$ as a \emph{free}~$\mathbb{Z}[\mu_{p^r}]$-module of rank one, on which~$\H^1(\beta,\Lambda)$ acts via
\[\theta:\H^1(\beta,\Lambda)\rightarrow \mathbb{Z}[\mu_{p^r}]^\times.\]
Let~$\R$ be a~$\mathbb{Z}[\mu_{p^r}]$-algebra. We set~$\theta_{\R}=\theta\otimes\R$ which gives the natural action of~$\theta$ on~a free~$\R$-module of rank~one.  In particular, if~$\R$ is an algebraically closed field of characteristic different to~$p$,~$\theta_{\R}$ agrees with the previous definitions of semisimple characters.  

We write~$\Cc_{h,\R}(\beta,\Lambda)$ for the set of~$\R$-valued semisimple characters, or ``semisimple~$\R$-characters'', obtained from~$\Cc_h(\beta,\Lambda)$ by considering the natural integral structure in each semisimple character as described above (for the fixed~$\mathbb{Z}[\mu_{p^r}]$-algebra structure on~$\R$).

We record the following properties, immediate from our construction and Lemma \ref{lemmaprojectivesfinitegroups}:
\begin{lemma}\label{basicpropssemisimplechars}
Let~$\R$ be a~$\mathbb{Z}[\mu_{p^r}]$-algebra, and $\theta_{\R}\in\Cc_{h,\R}(\beta,\Lambda)$.
\begin{enumerate}
\item The underlying~$\R$-module of~$\theta_{\R}$ is a free~$\R$-module of rank one.
\item If~$p$ is invertible in~$\R$,~$\theta_{\R}$ is a projective~$\R[\H^1(\beta,\Lambda)]$-module.
\item Suppose~$\mu_{p^r}\otimes 1$ has order~$p^r$ in~$\R^\times$.  Then the natural map
\[\Cc_{h,\mathbb{Z}[\mu_{p^r}]}(\beta,\Lambda)\rightarrow \Cc_{h,\R}(\beta,\Lambda).\]
is a bijection.  
\end{enumerate}
\end{lemma}

\begin{proposition}\label{propintertwiningthetas}
Let~$\R_{0,r}=\mathbb{Z}[1/p,\mu_{p^r}]$, and~$\R$ be an~$\R_{0,r}$-algebra.  For~$g\in \G$, the intertwining of~$\theta$ is given by
\begin{equation*}
\Hom_{\R[\H^1(\beta,\Lambda)\cap \H^1(\beta,\Lambda)^g]}(\theta_\R,(\theta_\R)^g)\simeq \begin{cases}
\R&\text{if }g\in \J^1(\beta,\Lambda)\G_{\beta}\J^1(\beta,\Lambda);\\
0&\text{otherwise.}\end{cases}\end{equation*}
\end{proposition}

\begin{proof}
By Lemma \ref{Homsandscalarextension}, it suffices to consider the case~$\R_{0,r}$.   Again, by Lemma \ref{Homsandscalarextension}, we have
\begin{align*}\Hom_{\R_{0,r}[\H^1(\beta,\Lambda)\cap \H^1(\beta,\Lambda)^g]}&(\theta_{\R_{0,r}},(\theta_{\R_{0,r}})^g)\otimes \mathbb{C}\\&\simeq \Hom_{\mathbb{C}[\H^1(\beta,\Lambda)\cap \H^1(\beta,\Lambda)^g]}(\theta_{\mathbb{C}},(\theta_{\mathbb{C}})^g).\end{align*}
And by the complex setting of \cite[Proposition 3.27]{St05}, \cite[Proposition 5.15]{SkodInnerFormI}, \cite[Theorem 6.12]{SkodInnerFormII} we have
\[ \Hom_{\mathbb{C}[\H^1(\beta,\Lambda)\cap \H^1(\beta,\Lambda)^g]}(\theta_{\mathbb{C}},(\theta_{\mathbb{C}})^g)\simeq\mathbb{C}.\]
if~$g\in\J^1(\beta,\Lambda)\G_{\beta}\J^1(\beta,\Lambda)$ and is zero otherwise.  

As~$\Hom_{\R_{0,r}[\H^1(\beta,\Lambda)\cap \H^1(\beta,\Lambda)^g]}(\theta_{\R_{0,r}},(\theta_{\R_{0,r}})^g)\subseteq \Hom_{\R_{0,r}}(\theta_{\R_{0,r}},(\theta_{\R_{0,r}})^g)\simeq \R_{0,r}$, it is a torsion-free~$\R_{0,r}$-module, and
\begin{align*}
\Hom&_{\R_{0,r}[\H^1(\beta,\Lambda)\cap \H^1(\beta,\Lambda)^g]}(\theta_{\R_{0,r}},(\theta_{\R_{0,r}})^g)\\&\hookrightarrow \Hom_{\R_{0,r}[\H^1(\beta,\Lambda)\cap \H^1(\beta,\Lambda)^g]}(\theta_{\R_{0,r}},(\theta_{\R_{0,r}})^g)\otimes \mathbb{C}
\end{align*}
and hence it is zero if~$g\not\in\J^1(\beta,\Lambda)\G_{\beta}\J^1(\beta,\Lambda)$.  

Suppose then~$g\in\H^1(\beta,\Lambda)\G_{\beta}\H^1(\beta,\Lambda)$.  Then, as~$\theta_{\mathbb{C}}$ is a character on the vector space~$\mathbb{C}$, the Hom-space is non-zero if and only if the restrictions of the morphisms~
\begin{align*}
\theta_{\mathbb{C}}&:\H^1(\beta,\Lambda)\rightarrow \mathbb{C}^\times\\
(\theta_{\mathbb{C}})^g&:\H^1(\beta,\Lambda)^g\rightarrow \mathbb{C}^\times
\end{align*}
to~$\H^1(\beta,\Lambda)\cap \H^1(\beta,\Lambda)^g$ are equal.  In particular, by definition, the restrictions of~$\theta_{\R_{0,r}}$ and~$(\theta_{\R_{0,r}})^g$ to~$\H^1(\beta,\Lambda)\cap \H^1(\beta,\Lambda)^g$ are equal, and hence for~$g\in\J^1(\beta,\Lambda)\G_{\beta}\J^1(\beta,\Lambda)$ we have
\[\Hom_{\R_{0,r}[\H^1(\beta,\Lambda)\cap \H^1(\beta,\Lambda)^g]}(\theta_{\R_{0,r}},(\theta_{\R_{0,r}})^g)=\End_{\R_{0,r}}(\theta_{\R_{0,r}},(\theta_{\R_{0,r}})^g)\simeq \R_{0,r}\]
as the underlying~$\R_{0,r}$-modules of~$\theta_{\R_{0,r}},(\theta_{\R_{0,r}})^g$ are free of rank one over~$\R_{0,r}$.
\end{proof}

\subsubsection{Integral Heisenberg representations} One can also choose a natural~integral model for a Heisenberg representation.  
We continue with the above notation,~$[\Lambda,\beta]$ be a semisimple stratum for~$\G$.

\begin{lemma}
\begin{enumerate}
\item The derived subgroup of~$\J^1(\beta,\Lambda)$ satisfies
\[[ \J^1(\beta,\Lambda), \J^1(\beta,\Lambda)]\leqslant \H^1(\beta,\Lambda).\]
\item The quotient~$\J^1(\beta,\Lambda)/\H^1(\beta,\Lambda)$ is an elementary $p$-group.
\end{enumerate}
\end{lemma}

\begin{proof} 
Let~$\L/\F$ be a maximal unramified field extension in~$\D$. We write $\J^1,\ \H^1,\ \J^1_\L,\  \H_\L^1$ for the subgroups
\[\J^1(\beta,\Lambda),\ \H^1(\beta,\Lambda),\ \J^1(\beta\otimes_\F 1,\Lambda),\ \H^1(\beta\otimes_\F1,\Lambda)\]
of~$\G$ and~$\tG\otimes \L$, respectively. By~\cite[Proposition~5.6]{SkodInnerFormI} we have
$\H^1=\H^1_{\L}\cap \G$ and therefore~$\J^1/\H^1$ identifies with a subgroup of~$\J^1_\L/\H^1_\L$. The first assertion follows  from~\cite[Corollary 3.12]{St05} and the second assertion from~\cite[Lemma~3.11 (ii) and (iv)]{St05}.
\end{proof}

Hence the pairing
\begin{align*}
{\mathbf{k}}_{\theta}: \J^1(\beta,\Lambda)/\H^1(\beta,\Lambda)\times \J^1(\beta,\Lambda)/\H^1(\beta,\Lambda)&\rightarrow \mathbb{Z}[\mu_{p^r}]\\
(j_1,j_2)&\mapsto \theta[j_1,j_2]
\end{align*}
takes values in the $p$-th roots of unity~$\mu_p$ (as $\J^1(\beta,\Lambda)/\H^1(\beta,\Lambda)$ is an elementary $p$-group). Identifying,~the cyclic group of $p$-th roots of unity with~$\mathbb{F}_p$, then by \cite[Proposition 3.28]{St05}, \cite[Proposition 3.9]{SecherreCh}, \cite[Lemma~B.4]{technicalpaper} and \cite[Lemma 4.2]{SkodlerackCuspQuart}, the pairing~$\mathbf{k}_{\theta}$ defines a non-degenerate symplectic form on the~$\mathbb{F}_p$-vector space~$\J^1(\beta,\Lambda)/\H^1(\beta,\Lambda)$.  

Choose a polarization of~$\J^1(\beta,\Lambda)/\H^1(\beta,\Lambda)=\mathcal{W}_+\oplus \mathcal{W}_-$ with respect to this form.  The inverse image~$\J^1(\beta,\Lambda)_+$ of~$\mathcal{W}_+$ in~$\J^1(\beta,\Lambda)$ defines a maximal abelian subgroup of~$\J^1(\beta,\Lambda)/\H^1(\beta,\Lambda)$, and we choose an extension of~$\theta$ to a character~$\theta_+:\J^1(\beta,\Lambda)_+\rightarrow \mathbb{Z}[\mu_{p^r}]^\times$ acting on the same free~$\mathbb{Z}[\mu_{p^r}]$-module as~$\theta$.

For any~$\mathbb{Z}[\mu_{p^r}]$-algebra~$\R$, we then define
\[\eta_{\R}=\eta_{\R}(\mathcal{W}_+,\theta_+)=\ind_{\J^1(\beta,\Lambda)_+}^{\J^1(\beta,\Lambda)}(\theta_+)\otimes \R\]
to be the Heisenberg~$\R$-representation associated to~$(\mathcal{W}_+,\theta_+)$.  
\begin{lemma}\label{Heislemma}
\begin{enumerate}
\item\label{classHeiscase} Suppose~$\R$ is an algebraically closed field of characteristic~different from $p$, then~$\eta_\R$ defines the unique isomorphism class of Heisenberg~$\R$-representations of previous works, and in particular~$\eta_{\R}$ is irreducible.
\item The underlying~$\R$-module of~$\eta_{\R}$ is free of rank
\[[\J^1(\beta,\Lambda):\J^1(\beta,\Lambda)_+]=[\J^1(\beta,\Lambda):\H^1(\beta,\Lambda)]^{1/2}.\]
\item  If~$p$ is invertible in~$\R$, then any Heisenberg~$\R$-representation is projective.  
\item If~$p$ is invertible in~$\R$, then the isomorphism class of~$\eta_{\R}(\mathcal{W}_+,\theta_+)$ is independent of the choice of~$(\mathcal{W}_+,\theta_+)$, and~$\End_{\R[\J^1(\beta,\Lambda)]}(\eta_{\R}(\mathcal{W}_+,\theta_+))$ is free of rank one over~$\R$.
\end{enumerate}
\end{lemma}
\begin{proof}
The first two statements are straightforward.  The~$\R$-representation~$\theta_{+}\otimes \R$ is projective by Lemma \ref{lemmaprojectivesfinitegroups}, and hence~$\eta_{\R}$ is projective as compact induction from an open subgroup preserves projectivity.  The final statement obviously reduces to the case~$\R_{0,r}=\mathbb{Z}[1/p,\mu_{p^r}]$ (using Lemma \ref{Homsandscalarextension} for the statement about endomorphism algebras).

By Mackey theory, we have a~$\R_{0,r}$-module decomposition
\begin{align*}
\Hom_{\R_{0,r}[\J^1(\beta,\Lambda)]}&(\eta_{\R_{0,r}}(\mathcal{W}_+,\theta_+),\eta_{\R_{0,r}}(\mathcal{W}_+',\theta_+'))\\&\simeq \bigoplus\Hom_{\R_{0,r}[\J^1(\beta,\Lambda)_+\cap (\J^1(\beta,\Lambda)_+')^g]}(\theta_+,(\theta_+')^g)\end{align*}
where for~$\J^1(\beta,\Lambda)_+'$ the prime denotes the subgroup defined by the maximal isotropic subspace~$\mathcal{W}_+'$.  As in the proof of Proposition \ref{propintertwiningthetas}, if any of these Hom-spaces are non-trivial, then they are free~$\R_{0,r}$-modules of rank one.  Hence,~$\Hom_{\R_{0,r}[\J^1(\beta,\Lambda)]}(\eta_{\R_{0,r}}(\mathcal{W}_+,\theta_+),\eta_{\R_{0,r}}(\mathcal{W}_+',\theta_+'))$ is a free~$\R_{0,r}$-module.

By Lemma \ref{Homsandscalarextension},
\begin{align*}
\Hom_{\R_{0,r}[\J^1(\beta,\Lambda)]}&(\eta_{\R_{0,r}}(\mathcal{W}_+,\theta_+),\eta_{\R_{0,r}}(\mathcal{W}_+',\theta_+'))\otimes \mathbb{C}\\
&\simeq \Hom_{\mathbb{C}[\J^1(\beta,\Lambda)]}(\eta_{\mathbb{C}}(\mathcal{W}_+,\theta_+),\eta_{\mathbb{C}}(\mathcal{W}_+',\theta_+'))\simeq \mathbb{C},\end{align*}
and in particular the integral $\Hom$-space is non-zero, and as it is a free~$\R_{0,r}$-module, the~$\Hom$-space is a free~$\R_{0,r}$-module of rank one.  

Choose any generator~$\phi$ of the~$\Hom$-space as a~$\R_{0,r}$-module, then
\[\phi\otimes 1:\eta_{\R_{0,r}}(\mathcal{W}_+,\theta_+)\otimes\mathbb{C}\rightarrow\eta_{\R_{0,r}}(\mathcal{W}_+',\theta_+')\otimes \mathbb{C},\]
is an isomorphism, as any non-zero morphism is by part \ref{classHeiscase}.  Moreover, as~$\mathbb{C}/\K$ is faithfully flat, it follows that
\[\phi\otimes 1:\eta_{\R_{0,r}}(\mathcal{W}_+,\theta_+)\otimes\K\rightarrow\eta_{\R_{0,r}}(\mathcal{W}_+',\theta_+')\otimes\K,\]
is an isomorphism.

 Moreover, for any maximal ideal~$\mathfrak{m}\in\text{m-}\Spec(\R_{0,r})$ consider the morphism
\[\phi\otimes 1:\eta_{\R_{0,r}}(\mathcal{W}_+,\theta_+)\otimes\R/\mathfrak{m}\rightarrow\eta_{\R_{0,r}}(\mathcal{W}_+',\theta_+')\otimes \R/\mathfrak{m}.\]
Now~$\Hom_{\R/\mathfrak{m}[\J^1(\beta,\Lambda)]}(\eta_{\R/\mathfrak{m}}(\mathcal{W}_+,\theta_+),\eta_{\R/\mathfrak{m}}(\mathcal{W}_+',\theta_+'))$ is free of rank one over~$\R/\mathfrak{m}$ (say by Lemma  \ref{Homsandscalarextension} again), and by faithfully flat descent from~$\Fl\simeq \overline{\R/\mathfrak{m}}$ (and part \ref{classHeiscase}), any non-zero morphism is an isomorphism.  Moreover,~$\phi\otimes 1$ is non-zero because as an~$\R_{0,r}$-module it generates~$\Hom_{\R_{0,r}[\J^1(\beta,\Lambda)]}(\eta_{\R_{0,r}}(\mathcal{W}_+,\theta_+),\eta_{\R_{0,r}}(\mathcal{W}_+',\theta_+'))$ and hence generates~
\begin{align*}
\Hom_{\R_{0,r}[\J^1(\beta,\Lambda)]}&(\eta_{\R_{0,r}}(\mathcal{W}_+,\theta_+),\eta_{\R_{0,r}}(\mathcal{W}_+',\theta_+'))\otimes \R/\mathfrak{m}\\&\simeq \Hom_{\R/\mathfrak{m}[\J^1(\beta,\Lambda)]}(\eta_{\R/\mathfrak{m}}(\mathcal{W}_+,\theta_+),\eta_{\R/\mathfrak{m}}(\mathcal{W}_+',\theta_+'))\end{align*} as an~$\R_{0,r}/\mathfrak{m}$-module.  It follows from Lemma \ref{isomatprimesDDgivesiso}, that~$\phi$ is an isomorphism.
\end{proof}

By \cite[Proposition 3.5]{St08}, \cite[Proposition~B.8]{technicalpaper}, \cite[Proposition 4.3]{SkodlerackCuspQuart}, we deduce that Heisenberg~$\R$-representations enjoy the following intertwining properties:

\begin{corollary}\label{etaintertwining}
Let~$\R$ be a~$\R_{0,r}$-algebra, and~$\eta_{\R}$ be a Heisenberg~$\R$-representation constructed as above.  For~$g\in \G$, the intertwining of~$\eta$ is given by
\begin{equation*}
\Hom_{\R[\J^1(\beta,\Lambda)\cap \J^1(\beta,\Lambda)^g]}(\eta_{\R},\eta_{\R}^g)\simeq \begin{cases}
\R&\text{if }g\in \J^1(\beta,\Lambda)\G_{\beta}\J^1(\beta,\Lambda);\\
0&\text{otherwise.}\end{cases}\end{equation*}
\end{corollary}

\begin{proof}
Suppose~$g$ is in the intertwining.  Then as~$\eta_{\R}=\ind_{\J^1(\beta,\Lambda)_+}^{\J^1(\beta,\Lambda)}(\theta_+)$, by Mackey theory we have an~$\R$-module decomposition\[\Hom_{\R[\J^1(\beta,\Lambda)\cap \J^1(\beta,\Lambda)^g]}(\eta_{\R},\eta_{\R}^g)\] into a direct sum of Hom-spaces between characters, 
and again we find only one space is non-zero by extending scalars to~$\mathbb{C}$ where we know the intertwining.  The rest follows mutatis mutandis the proof of Proposition \ref{propintertwiningthetas}.
\end{proof}

\subsection{Heisenberg representations and parabolic induction}
We let~$[\Lambda,\beta]$ be a semisimple stratum for~$(\G,h)$. We take~$\R_{0,r}=\mathbb{Z}[1/p,\mu_{p^r}]$, for~$r$ sufficiently large for all semisimple characters in~$\mathcal{C}_h(\beta,\Lambda)$, and assume throughout this section that~$\R$ is an~$\R_{0,r}$-algebra.

Let~$\V=\bigoplus_{j\in\mathcal{S}}\V^{(j)}$ be a decomposition of~$\V$ which is properly subordinate to~$[\Lambda,\beta]$ if~$h=0$ and properly self-dual subordinate to the stratum if~$h\neq 0$, in the sense of~\cite[Definition~8.(ii),~8.2]{SkodlerackCuspQuart} (cf.~\cite[Definition 5.1]{St08}).  Let~$\M=\G\cap (\prod \Aut_{\D}(V^{(j)}))$ be the Levi subgroup of~$\G$ defined by this decomposition, and~$\P=\M\N$ be any parabolic subgroup of~$\G$ with Levi factor~$\M$, and~$\P^{\circ}=\M\N^{\circ}$ the opposite parabolic subgroup of~$\P$, so that~$\P\cap\P^{\circ}=\M$. Associated to the decomposition, we also get a decomposition of the stratum as a sum of semisimple strata~$[\Lambda^{(j)},\beta^{(j)}]$. 

In this situation, let~$\theta\in\Cc_h(\beta,\Lambda)$ which we consider as a character~$\theta:\H^1(\beta,\Lambda)\rightarrow \mathbb{Z}[\mu_{p^r}]^\times$ acting on a free module of rank one over~$\R$.  Then~$\theta\mid_{\H^1(\beta,\Lambda)\cap \N},\theta\mid_{\H^1(\beta,\Lambda)\cap \N^{\circ}}$ are trivial, and
\[\theta\mid_{\H^1(\beta,\Lambda)\cap \M}=\bigotimes \theta^{(j)},\]
decomposes into a product of semisimple characters~$\theta^{(j)}\in\Cc_h(\Lambda^{(j)},\beta^{(j)})$ (which we consider as acting on the same free rank one $\R$-module as~$\theta$). 
In the following, we abbreviate~$\H^1=\H^1(\beta,\Lambda)$ etc.

\begin{lemma}
\begin{enumerate}
\item The subspaces~$(\J^1\cap \N)/(\H^1\cap \N)$ and~$(\J^1\cap \N^{\circ})/(\H^1\cap\N^{\circ})$ of~$\J^1/\H^1$ are both totally isotropic for the form~$\mathbf{k}_{\theta}$ and orthogonal to~$(\J^1\cap\M)/(\H^1\cap\M)$.
\item\label{polii} The restriction of~$\mathbf{k}_{\theta}$ to~\[(\J^1\cap\M)/(\H^1\cap \M)=\prod \J^1(\beta^{(j)},\Lambda^{(j)})/\H^1(\beta^{(j)},\Lambda^{(j)})\] is the orthogonal sum of the pairings~$\mathbf{k}_{\theta^{(j)}}$. 
\item\label{poliii} We have an orthogonal sum decomposition
\[\J^1/\H^1=((\J^1\cap\M)/(\H^1\cap \M)) \oplus \left((\J^1\cap \N)/(\H^1\cap \N)\times (\J^1\cap \N^{\circ})/(\H^1\cap\N^{\circ})\right).\]
\end{enumerate}
\end{lemma}
\begin{proof}
Follows from a mild adaptation of~\cite[7.2.3]{BK93}~and~\cite[Lemma 5.6]{St08}.
\end{proof}

We now choose our totally isotropic subspace~$\mathcal{W}_+$ of~$\J^1/\H^1$ with respect to the decomposition of~\ref{poliii}: we choose a totally isotropic subspace~$\mathcal{W}_{\M,+}$ of the image of~$(\J^1\cap\M)$ in~$\J^1/\H^1$ and take~$\W_{\N,+}$ the image of~$(\J^1\cap \N)$ in~$\J^1/\H^1$
\[\mathcal{W}_+=\mathcal{W}_{\M,+}\oplus \W_{\N,+}.\]
Then write~$\J^1_{+}$ for the pre-image in~$\J^1$ of~$\mathcal{W}_+$ .
Moreover, by \ref{polii}, we can decompose
\[\mathcal{W}_{\M,+}=\prod \mathcal{W}_{(j),+}\]
where~$ \mathcal{W}_{(j),+}$ is a totally isotropic subspace of~$ \J^1(\beta^{(j)},\Lambda^{(j)})/\H^1(\beta^{(j)},\Lambda^{(j)})$ with respect to~$\mathbf{k}_{\theta^{(j)}}$, and choosing extensions~$\theta^{(j)}_+$ of~$\theta^{(j)}$ to~$\J^1(\beta^{(j)},\Lambda^{(j)})_+$, acting on the same free~$\R$-module of rank one as~$\theta^{(j)}$, we let~$\theta_{\M,+}=\bigotimes \theta^{(j)}_+$. We also write~$\theta_+$ for the extension of~$\theta$ to~$\J^1_{+}$ which is trivial on~$\J^1\cap\N$ and agrees with~$\theta_{\M,+}$ on~$\J^1_{+}\cap\M = \prod \J^1(\beta^{(j)},\Lambda^{(j)})_+$.

We set
\[\eta_{\R}(\mathcal{W}_{\M,+},\theta_{\M,+}):=\bigotimes\eta_{\R}(\mathcal{W}_{(j),+},\theta^{(j)}_+).\]
By Lemma \ref{Heislemma} the isomorphism class of~$\eta_{\R}(\mathcal{W}_{\M,+},\theta_{\M,+})$, which we denote by~$\eta_{\M,\R}$, is independent of the choice of~$\mathcal{W}_{\M,+}$ and~$\theta_{\M,+}$. Letting~
\[
\J^1_{\P}=(\H^1\cap\N^{\circ})(\J^1\cap \P),\]
we define~$\eta_{\P,\R}(\mathcal{W}_{\M,+},\theta_{\M})$, to be the unique representation of~$\J^1_{\P}$ on the 
space of~$\eta_{\R}(\mathcal{W}_{\M,+},\theta_{\M,+})$ (which is a free~$\R$-module of rank~$[(\J^1\cap\M):(\H^1\cap\M)]^{1/2}$) on which~$\J^1\cap\M$ acts via~$\eta_{\R}(\mathcal{W}_{\M,+},\theta_{\M,+})$ and~$\J^1\cap\N$ and~$\H^1\cap\N^{\circ}$ act trivially; this does indeed define a representation because we can also realise it as~$\ind_{\J^1_{+}}^{\J^1_{\P}}(\theta_+)$.
It follows from Lemma \ref{Heislemma} that its isomorphism class, which we denote by~$\eta_{\P,\R}$, is independent of the choice of polarization. Moreover, by transitivity of induction, we have 
\[\eta_{\R}\simeq \ind_{\J^1_{+}}^{\J^1_{h}}(\theta_+) \simeq \ind_{\J^1_{\P}}^{\J^1_{h}}\ind_{\J^1_{+}}^{\J^1_{\P}}(\theta_+) \simeq\ind_{\J^1_{\P}}^{ \J^1}(\eta_{\P,\R}).\] 
The main result we will need on Heisenberg representations is:
\begin{theorem}\label{theoremHcovers}
For any parabolic subgroups~$\P,\P'$ with common Levi factor~$\M$, we have isomorphisms:
\begin{equation}\label{Coversprop}
\ind_{\J^1}^{\G}(\eta_{\R})\simeq \ind_{\J^1_{\P}}^{\G}(\eta_{\P,\R})\simeq \ind_{\P'}^{\G}(\ind_{\J^1\cap\M}^{\M}(\eta_{\M,\R})).\end{equation}
In particular, the isomorphism classes of~$\ind_{\J^1_{\P}}^{\G}(\eta_{\P,\R})$ and~$\ind_{\P'}^{\G}(\ind_{\J^1\cap\M}^{\M}(\eta_{\M,\R}))$ are independent of the choice of parabolic.
\end{theorem}

\begin{proof}
The first morphism is just transitivity of induction.  As~$\ind_{\J^1}^{\G}(\eta_{\R})$ is independent of the parabolic~$\P$, so too is the isomorphism class of~$\ind_{\J^1_{\P}}^{\G}(\eta_{\P,\R})$, and it suffices to prove the second isomorphism for~$\P'=\P^{\circ}$.  Fix an~$\R$-valued Haar measure on~$\N$, then we have a morphism
\begin{align*}
\Psi:\ind_{\J^1_{\P}}^{\G}(\eta_{-,\P})&\rightarrow \ind_{\P^{\circ}}^{\G}(\ind_{\J^1\cap\M}^{\M}(\eta_{-,\M}))\\
f&\mapsto \frac{1}{|\J\cap\N^{\circ}|}\int_{\N^{\circ}} f(nx) dn,
\end{align*}
where we choose to normalize our morphism by~$|\J\cap\N^{\circ}|$ (as in the complex case) which is a power of $p$.  Note, over an algebraically closed field, the pair~$(\J^1_{\P},\eta_{-,\P})$ is a~$\G$-cover of~$(\J^1\cap{\M},\eta_{-,\M})$ relative to~$\P$.  When $h\neq 0$ and $\D=\F$ this is deduced in \cite[Theorem 9.3]{RKSS} from the construction of covers of \cite{St08}, following an idea of M\'inguez and S\'echerre.  The same argument works in the other cases starting with the input being the construction of covers of \cite{SecherreStevensVI, SkodlerackYe}.  Hence by \cite[Th\'eor\`eme 2]{Blondel},~$\Psi\otimes \K$ is an isomorphism for all algebraically closed fields fields, and hence $\Psi$ defines an isomorphism by Lemma~\ref{isomatprimesDDgivesiso}.
\end{proof}

\begin{proposition}\label{EssGisoinductions}
Let~$\theta,\theta'$ be c-$\G$-conjugate m-semisimple characters of endo-parameter~$\mathfrak{t}$ for~$\G$.  Let~$\eta,\eta'$ be Heisenberg representations associated to~$\theta,\theta'$ on~$\J^1,\J'^1$ (defined after having fixed parametrizations).  Then~$\ind_{\J^1}^{\G}(\eta)\simeq \ind_{\J'^1}^{\G}(\eta')$.
\end{proposition}

\begin{proof}
By conjugating in~$\G$, we may assume that~$\theta\in\mathcal{C}_h(\beta,\Lambda)$,~$\theta'\in\mathcal{C}_h(\beta',\Lambda')$, the matching given by their intertwining is the identity,~$\M_c(\beta)=\M_c(\beta')$ which we abbreviate to~$\M$, and~$\theta_{|\M\cap \H^1}=\theta'_{|\M\cap\H'^1}$.  %
%
%
%
In particular, we have~$\M\cap \J^1=\M\cap \J'^1$, noting that~$\M\cap \J^1$ is the pro-$p$-radical of the the maximal compact subgroup of the normalizer of~$\theta|_{\M\cap\H^1}$ in~$\M$. Thus the Heisenberg extensions~$\eta_{\M}$ and~$\eta'_{\M}$ also coincide and the result follows immediately from~\eqref{Coversprop}.
\end{proof}

\section{Endo-splitting}\label{secESing}

Let~$\Rep_\R(\G)$ denote the abelian category of smooth $\R$-representations of~$\G$.  We fix our base ring~$\R_0=\mathbb{Z}[1/p,\mu_{p^{\infty}}]$, which contains all values of all semisimple characters over~$\mathbb{C}$, and base rings~$\R_{0,r}=\mathbb{Z}[1/p,\mu_{p^{r}}]$ which contain all values of all semisimple characters over~$\mathbb{C}$ of depth $\leq d(r)$, so~$\R_0=\bigcup \R_{0,r}$.

\begin{definition}
Let~$\R$ be a~$\mathbb{Z}[1/p]$-algebra, $\Sigma$ a collection of finitely generated projective smooth~$\R$-representations of compact open subgroups of~$\G$, and~$\mathcal{H}$ a full abelian subcategory of~$\Rep_{\R}(\G)$.  We say that~$\Sigma$ \emph{exhausts}~$\mathcal{H}$ if, for any smooth~$\R$-representation~$\pi$ of~$\G$ contained in~$\mathcal{H}$, there exists~$(\K,\rho)\in\Sigma$ such that~$\Hom_{\R[\K]}(\rho,\pi)\neq 0$.
\end{definition}

We begin by refining a result of Dat \cite[Propositions 7.5 \& 8.5]{Dat09}, that semisimple~$\R_0$-characters for~$\G$ exhaust the category of smooth~$\R_0$-representations.  

\begin{proposition}\label{exhaustionprop}
\begin{enumerate}[(1)]
\item \begin{enumerate}
\item Suppose~$\R$ is an~$\R_{0,r}$-algebra.  The collection of~$m$-semisimple~$\R$-characters for~$\G$ of depth~$\leq d(r)$ exhausts~$\Rep_{\R}(\G)_{\leqslant d(r)}$.
\item Suppose~$\R$ is an~$\R_0$-algebra.  The collection of~$m$-semisimple~$\R$-characters for~$\G$ exhausts~$\Rep_{\R}(\G)$. 
\end{enumerate}
\item\begin{enumerate}
\item\label{parta} Suppose~$\R$ is an~$\R_{0,r}$-algebra.  The collection of~Heisenberg~$\R$-representations of~$m$-semisimple~$\R$-characters for~$\G$ of depth~$\leq d(r)$ exhausts~$\Rep_{\R}(\G)_{\leqslant d(r)}$.
\item\label{partb} Suppose~$\R$ is an~$\R_0$-algebra.  The collection of~Heisenberg~$\R$-representations of~$m$-semisimple~$\R$-characters for~$\G$ exhausts~$\Rep_{\R}(\G)$. 
\end{enumerate}
\end{enumerate}
\end{proposition}
\begin{proof}
First, we prove this in restricted depth~$\leqslant d(r)$ for~$\R_{0,r}$-representations.   The pattern of proof is identical for $m$-semisimple characters, or their Heisenberg representations -- to reduce to the case of an algebraically closed field, considered in previous works, so we just give the proof for Heisenberg~ representations.  Let~$\pi$ denote an~$\R_{0,r}$-representation, we need to show there exists a Heisenberg~$\R_{0,r}$-representation~$\eta$ of an~$m$-semisimple~$\R_{0,r}$-character so that
\[\Hom_{\R_{0,r}[\J^1]}(\eta,\pi)\neq0\]
Every smooth~$\R_{0,r}$-representation has an irreducible subquotient, and~$\Hom_{\R_{0,r}[\X^1]}(\eta,-)$ is exact by projectivity of~$\eta$, so we can reduce to the case that~$\pi$ is irreducible.
%
%

By Corollary \ref{NFcases} and Lemma \ref{TFdontappear}\footnote{In fact, we could avoid using this lemma and consider the (empty by this lemma) torsion free case in the same way we consider the torsion case.}, there exists a unique maximal ideal~$\mathfrak{m}$ of~$\R_{0,r}$ which annihilates~$\pi$, and we let~$\R=\overline{\R_{0,r}/\mathfrak{m}}$ an algebraic closure of the finite field~$\R_{0,r}/\mathfrak{m}$ (which has characteristic different from~$p$).  As in the same corollary,~$\pi\otimes(\R_{0,r}/\mathfrak{m})$ is an irreducible~$(\R_{0,r}/\mathfrak{m})$-representation, and hence~$\pi\otimes\R$ is a non-zero~$\R$-representation of~$\G$ and thus contains a Heisenberg~$\R$-representation~$\widetilde{\eta}$ of an~$m$-semisimple~$\R$-character~$\widetilde{\theta}$ by \cite[Theorem 5.7]{technicalpaper}.
%

Let~$\theta$ be a semisimple~$\R_{0,r}$-character (by definition acting on a free~$\R_{0,r}$-module of rank one) lifting~$\overline{\theta}$ (possible by construction of semisimple characters - see Lemma \ref{basicpropssemisimplechars}), and~$\eta$ the associated Heisenberg~$\R_{0,r}$-representation.  Then, by Lemma \ref{Homsandscalarextension},
\[\Hom_{\R_{0,r}[\H^1]}(\eta, \pi)\otimes\R\simeq \Hom_{\R[\H^1]}(\widetilde{\eta},\pi\otimes\R)\neq 0,\]
and \ref{parta} follows for~$\R=\R_{0,r}$.  Now, let~$\pi'$ is a smooth~$\R$-representation where~$\R$ is an~$\R_{0,r}$-algebra, then we have an isomorphism
\begin{align}
\label{equationreducing}\Hom_{\R[\H^1]}(\eta\otimes_{\R_{0,r}}\R,\pi')&\simeq \Hom_{\R_{0,r}[\H^1]}(\eta,\pi'_{|\R_{0,r}})\\
\notag\phi&\mapsto [ \psi_{\phi}:t\mapsto \phi(t\otimes 1)]\\
\notag[\phi_{\psi}:t\otimes r\mapsto r \psi(t)]&\mapsfrom \psi,
\end{align}
and so \ref{parta} follows, and~\ref{partb} follows immediately from \ref{parta}.
\end{proof}

\begin{definition}\label{Progen}
Let~$\ft$ be an endo-parameter for~$\G$.  Let $\P(\ft)$ denote the following finitely generated projective~$\R_{0,r}$-representation of~$\G$
\[\P(\ft)=\bigoplus_{i=1}^{r}\ind_{\J_i^1}^\G(\eta_i)\]
where 
\begin{enumerate}
\item The set $\{(\H_i^1=\H^1(\beta_i,\Lambda_i),\theta_i):1\leqslant i\leqslant r\}$ is a (finite by Theorem \ref{finiteness}) set of representatives for the set of c-$\G$-conjugacy classes of m-semisimple~$\R_{0,r}$-characters of endo-parameter $\ft$.\footnote{Note, we can first choose a set of representatives over~$\mathbb{C}$, and then we can choose~$r$ minimal so that these semisimple~$\mathbb{C}$-characters are valued in~$\mathbb{Z}[\mu_{p^r}]^\times$.}
\item $(\J_i^1,\eta_i)$ is the unique (up to isomorphism) Heisenberg~$\R_{0,r}$-representation of~$\J_i^1=\J^1(\beta_i,\Lambda_i)$ associated to~$\theta_i$.
\end{enumerate}
We extend this definition to Levi subgroups:  Let~$\M$ be a Levi subgroup of~$\G$, and~$\ft_{\M}$ an endo-parameter for~$\M$.  Recall, that writing~$\M=\prod_{j=0}^s \M_i$, as a product of with~$\M_0$ a (quarternionic) classical group (or not appearing in the product) and~$\M_i$ (inner forms of) general linear groups, we can write~$\ft_{\M}=(\ft_j)_{j=0}^s$, with~$\ft_j$ endo-parameters for~$\M_j$.    We let~$\P(\ft_{\M})$ denote the following finitely generated projective~$\R_{0,r}$-representation of~$\M$
\[\P(\ft_{\M})=\bigotimes_{j=0}^s \P(\ft_j)\]
\end{definition}

We observe some simple properties of~$\P(\ft_{\M})$:
\begin{itemize}
\item $\P(\ft)_{\M}$ is \emph{finitely generated and projective} as compact induction from an open subgroup preserves finite generation, and preserves projectivity (as it is left adjoint to restriction, an exact functor).
\item While the definition of~$\P(\ft_{\M})$ depends on the choice of the representatives~$\theta_i$ (chosen for each~$\P(\ft_j)$), the isomorphism class of $\P(\ft_{\M})$ is independent of this choice by Proposition \ref{EssGisoinductions}.
\item For~$\ft_{\M}\neq \ft'_{\M}$,~$\Hom_{\R[\M]}(\P(\ft_{\M}),\P(\ft'_{\M}))=0$, by Mackey theory as the~$m$-semisimple characters and the Heisenberg representations~$\P(\ft_{\M})$ is induced from do not intertwine with those which induce~$\P(\ft'_{\M})$.
\end{itemize}

%
%
From now on, we fix a nice choice of representatives, and a choice of~$\P(\ft)$ representing its isomorphism class by Lemma \ref{representatives}:

\begin{lemma}\label{lemmaniceprogchoice}
 \begin{enumerate}
 \item Choose~$\theta\in\Cc_h(\beta,\Upsilon)$ a semisimple character of endo-parameter~$\ft$, and fix a set of representatives~$\mathrm{Vert}_\beta$ for the~$\G_\beta$-classes of vertices in a fixed chamber of~$\mathcal{B}(\mathbb{G}_\beta,\F)$. Let~$\theta_{\Lambda}=\tau_{\Upsilon,\Lambda,\beta}(\theta)$, and~$\eta_{\Lambda}$ the Heisenberg~$\R_{0,r}$-representation associated to~$\theta_{\Lambda}$, and write~$\J^1_{\Lambda}:=\J^1(\beta,\Lambda)$.   Then setting
\[\P(\theta):= \bigoplus_{\Lambda\in \mathrm{Vert}_{\beta}}\ind_{\J^1_{\Lambda}}^{\G}(\eta_{\Lambda}),\]
we have~$\P(\ft)\simeq \P(\theta)$.
\item  Let~$\M$ be a Levi subgroup of~$\G$, and~$\theta=\bigotimes\theta_i$ a semisimple character of endo-parameter~$\ft_{\M}$.  Then, setting~$\P(\theta_{\M})=\bigotimes \P(\theta_i)$ (after fixing representatives for the vertices in each~$\mathcal{B}(\M_{i,\beta_i},\F)$), we have~$\P(\ft_{\M})\simeq \P(\theta_{\M})$.\end{enumerate}
%
\end{lemma}

\begin{definition}
Let~$\ft_{\M}$ be an endo-parameter for~$\M$.   A smooth $\R$-representation $\pi$ of $\M$ is of \emph{class} $\ft_{\M}$ if and only if every semisimple~$\R$-character~$\theta$ for~$\M$ such that~$\Hom_{\R[\H^1]}(\theta,\pi)\neq 0$ has endo-parameter~$\ft_{\M}$.
\end{definition}

As semisimple~$\R$-characters are projective, a smooth~$\R$-representation~$\pi$ of~$\M$ has class~$\ft_{\M}$ if and only if every (irreducible) subquotient of~$\pi$ has class~$\ft_{\M}$.  Hence we make the following definition:

\begin{definition}
We let~$\Rep_\R(\ft_{\M})$ denote the full abelian subcategory of $\Rep_\R(\M)$ consisting of all representations of class~$\ft$.  We call~$\Rep_\R(\ft_{\M})$ an \emph{endo-factor} of~$\Rep_\R(\M)$.
\end{definition}

\begin{corollary}\label{Corollarythetaisotypic}
The representation~$\P(\ft_{\M})$ is a compact generator of~$\Rep_{\R}(\ft_{\M})$.  In particular, if $\pi$ is of class $\ft_{\M}$, then 
\[\sum_{\phi\in\Hom_{\R[\G]}(\P(\ft_{\M}),\pi)}\mathrm{Im}(\phi)=\pi.\]
Moreover, a smooth representation has \emph{class} $\ft_{\M}$ if and only if every~$m$-semisimple~$\R$-character~$\theta$ for~$\M$ such that~$\Hom_{\R[\H^1]}(\theta,\pi)\neq 0$ has endo-parameter~$\ft_{\M}$.
\end{corollary}
\begin{proof}
The functor~$\X\mapsto \Hom_{\R[\M]}(\P(\ft_{\M}),\X)$ is faithful by Propositions \ref{EssGisoinductions} and \ref{exhaustionprop}, and~$\P(\ft_{\M})$ is compact as it is finitely generated and projective.
\end{proof}

We now prove our first main theorem \emph{endo-splitting} the category~$\Rep_\R(\G)$:

\begin{theorem}
\label{endosplit}
Let~$\R$ be a~$\mathbb{Z}[\mu_{p^\infty},1/p]$-algebra and~$\G$ be as in~\S\ref{subsecClassicalGroups}.
\begin{enumerate}
\item Let~$\M$ be a Levi subgroup of~$\G$.  We have a decomposition of categories~\[\Rep_\R(\M)=\prod_{\ft}\Rep_\R(\ft_{\M})\] where the product is taken over all endo-parameters for $\M$, and the representation~$\P(\ft_{\M})$ is a finitely generated projective generator of~$\Rep_\R(\ft_{\M})$. 
\item Parabolic induction and restriction are compatible with these decompositions: let~$\P$ be a parabolic subgroup of~$\G$ with Levi decomposition~$\P=\M\N$,~$\ft$ be an endo-parameter for~$\G$, and~$\ft_\M$ an endoparameter for~$\M$.  Then~
 \begin{align*}
 i^\G_\P:\Rep_\R(\ft_\M)&\rightarrow \Rep_\R(i_{\M}^\G(\ft_{\M})),\\
 r^\G_{\P}:\Rep_\R(\ft)&\rightarrow \Rep_\R(r_{\M}^\G(\ft)),
 \end{align*}
 where~$\Rep_\R(r_{\M}^\G(\ft))$ denotes the product~$\prod_{\ft_\M\in r_{\M}^\G(\ft)}\Rep_\R(\ft_\M)$.
\end{enumerate}
\end{theorem}

\begin{proof}
\begin{enumerate}
\item With our preparation on endo-parameters in hand, the proof of the decomposition proceeds in the same way as the proof of the decomposition by depth, cf.~\cite[Appendice A.1]{Dat09}.  Let~$\pi$ be a smooth~$\R$-representation of~$\M$, and set
\[\pi_{\ft_{\M}}:=\sum_{\phi\in\Hom_{\R[\M]}(\P(\ft_{\M}),\pi)}\mathrm{Im}(\phi).\]
Then,~$\pi=\sum_{\ft_{\M}}\pi_{\ft_{\M}}$ by Proposition \ref{exhaustionprop} and Corollary \ref{Corollarythetaisotypic}.  
%
As the representations~$\P(\ft_{\M})$ are finitely generated and projective and for~$\ft_{\M}'\neq \ft_{\M}$ we have~$\Hom_{\R[\G]}(\P(\ft_{\M}),\P(\ft'_{\M}))=0$, the sum is direct~$\pi=\bigoplus \pi_{\ft_{\M}}$.
\item Let~$\ft_\M$ be an endo-parameter in~$\M$.  Let~$\{(\H^1_{\M,i},\theta_{\M,i}): 1\leqslant i\leqslant r\}$ be a set of representatives of the~c-$\M$-conjugacy classes of $m$-semisimple~$\R$-characters of endo-parameter~$\ft_\M$.  Let~$(\J^1_{\M,i},\eta_{\M,i})$ be a Heisenberg~$\R$-representation of~$\M$ containing~$(\H^1_{\M,i},\theta_{\M,i})$.   For each $i$, as in Theorem \ref{theoremHcovers}, we can write
\[\ind_{\J^1_i}^{\G}(\eta_i)\simeq \ind_{\P}^{\G}(\ind_{\J^1(\beta,\Lambda)_\M}^{\M}(\eta_{\M,i}))\]
where the~$\eta_i$ are Heisenberg~$\R$-representations of semisimple~$\R$-characters for~$\G$ with endo-parameter~$i_{\M}^{\G}(\mathfrak{t}_{\M})$.  By exactness of~parabolic induction, it thus takes~$\Rep_{\R}(\mathfrak{t}_\M)$ to $\Rep_\R(i_{\M}^\G(\ft_{\M}))$.

By Frobenius reciprocity, we obtain the statement for parabolic restriction~$r^\G_\P$:  Suppose, for contradiction, that~$r^\G_\P(\P(\ft))(\ft_\M)\neq 0$ for some~$\ft_\M\not\in r_\M^\G(\ft)$.  Then\[\Hom_{\M}(r^\G_\P(\P(\ft)),\cW)\neq 0\] for a representation~$\cW$ of class~$\ft_\M$.  
By Frobenius reciprocity,~$\Hom_{\M}(\P(\ft),i^\G_\P(\cW))\neq 0$; hence~$i^\G_\P(\cW)$ has endo-parameter~$\ft$ and hence~$\cW$ has endo-parameter in~$r^{\G}_{\M}(\ft)$ which is absurd.
\end{enumerate}
\end{proof}

We can further provide another description of our finitely generated projective generators using parabolic induction.  We do this just for a single general linear or classical group~$\G$ to simplify notation:

\begin{lemma} 
Suppose $\ft$ is an endo-parameter for~$\G$ and choose $\M\in[\M_c(\ft)]$, the conjugacy class of Levi subgroups associated to $\ft$. Let~$\P$ be any parabolic subgroup of~$\G$ with Levi factor~$\M$.   Then we have an isomorphism of finitely generated projective generators of~$\Rep_{\R}(\ft)$
\[\P(\ft)\simeq \ind_{\P}^\G (\P(\ft_\M)).\] 
\end{lemma}

\begin{proof}
Choose~$\theta_{\M}\in\Cc_h(\beta_\M,\Upsilon)$ of endo-parameter~$\ft_{\M}$ and~$\mathrm{Vert}_{\M,\beta}$ denotes a set of representatives for the~$\M_{\beta}$-classes of vertices in a fixed chamber in~$\mathcal{B}(\mathbb{M}_{\beta},\F)$, and
\begin{enumerate}
\item for~$\Lambda\in\mathrm{Vert}_{\M,\beta}$ we write~$\theta_{\Lambda}=\tau_{\Upsilon,\Lambda,\beta_{\M}}(\theta)$, and~$\eta_{\Lambda}$ the unique (up to isomorphism) Heisenberg~$\R_{0,r}$-representation of~$\J_\Lambda^1=\J^1(\beta_{\M},\Lambda)$ associated to~$\theta_\Lambda$;
\item and $\P(\theta_\M)=\bigoplus_{\Lambda\in\mathrm{Vert}_{\M,\beta}}\ind_{\J_\Lambda^1}^\M(\eta_\Lambda)$.
\end{enumerate}
Then~$\P(\ft_{\M})\simeq \P(\theta_{\M})$.  And it follows, from Lemma \ref{EPsupportmap} and Theorem \ref{theoremHcovers} and exactness of induction, that~$\P(\theta_{\M})\simeq \P(\theta)$.
\end{proof}

\section{Beta extensions and types for Bernstein blocks}\label{secbetaextsandtypes}

Let~$\mathfrak{t}$ be an endo-parameter for~$\G$, choose~$\theta\in\Cc_h(\beta,\Upsilon)$ a semisimple character of endo-parameter~$\ft$, and set 
~$\P(\ft):=\P(\theta)=\bigoplus_{\Lambda\in\mathrm{Vert}_{\beta}}\ind_{\J_\Lambda^1}^\G(\eta_\Lambda)$ be a finitely generated projective generator of~$\Rep_{\mathbb{Z}[1/p,\mu_{p^r}]}(\ft)$ as constructed in Lemma \ref{lemmaniceprogchoice}.  

In this section, we collect results which will allow us to decompose~$\P(\mathfrak{t})\otimes \K$, where~$\K$ is an algebraically closed field of characteristic~$\ell\neq p$.  Indeed, it follows from a simple cohomology calculation (cf.~\cite[5.2.4]{BK93}, \cite[Theorem 4.1]{St08}) in characteristic zero, and either by an analogous argument or a simple reduction modulo~$\ell$ argument for positive characteristic, that~$\eta_{\Lambda}\otimes \K$ extends to an irreducible~$\K$-representation of~$\J_\Lambda=\J(\beta,\Lambda)$.   If we now choose  extensions~$\kappa_{\Lambda}$ of~$\eta_{\Lambda}\otimes\K$, then we can write
 \[\P(\ft)\otimes \K\simeq \bigoplus_{\Lambda\in\mathrm{Vert}_{\beta}} \ind_{\J_\Lambda}^{\G}(\kappa_\Lambda\otimes \ind_{\J_\Lambda^1}^{\J_\Lambda}(1)),\]
and decompose this further using finite group theory and the decomposition of the group algebra~$\K[\J_{\Lambda}/\J_\Lambda^1]$ into blocks, or the finer decomposition into projective indecomposables.  For the construction of types for Bernstein blocks it is useful to choose extensions with strong intertwining properties, this leads to the notion of ``beta extensions''; and for questions related to understanding when the $\Lambda$-th and $\Lambda'$-th component share an irreducible subquotient over~$\K$, we would like choose beta extensions~$\kappa_\Lambda$ and~$\kappa_{\Lambda'}$ ``compatibly''. 

\subsection{Beta extensions}
We suppose that~$\R$ is an algebraically closed field of characteristic~$\ell\neq p$.  Let~$\theta$ be an~$m$-semisimple~$\R$-character for~$\G$, and choose a hermitian form~$h$ and semisimple stratum~$[\Lambda,\beta]$ so that~$\theta\in\mathcal{C}_h(\beta,\Lambda)$.  Let~$\eta$ be a Heisenberg~$\R$-representation for~$\theta$.

If~$\Lambda$ is a vertex in~$\mathcal{B}(\mathbb{G}_\beta,\F)$, \emph{beta extensions} of~$\eta$ are defined as in~\cite[Theorem 4.1]{St08},~\cite[Proposition 6.1(i)]{SkodlerackCuspQuart} as representations which extend a certain canonical representation on a pro-$p$-Sylow subgroup.  We write~$\bext(\Lambda):=\{\text{beta extensions of~$\eta$ to~$\J_{\Lambda}:=\J(\beta,\Lambda)$}\}$.  Any two beta extensions of~$\eta$ differ by a character of~$\G_{\beta,\Lambda}/\G^1_{\beta,\Lambda}$ which is trivial on the subgroup generated by all its unipotent subgroups, and the set~$\bext(\Lambda)$ does not depend on the choice of~$\beta$, or of hermitian form~$h$, or the semisimple stratum~$[\Lambda,\beta]$.

Suppose now~$\Lambda$ is not necessarily a vertex.  Let~$\overline{\Lambda}$ denote the facet in $\mathcal{B}(\mathbb{G}_\beta,\F)$, under the strong simplicial structure, containing~$\Lambda$.  As in {\cite{technicalpaper}}, we introduce
\[\J^{\st}(\beta,\Lambda)=\G_{\beta,\overline{\Lambda}}\J^1(\beta,\Lambda).\]
The set of beta extensions of~$\eta$ to~$\J(\beta,\overline{\Lambda})$ is now defined after choosing a vertex~$\Lambda'$ in the closure of the facet of~$\overline{\Lambda}$, by a compatibility condition {\cite[Definition~7.4]{technicalpaper}}, and we write~$\bext(\Lambda,\Lambda')$ for the set of beta extensions of~$\eta$, compatible with the vertex~$\Lambda'$.  When~$\Lambda$ is a vertex, the only possible choice is~$\Lambda'=\Lambda$, and~$\bext(\Lambda,\Lambda)=\bext(\Lambda)$.

As~$\eta$ is determined (up to isomorphism) by~$\theta$, and its choice does not affect the isomorphism classes of beta extensions, we often refer to beta extensions as beta extensions of~$\theta$.

\begin{lemma}[{\cite[Lemma 6.5]{technicalpaper}}]
\label{positivedepthparaindinfinitegroup} Let~$\theta$ be an m-semisimple character for~$\G$.  Choose~$[\Lambda,\beta]$ a semisimple stratum for~$(\G,h)$ so that~$\theta\in\mathcal{C}_h(\beta,\Lambda)$, and let~$\kappa\in\bext(\Lambda)$.  Let~$\mathtt{P}$ be a parabolic subgroup of~$\G=\J_{\Lambda}^{\st}/\J^1_{\Lambda}$ corresponding to a parahoric subgroup~$\P^{\st}(\Upsilon_{\E})$ of~$\P^{\st}(\Lambda_\E)$, and self-dual~$\mathfrak{o}_{\E}$-lattice sequence~$\Upsilon_\E$, and set~$\Gf=\J_{\Upsilon}^{\st}/\J^1_{\Upsilon}$.  Then, for any~$\K$-representation~$\rho$ of~$\Gf$,
\[\ind_{\J^{\st}_{\Upsilon}}^{\G}(\kappa_{\Upsilon}\otimes \rho)\simeq \ind_{\J^{\st}_{\Lambda}}^{\G}(\kappa\otimes \ind_{\mathtt{P}}^{\Gf}(\rho)). \]
where~$\kappa_{\Upsilon}$ is compatible with~$\kappa$ as in \cite[Definition 7.4]{technicalpaper}.
\end{lemma}

\subsection{Types for Bernstein blocks}\label{typesoverC}
As in Section \ref{projmoduleschar0}, write~$\mathfrak{B}_{\K}(\G)$ for the set of inertial classes of supercuspidal supports for~$\G$.    
\begin{definition}
Suppose~$\K$ is an algebraically closed field of characteristic zero.   Let~$\mathfrak{s}\in\mathfrak{B}_{\K}(\G)$.  A pair~$(\U,\Sigma)$, with $\U$ a compact open subgroup of~$\G$, and~$\Sigma$ an irreducible representation of~$\U$, is called an~\emph{$\mathfrak{s}$-type} if~$\ind_{\U}^{\G}(\Sigma)$ is a (finitely generated projective) generator of~$\Rep_{\K}(\mathfrak{s})$. 
\end{definition}

For classical~$p$-adic groups,~$\GL_m(\D)$, and quarternionic forms of classical groups, we have a construction of types for Bernstein blocks of Miyauchi--Stevens \cite{MiSt}, S\'echerre--Stevens \cite{SecherreStevensVI}, Skodlerack-Ye \cite{SkodlerackYe}, and in depth zero for an arbitrary connected reductive group, Morris in \cite{Morris} has constructed types for Bernstein blocks.  

Let~$\mathfrak{s}\in\mathfrak{B}_{\K}(\G)$ with representative~$(\M,\rho)$.  This determines a supercuspidal inertial class~$\mathfrak{s}_{\M}\in \mathfrak{B}_{\K}(\M)$ with representative~$(\M,\rho)$. 

A Levi subgroup~$\M$ in~$\G$ decomposes as a product~$\M=\prod \M_i$ of (inner forms) of general linear groups, or of (inner forms) of general linear groups and (an inner form) of a classical group.  We define an (m-semisimple) semisimple stratum in~$\M$, to be a direct sum of (m-semisimple) semisimple stratum in the corresponding~$\M_i$, and define the groups, characters, Heisenberg representations, and beta extensions, associated to stratum in~$\M$ by taking the appropriate product or tensor product over~$i$. 

 The construction of cuspidal representations has been extended to all algebraically closed fields of characteristic~$\ell\neq p$ and, including the depth zero case for any reductive~$p$-adic group, we have:

\begin{theorem}[{Depth zero \cite{Morris, Vigbarcelona}, Classical groups \cite{St08, RKSS}, $\GL_m(\D)$ \cite{SecherreStevensIV,MinSec}, inner forms of classical \cite{SkodlerackCuspQuart}, \cite[Theorem 6.12]{technicalpaper}}]
Let~$\K$ be an algebraically closed field of characteristic~$\ell\neq p$, and~$\rho$ be an irreducible cuspidal~$\K$-representation of~$\M$.  There exist:
\begin{enumerate}
\item in positive depth, 
a beta extension~$\kappa$ to~$\J_\M:=\J(\beta,\Lambda)$ of a semisimple character for an $m$-semisimple stratum~$[\Lambda,n,0,\beta]$; or
\item in depth zero, the trivial character~$\kappa$ of~$\J_\M:=\M_{\Lambda_E}^+$ of~$\M$, where~$\Lambda_{\E}$ is a vertex, in which case write~$\E=\F$ and~$\J_\M^1$ for the pro-unipotent radical of~$\J_\M$; and
%
\item 
in both cases, an irreducible cuspidal representation~$\sigma_{\M}$ of~$\M_{\Lambda_{\E}}/\M_{\Lambda_{\E}}^1$, such that setting~$\Sigma_{\M}=\kappa_{\M}\otimes\sigma_{\M}$, there is an irreducible representation~$\widetilde{\Sigma}_\M$ of~$\mathbf{J}_{\M}=\N_{\M}(\J(\beta,\Lambda))$ whose restriction contains~$\Sigma_{\M}$;
\end{enumerate}
such that, $\rho \simeq \ind_{\mathbf{J}_{\M}}^{\M}(\widetilde{\Sigma}_\M)$.  Moreover, if~$\ell=0$, then~$(\J(\beta,\Lambda),\Sigma_{\M})$ is an~$\mathfrak{s}_{\M}$-type for~$\mathfrak{s}_{\M}=[\M,\rho]_{\M}$.
\end{theorem}

When the characteristic~$\ell$ of~$\K$ is zero, the construction of covers allows one to construct an $\mathfrak{s}$-type, as we now explain again working for the construction in the broader setting of algebraically closed fields of characteristic~$\ell\neq p$ (though when~$\ell\neq 0$ one no longer obtains a generator of a direct factor category in general as the structure of the~$\ell$-blocks is in general much more complicated):  Let~$\P=\M\N$ be a parabolic subgroup of~$\G$.  Let~$\rho$ be an irreducible cuspidal~$\K$-representation of~$\M$.  

In positive depth, we can choose a semisimple stratum~$[\Lambda,\beta]$ for~$(\G,h)$ and a semisimple character~$\theta\in\Cc_h(\beta,\Lambda)$ such that the decomposition of~$\V$ associated to~$\M$ is properly subordinate to~$[\Lambda,\beta]$, and~$\theta\mid_{\H^1(\beta,\Lambda)\cap \M}$ is an $m$-semisimple character contained in~$\rho$.  We set
\[\J_\P=\J_{\P}(\beta,\Lambda)=\H^1(\beta,\Lambda)(\J^{\st}(\beta,\Lambda)\cap \P).\]
This agrees with the~$\J_{\P}$ group considered in \cite{MiSt}, \cite{SkodlerackYe}, by \cite[Propositions~C.10,~C.11]{technicalpaper}.   

In depth zero, following Morris \cite[155]{Morris}, we can choose a facet in~$\mathcal{B}(\G,\F)$, such that setting~$\J_{\P}$ to be the full stabiliser of the facet, then for a cuspidal~$\K$-representation of~$\J_{\P}/\J_{\P}^1$, the pair~$(\J_{\P},\sigma_{\P})$ defines an~$\mathfrak{s}$-type.

Let~$\kappa$ be a beta extension to~$ \J^{\st}(\beta,\Lambda)$, which we interpret to be the trivial representation in depth zero.  We form the natural representation~$\kappa_{\P}$ of~$\J_{\P}$ on the space of~$(\J^{\st}(\beta,\Lambda)\cap \U)$-fixed vectors in~$\kappa$.  Then, in depth zero setting~$\eta_{\P}$ to be the trivial representation of~$\J_{\P}^1$, then~$\kappa_{\P}$ extends~$\eta_\P$ and~$\ind_{\J_{\P}}^{\J^{\st}(\beta,\Lambda)}(\kappa_{\P})\simeq \kappa$.

\begin{theorem}[{Depth zero \cite{Morris}, Classical groups \cite{MiSt}, $\GL_m(\D)$ \cite{SecherreStevensVI}, inner forms of classical \cite{SkodlerackYe}, \cite[Theorem 6.13]{technicalpaper}}]\label{Gcoverstheorem}
  Under the above notation, writing~$\Sigma_{\P}=\kappa_{\P}\otimes\sigma_{\P}$, we have 
\begin{enumerate}
\item $(\J_{\P},\Sigma_{\P})$ is a~$\G$-cover of~$(\J_{\M},\Sigma_{\M})$ relative to~$\P$. 
\item If~$\ell=0$, then $(\J_{\P},\Sigma_{\P})$ is a~$\mathfrak{s}$-type. 
\end{enumerate}\end{theorem}

\subsection{Supercuspidal support of types}\label{scsupporttypes}
We use the notation of Section \ref{parahorics1} in depth zero: 

\begin{proposition}\label{depthzeromaxtotypes}
Let~$\K$ be an algebraically closed field of characteristic~$\ell\neq p$.  Suppose we have a pair~$(\J,\kappa\otimes\pi)$ consisting of an irreducible representation~$\kappa\otimes\pi$ of a compact open subgroup~$\J$ of~$\G$ constructed in the following fashion:
\begin{enumerate}
\item (Depth zero case) $\J=\G_x^+$ where~$x$ is a vertex in the Bruhat--Tits building of~$\G$,~$\kappa$ is trivial, and~$\pi$ is an irreducible representation of~$\Gf_x^+=\G_x^+/\G_x^1$. 
\item (Positive depth case) there is an~$m$-semisimple character~$\theta$ for~$\G$, a semisimple stratum~$[\Lambda,\beta]$ for~$(\G,h)$ so that~$\theta\in\mathcal{C}_h(\beta,\Lambda)$, and~$\kappa$ is a~$\beta$ extension of~$\theta$, and~$\pi$ is an irreducible~$\K$-representation of~$\J^{\st}(\beta,\Lambda)/\J^1(\beta,\Lambda)$ over an algebraically closed field of characteristic zero.  
\end{enumerate}
Then there exists a pair~$(\J_\P',\Sigma_\P)$ constructed in Theorem \ref{Gcoverstheorem}, which in particular is an~$\mathfrak{s}$-type if~$\ell=0$, such that~$\ind_{\J}^{\G}(\kappa\otimes \pi)$ is a subquotient of~$\ind_{\J_\P'}^{\G}(\Sigma_{\P})$.   
\end{proposition}

\begin{proof}
This follows from the construction of~$\mathfrak{s}$-types explained in the last section, together with Lemma \ref{fgLemmascsupport}, and Lemma \ref{positivedepthparaindinfinitegroup} in positive depth.
\end{proof}

\begin{definition}
In the notation of Proposition \ref{depthzeromaxtotypes}, we say that the pair~$(\J_\P,\Sigma_\P)$ is in the \emph{supercuspidal support} of~$(\J,\kappa\otimes\pi)$.
\end{definition}

\subsection{$\mathfrak{s}$-types and Bernstein projectives}
We continue in the setting of an algebraically closed field~$\K$ of characteristic not equal to~$p$.  We have the parabolically induced representation
\[\Pi_{\mathrm{BK}}:=\ind_{\M,\P}^{\G}(\ind_{\J_{\M}}^{\M}(\Sigma_{\M}))\simeq \ind_{\J_\P}^{\G}(\Sigma_\P),\]
the last isomorphism by \cite[Th\'eor\`eme 2]{Blondel}.  If~$\ell=0$, this is a finitely generated projective generator of~$\Rep_{\K}(\mathfrak{s})$ for the inertial class of~$(\M,\pi_{\M})$ where~$\pi_{\M}\simeq \ind_{\mathbf{J}_\M}^{\M}(\widetilde{\Sigma_{\M}})$ is an irreducible supercuspidal~$\K$-representation of~$\M$ and~$\mathbf{J}_\M=\N_{\M}(\J_{\M})$.    We can also consider
\[\Pi_{\mathrm{Bern}}:=i_{\M,\P}^{\G}(\ind_{\M^\circ}^{\M}(\pi_{\M}))\simeq i_{\M,\P}^{\G}(\pi_{\M}\otimes\chi_{\mathrm{univ}}),\]
which if~$\ell=0$ is Bernstein's finitely generated projective generator of~$\Rep_{\K}(\mathfrak{s})$.  

\begin{lemma}[{\cite[Appendix B]{SavinBakic} when~$\ell=0$}]
\label{projcomplemma}
The induced representation~$\ind_{\J_{\M}}^{\M^{\circ}}(\Sigma_{\M})$ is an irreducible summand of~$\pi_{\M}\mid_{\M^{\circ}}$.
\end{lemma}

\begin{proof}
By Mackey Theory, the restriction of~$\pi_{\M}$ to~$\M^{\circ}$ decomposes into a finite sum containing the direct sum
\[\bigoplus_{\mathbf{J}_{\M}\M^\circ\backslash \M} (\ind_{\J_{\M}^j}^{\M^{\circ}}(\Sigma_{\M}^j))\hookrightarrow \pi_{\M}\mid_{\M^{\circ}},\]
as~$\widetilde{\Sigma}_{\M}\mid_{\J_{\M}}$ contains~$\Sigma_{\M}$, and~$\mathbf{J}_{\M}^j\cap \M^{\circ}=\J_{\M}^j$, because:
\begin{enumerate}
\item for~$\G$ classical, one can reduce to a single~$\GL_r(\F)$ block of a Levi subgroup, where we have~$\E^\times \J_i\cap \GL_r(\F)^{\circ}=(\E^\times \cap \GL_r(\F)^{\circ})\J_i=\J_i$, as~$\GL_r(\F)^{\circ}=\{g\in\GL_r(\F):\det(g)\in\mathfrak{o}_\F^\times\}$.
\item for an inner form of a classical or general linear group, the equality follows similarly with the reduced norm replacing the determinant.
\item for~depth zero representations,~$\mathbf{J_{\M}}$ is contained in~$\mathrm{N}_{\M}(\M_x^+)$ for~$x$ a vertex in the Bruhat--Tits building of~$\M$, and in this case~$\mathrm{N}_{\M}(\M_x^+)\cap \M^\circ=\mathrm{N}_{\M^{\circ}}(\M_x^+)=\M_x^+$ from results of Bruhat--Tits referenced in Section \ref{parahorics1}.
\end{enumerate}
It remains to show that~$\ind_{\J_{\M}}^{\M^{\circ}}(\Sigma_{\M})$ is irreducible, which follows from Vign\'eras' simple criterion for irreducibility \cite[Lemma 4.2]{Vigbarcelona}, analogous to the construction of cuspidal representations -- cf.~\cite[Theorem 12.1]{RKSS}.
\end{proof}

Letting~$\rho_{\M^{\circ}}=\ind_{\J_{\M}}^{\M^{\circ}}(\Sigma_{\M})$, we can consider the finitely generated representation
\[\Pi':=i_{\M,\P}^{\G}(\ind_{\M^\circ}^{\M}(\rho_{\M^{\circ}})),\]
by exactness of induction it is a summand of~$\Pi_{\mathrm{Bern}}$.

\begin{proposition}[{\cite[Appendix B]{SavinBakic} when~$\ell=0$}]
\label{BerBKcomparison} The finitely generated projective representation~$\Pi_{\mathrm{BK}}$  is isomorphic to a summand of~$\Pi'$.  Hence~$\Pi_{\mathrm{BK}}$ is a summand of~$\Pi_{\mathrm{Bern}}$.
\end{proposition}

\begin{proof}
As the modulus character of~$\delta_\P$ is an unramified character of~$\M$, the first statement follows from the lemma and transitivity of induction.  The second statement then follows from Lemma \ref{projcomplemma}.
\end{proof}

\section{Block decompositions via type theory}\label{secblocks}

Let~$\mathfrak{t}$ be an endo-parameter for~$\G$.   Then we have a direct factor~$\Rep_{\overline{\mathbb{Z}}[1/p]}(\mathfrak{t})$ subcategory of~$\Rep_{\overline{\mathbb{Z}}[1/p]}(\G)$, with finitely generated projective generator~$\P(\ft)$ as in Lemma \ref{lemmaniceprogchoice}.  
Recall, with~$\J^1_\Lambda=\J^1(\beta,\Lambda)$,~$\theta\in\mathcal{C}_h(\beta,\Lambda)$ of endo-parameter~$\ft$, for some~$\beta\in\mathrm{Lie}(\G)$, we choose~$\P(\ft)$ of the form
 \[\P(\ft)=\P(\theta)=\bigoplus_{\Lambda\in\mathrm{Vert}_\beta}\ind_{\J^1_\Lambda}^{\G}(\eta_\Lambda).\]

We also allow the trivial endo-parameter~$\mathfrak{t}=\mathbf{0}_\G$ of any reductive~$p$-adic group~$\G$, corresponding to the depth zero subcategory of~$\Rep_{\overline{\mathbb{Z}}[1/p]}(\G)$, with finitely generated projective generator a finite sum
\[\P(\mathbf{0}_\G)=\bigoplus_{\Lambda\in \mathrm{Vert_0}} \ind_{\G_{\Lambda}^1}^{\G}(1)\]
where the sum is over a set of representatives for the~$\G$-representatives of the vertices in the Bruhat-Tits building of~$\G$, which we can suppose all lie in the same chamber.  

Let~$r$ be sufficiently large for~$\mathfrak{t}$ so that all the~$\eta_\Lambda$ are defined over~$\R_{0,r}=\mathbb{Z}[1/p,\mu_{p^r}]$ (or in the depth zero case, we let~$r=0$).  In particular, the idempotent cutting out~$\Rep_{\overline{\mathbb{Z}}[1/p]}(\mathfrak{t})$ is defined over~$\mathbb{Z}[1/p,\mu_{p^r}]$.   

In this section~$\R$ denotes a commutative~$\mathbb{Z}[1/p,\mu_{p^r}]$-algebra, and we study the~$\R$-block decomposition of~$\Rep_{\R}(\mathfrak{t})$.  We further suppose that $\R$ is a domain, in particular the field of fractions~$\K'$ of~$\R$ is flat over~$\R$, and we fix an algebraic closure~$\K$ of~$\K'$.  The key examples we consider include:~$\R=\mathbb{Z}[1/p,\mu_{p^r}]$,~$\R=\overline{\mathbb{Z}}[1/p]$,~$\R=\W(\Fl)$,~$\R=\Zl$, and~$\R=\Fl$ for~$\ell\neq p$. 

\subsection{The fine and coarse graphs}
Let~$\Lambda\in\mathrm{Vert}_\beta$.  We write~$\J_{\Lambda}:=\J(\beta,\Lambda)$, 
and choose a decomposition~$\ind_{\J^1_\Lambda}^{\J_\Lambda}(\eta_\Lambda\otimes\R)\simeq \bigoplus_{i\in\I_{\Lambda}} \overline{P}_{\Lambda,i}$ into projective indecomposable~$\R$-representations~$\overline{P}_{\Lambda,i}$, where~$\I_\Lambda$ is an (finite) index set for the decomposition.  Set\[P_{\Lambda,i}:=\ind_{\J_i}^{\G}( \overline{P}_{\Lambda,i}).\] By exactness of compact induction
\[P(\ft)\otimes \R\simeq \bigoplus_{\substack{\Lambda\in\mathrm{Vert}_\beta, i\in\I_\Lambda}} P_{\Lambda,i}\] defines a decomposition of~$P(\ft)\otimes \R$ into finitely generated projective~$\R$-representations~$P_{\Lambda,i}$.

\begin{remark}\label{remarkotherdecomps}
\begin{enumerate}
\item If~$\R$ is Artinian then, setting~$\H$ a compact open normal subgroup of~$\J_i$ such that~$\eta_{\Lambda}\vert_{\H}$ is trivial, the group ring~$\R[\J_i/\H]$ is Artinian.  This guarantees the uniqueness up to isomorphism of the summands~$\overline{P}_{\Lambda,i}$ in the decomposition of~$\ind_{\J^1_\Lambda}^{\J_\Lambda}(\eta_\Lambda\otimes\R)$ (by the Krull--Schmidt theorem).  While similar uniqueness statements hold for various local rings (see for example \cite{ReinerKS}), they fail in general \cite{ReinerFailure}, so depending on the context this may be a choice of decomposition we are making.
\item \label{remarkotherdecompsii}For our applications, the decomposition~$\ind_{\J^1_\Lambda}^{\J_\Lambda}(\eta_\Lambda\otimes\R)\simeq \bigoplus_i\overline{P}_{\Lambda,i}$ into projective indecomposable representations can be replaced by any decomposition~$\ind_{\J^1_\Lambda}^{\J_\Lambda}(\eta_\Lambda\otimes\R)\simeq\bigoplus \Pi_{\Lambda,i}$ satisfying, for all~$i$, that there are no non-trivial central idempotents of~$\End_{\R[\J_\Lambda]}(\Pi_{\Lambda,i})$.   We can use any such decomposition to parametrize the~$\R$-blocks following the methods in this section and construct finitely generated projective generators of the~$\R$-blocks.  Different choices of decomposition can lead to different decompositions of the finitely generated projective generators we construct as direct sums of finitely generated projective representations.
\end{enumerate}
\end{remark}

Choose a beta extension~$\kappa'_\Lambda$ of~$\eta_\Lambda\otimes \overline{\mathbb{Q}}$, then it is defined over~$\mathrm{S}[1/p]$ where~$\mathrm{S}$ is the ring of integers of a number field, by \cite{CurtisReiner}. 

\begin{lemma}\label{lemmakappafunctor}
Suppose~$\R$ is an~$\mathrm{S}[1/p]$-algebra, and set~$\kappa_\Lambda=\kappa_\Lambda'\otimes \R$.  Let~$ \ind_{\J_{\Lambda}^1}^{\J_{\Lambda}}(1)=\bigoplus_{i\in\I_{\Lambda}}  \Q_{\Lambda,i}$ be a decomposition of~$\ind_{\J^1_{\Lambda}}^{\J_{\Lambda}}(1)$ into projective indecomposable representations of~$\J_{\Lambda}/\J_{\Lambda}^1$.  Then
\[\ind_{\J^1_{\Lambda}}^{\J_{\Lambda}}(\eta_{\Lambda}\otimes\R)\simeq \bigoplus_{i\in\I_{\Lambda}} \kappa_{\Lambda}\otimes \Q_{\Lambda,i},\]
defines a decompostion of~$\ind_{\J^1_{\Lambda}}^{\J_{\Lambda}}(\eta_{\Lambda}\otimes\R)$ into projective indecomposable representations. 
\end{lemma}

\begin{proof}
We have
\[\ind_{\J^1_{\Lambda}}^{\J_{\Lambda}}(\eta_{\Lambda}\otimes\R)\simeq\kappa_{\Lambda}\otimes \ind_{\J_{\Lambda}^1}^{\J_{\Lambda}}(1)\simeq \bigoplus_{i\in\I_{\Lambda}} \kappa_{\Lambda}\otimes \Q_{\Lambda,i},\]
and it suffices to show that the $\kappa_{\Lambda}\otimes \Q_{\Lambda,i}$ are indecomposable.  Suppose that, for some~$i$, we have a non-trivial decomposition~$ \kappa_{\Lambda}\otimes \Q_{\Lambda,i}=\Pi_1\oplus\Pi_2$. Then applying~$\Hom_{\R[\J^1_{\Lambda}]}(\kappa_{\Lambda},-)$, as~$\Q_{\Lambda,i}\mid_{\J^1}$ is trivial and~$\End_{\R[\J^1]}(\eta_{\Lambda})\simeq\R$ (Corollary \ref{etaintertwining}), we have
\begin{align*}
\Q_{\Lambda,i}&\xrightarrow{\sim}\Hom_{\R[\J^1_{\Lambda}]}(\kappa_{\Lambda}, \kappa_{\Lambda}\otimes \Q_{\Lambda,i})\\
w&\mapsto \alpha_{w}: v\mapsto v\otimes w,
\end{align*}
where~$\J_{\Lambda}/\J_{\Lambda}^1$ acts on~$\Hom_{\R[\J^1_{\Lambda}]}(\kappa_{\Lambda},\kappa_{\Lambda}\otimes \Q_{\Lambda,i})$ via~$j\cdot f = j\circ f \circ j^{-1}$.  And we find
\[\Q_{\Lambda,i}\simeq \Hom_{\R[\J^1_{\Lambda}]}(\kappa_{\Lambda}, \Pi_1\oplus\Pi_2)\simeq  \Hom_{\R[\J^1_{\Lambda}]}(\kappa_{\Lambda},  \Pi_1)\oplus\Hom_{\R[\J^1_{\Lambda}]}(\kappa_{\Lambda}, \Pi_2)\] 
as representations of~$\J_{\Lambda}/\J_{\Lambda}^1$, and each hom-space is non-zero as~$\Pi_1$ and~$\Pi_2$ are~$\eta_{\Lambda}$-isotypic.
\end{proof}

\begin{definition} 
\begin{enumerate}
\item We define the \emph{fine~$(\mathfrak{t},\R)$-graph}~$\mathcal{G}_\ft=(V,E)$ for~$\mathfrak{t}$ over~$\R$ by its vertex set
\[V=\{P_{\Lambda,i}:\Lambda\in\mathrm{Vert}_{\beta}, i\in \I_{\Lambda}\}\]
 and drawing an edge between~$P_{\Lambda,i}$ and~$P_{\Lambda,j}$ if either~\[\Hom_{\R[\G]}(P_{\Lambda,i}, P_{\Lambda,j})\neq 0\text{ or }\Hom_{\R[\G]}(P_{\Lambda,j}, P_{\Lambda,i})\neq 0.\] 
\item We define the \emph{coarse~$(\mathfrak{t},\R)$-graph}~$\widetilde{\mathcal{G}}_\ft=(V,\widetilde{E})$ for~$\mathfrak{t}$ over~$\R$ by its vertex set\[V=\{P_{\Lambda,i}:\Lambda\in\mathrm{Vert}_{\beta}, i\in \I_{\Lambda}\}\] and drawing an edge between~$P_{\Lambda,i}$ and~$P_{\Lambda,j}$ if~$P_{\Lambda,i}\otimes \K$ and~$P_{\Lambda,j}\otimes \K$ have direct summands lying in the same~$\K$-block.
\end{enumerate}
\end{definition}

\begin{remark}
\begin{enumerate}
\item To avoid repetition in the projective generators we will construct from our graphs, if a representation appears (up to isomorphism) with multiplicity greater than one in the decomposition~$\ind_{\J^1_\Lambda}^{\J_\Lambda}(\eta_\Lambda\otimes\R)\simeq \bigoplus_{i\in\I_{\Lambda}} \overline{P}_{\Lambda,i}$, then we can identify the corresponding vertices in the~$(\mathfrak{t},\R)$-graphs.
\item In the setting of Lemma \ref{lemmakappafunctor}, a natural decomposition of~$\R[\J_\Lambda/\J_\Lambda^1]$, rather than choosing a decomposition of~$\R[\J_{\Lambda}/\J_{\Lambda}^1]$ into projective indecomposable modules, is the block decomposition of~$\R[\J_{\Lambda}/\J_{\Lambda}^1]$ which exists as~$\R$ is an integral domain (any decomposition of~$1$ into infinitely many orthogonal idempotents in the centre of~$\R[\J_{\Lambda}/\J_{\Lambda}^1]$ would give such a decomposition in~$\K[\J_{\Lambda}/\J_{\Lambda}^1]$ which is Noetherian, a contradiction).  
\item To obtain Lemma \ref{lemmakappafunctor} over a smaller base ring~$\mathrm{S}$, using the theory of the Weil representation (cf.,~\cite{FinKalSpi}) we expect one can construct an extension of~$\eta_{\Lambda}$ to~$\J_{\Lambda}$, then proceed with this extension or twist by a character of~$\J_{\Lambda}/\J_{\Lambda}^1$ to obtain a beta extension (in which case one would need to extend our ring $\R_{0,r}$ by these character values).
\end{enumerate}
\end{remark}

\begin{proposition}
\begin{enumerate}
\item The fine~$(\mathfrak{t},\R)$-graph~$\mathcal{G}_{\ft}$ is a subgraph of the coarse~$(\mathfrak{t},\R)$-graph~$\widetilde{\mathcal{G}}_{\ft}$.
\item Suppose~$\R$ is a characteristic zero domain.  Then the fine~$(\mathfrak{t},\R)$-graph and the coarse~$(\mathfrak{t},\R)$-graph coincide,~i.e.,~$\mathcal{G}_{\ft}=\widetilde{\mathcal{G}}_{\ft}$.
\end{enumerate}
\end{proposition}

\begin{proof}
By Mackey's decomposition, as~$\K'$ is flat over~$\R$ and~$\K$ (faithfully) flat over~$\K'$,
\begin{equation}
\label{Mackeyeq}\Hom_{\R[\G]}(P_{\Lambda,i}, P_{\Lambda',i'}) \otimes \K = \Hom_{\K[\G]}(P_{\Lambda,i} \otimes \K, P_{\Lambda',i'} \otimes \K).\end{equation}
Since~the projectives $P_{\Lambda,i}$ and $P_{\Lambda',i'}$ are torsion free, this shows that there is a nonzero map from $P_{\Lambda,i}$ to $P_{\Lambda',i'}$ if and only if there is such a map from~$P_{\Lambda,i}\otimes \K$ to~$P_{\Lambda',i'}\otimes\K$.  And we deduce the first statement.  

So suppose~$\R$ is a characteristic zero domain.  From Lemma \ref{mainlemmaKprojs} \ref{PP'hom}, we see that if~$P_{\Lambda,i}$ and $P_{\Lambda',i'}$ are joined by an edge of the coarse graph, then there is a nonzero element of $\Hom_{\K[\G]}(P_{\Lambda,i} \otimes \K,P_{\Lambda',i'} \otimes \K)$. By Equation \ref{Mackeyeq}, the latter is the same as $\Hom_{\R[\G]}(P_{\Lambda,i},P_{\Lambda',i'}) \otimes \K$, so $P_{\Lambda,i}$ and $P_{\Lambda',i'}$ are joined by an edge of the fine graph as well.
\end{proof}

\subsection{Computing the coarse graph}\label{Computingthecoarsegraph}
The advantage of the coarse graph is that using type theory we can give an explicit recipe to compute it.

\subsubsection{Characteristic zero domains}
Suppose that~$\R$ is a characteristic zero domain.  As in Section \ref{typesoverC}, we can index the Bernstein~$\K$-blocks by types, and using this parameterisation write down which blocks over~$\K$ appear in each projective~$P_{\Lambda,i}\otimes \K$.  

\begin{enumerate}
\item For each~$\Lambda\in\mathrm{Vert}_{\beta}$, we choose a beta extension~$\kappa_\Lambda$ of~$\eta_\Lambda$ over~$\K$.  (When~$\beta=0$ in depth zero, we take~$\kappa_{\Lambda}$ the trivial~$\K$-representation of~$G_{\Lambda}$).   This allows us to decompose
\[\ind_{\J^1_\Lambda}^{\J_{\Lambda}}(\eta_\Lambda)\otimes \K\simeq \kappa_\Lambda \otimes \ind_{\J^1_\Lambda}^{\J_{\Lambda}}(1\otimes\K)\simeq \bigoplus_{\pi\in \Irr_{\K}(\J_{\Lambda}/\J_{\Lambda}^1)}\kappa_\Lambda\otimes \pi^{\oplus \mathrm{dim}(\pi) };\]
and 
\[P_{\Lambda,i}\otimes \K\simeq\bigoplus_{\pi\in [\overline{P}_{\Lambda,i}\otimes \K]}\ind_{\J_{\Lambda}}^{\G}( \kappa_\Lambda\otimes \pi)^{\oplus m_{\pi}};\]
where~$m_{\pi}$ is the multiplicity of~$\pi$ in~$ \overline{P}_{\Lambda,i}\otimes \K$. 
\item For each pair~$(\J_{\Lambda},\kappa_\Lambda\otimes \pi)$, we let~$(\J_{\Lambda,\Q},\kappa_{\Lambda,\Q}\otimes\rho)$ be in its supercuspidal support (where~$\Q$ and~$\rho$ depend on~$\pi$) -- cf., Section \ref{scsupporttypes}.  Then
\[P_{\Lambda,i}\otimes \K\hookrightarrow \bigoplus_{\pi\in [\overline{P}_{\Lambda,i}\otimes \K]} \ind_{\J_{\Lambda,\Q}}^\G(\kappa_{\Lambda,\Q}\otimes \rho)^{\oplus m_{\pi}}\]
as a direct summand, and each~$(\J_{\Lambda,\Q},\kappa_{\Lambda,\Q}\otimes \rho)$ is a type for a Bernstein block over~$\K$ as recalled in Section \ref{typesoverC}.  
\item In this way each summand $\kappa_\Lambda\otimes\pi$ of $\overline{P}_{\Lambda,i} \otimes \K$ gives rise to a Bernstein block of $\Rep_\K(\G)$, and $P_{\Lambda,i}\otimes \K$ straddles the union of these blocks.
\item To compute whether there is an edge between two vertices~$P_{\Lambda,i}$ and~$P_{\Lambda',i'}$ in the coarse graph, we are reduced to computing the intertwining in~$\G$ of the types for the Bernstein components which appear in the decomposition of the projectives~$P_{\Lambda,i}\otimes \K$ and~$P_{\Lambda',i'}\otimes\K$.  And if one can choose a compatible family of beta extensions as in \cite[Definition 6.6]{technicalpaper}, then this can be reduced to depth zero.
\end{enumerate}

\subsection{Graphs and~$\R$-blocks}\label{GraphsStatements}
The main result of this section is:
\begin{theorem}\label{maintheoremblocks}
The~$\R$-blocks of~$\Rep_{\R}(\mathfrak{t})$ correspond to the connected components of the fine~$(\mathfrak{t},\R)$-graph.  \end{theorem}
More precisely,~the finitely generated projective~$\R$-representation of~$\G$ defined as the direct sum over the vertices in a connected component of the fine~$(\mathfrak{t},\R)$-graph is a finitely generated projective generator of an~$\R$-block, and running over the connected components of the fine~$(\mathfrak{t},\R)$-graph defines the~$\R$-block decomposition of the finitely generated projective generator~$P(\ft)\otimes \R$ of~$\Rep_{\R}(\mathfrak{t})$, hence defines the~$\R$-block decomposition of~$\Rep_{\R}(\mathfrak{t})$.

\begin{proof}[{Proof of Theorem \ref{maintheoremblocks}}]
Suppose that~$P_{\Lambda,i}$ lies in a unique~$\R$-block for all~$i,j$.  Then a central idempotent of~$\End_{\R[\G]}(P(\ft))$ acts by $1$ or $0$ on each~$P_{\Lambda,i}$.  If there is a non-zero map~$P_{\Lambda,i}\rightarrow P_{\Lambda',i'}$ (or $P_{\Lambda',i'}\rightarrow P_{\Lambda,i}$) for a central idempotent to commute with this map it must be $1$ on both or $0$ on both; this guarantees that the idempotents coming from connected components of the fine graph are primitive.  We are reduced to showing that~$P_{\Lambda,i}$ lies in a unique~$\R$-block.

Let us first consider the special case where~$\R=\K$ is an algebraically closed field.  We choose a beta extension~$\kappa_\Lambda$ of~$\eta_\Lambda$, so that~$P_{\Lambda,i}\simeq \ind_{\J_{\Lambda}}^{\G}(\kappa_\Lambda\otimes \zeta_{\Lambda,i})$ where~$\zeta_{\Lambda,i}$ is an indecomposable~$\K$-representation of~$\J_{\Lambda}/\J_{\Lambda}^1$ of finite length.  The ``Ext-graph'' of~$\zeta_{\Lambda,i}$ with vertices the irreducible subquotients of~$\zeta_{\Lambda,i}$ and an edge between irreducible subquotients~$\sigma$ and~$\sigma'$ if~$\mathrm{Ext}^1_{\R[\J_\Lambda]}(\sigma,\sigma')\neq 0$ is connected as~$\zeta_{\Lambda,i}$ is indecomposable.  Moreover, $\mathrm{Ext}^1_{\R[\J_{\Lambda}]}(\sigma,\sigma')$ embeds into~$\mathrm{Ext}^1_{\R[\G]}(\ind_{\J_{\Lambda}}^\G (\kappa_\Lambda\otimes \sigma),\ind_{\J_{\Lambda}}^\G (\kappa_\Lambda\otimes \sigma'))$, which when non-zero implies that~$\ind_{\J_{\Lambda}}^\G (\kappa_\Lambda\otimes \sigma)$ and~$\ind_{\J_{\Lambda}}^\G (\kappa_\Lambda\otimes \sigma')$ have an~$\R$-block in common.  Hence, working around the Ext-graph, it suffices to show that~$\ind_{\J_{\Lambda}}^\G (\kappa_\Lambda\otimes \sigma)$ is contained in a single $\K$-block where~$\sigma$ is an irreducible~$\K$-representation of~$\J_{\Lambda}/\J_{\Lambda}^1$.

By the theory of covers and Proposition \ref{depthzeromaxtotypes},~$\ind_{\J_{\Lambda}}^\G (\kappa_\Lambda\otimes \sigma)$ is subquotient of~$\ind_{\J_{\Lambda},\Q}^{\G}(\kappa_{\Lambda,\Q}\otimes\rho)$, and it suffices to show that~$\ind_{\J_{\Lambda},\Q}^{\G}(\kappa_{\Lambda,\Q}\otimes\rho)$ is in a unique~$\K$-block, where~$(\J_{\Lambda,\Q},\kappa_{\Lambda,\Q}\otimes\rho)$ is a cover of a supercuspidal type~$(\J_{\Lambda,\M},\kappa_{\Lambda,\M}\otimes\rho)$. If~$\K$ is of characteristic zero, we are done, as~$(\J_{\Lambda,\Q},\kappa_{\Lambda,\Q}\otimes\rho)$ is an~$\mathfrak{s}$-type.  In any case,~$\ind_{\J_{\Lambda,\Q}}^{\G}(\kappa_{\Lambda,\Q}\otimes\rho)$ is a summand of~$\Pi:=i_{\P}^{\G}(\ind_{\M^\circ}^{\M}(\pi_{\M}))$ where~$\pi_{\M}$ is a supercuspidal representation of~$\M$ containing~$(\J_{\Lambda,\M},\kappa_{\Lambda,\M}\otimes\rho)$ by Proposition \ref{BerBKcomparison}.  And it suffices to show that~$\Pi$ is contained in a unique~$\K$-block, which follows from Proposition \ref{genirredproposition} as~$\K[\M/\M^{\circ}]$ is an integral domain.  This completes the case when~$\R=\K$ is an algebraically closed field.

Now we go back to the general case:~$\R$ is an integral domain with field of fractions~$\K'$, and~$\K$ an algebraic closure of~$\K'$ and~$P_{\Lambda,i}=\ind_{\J_{\Lambda}}^{\G}(\overline{P}_{\Lambda,i})$.   We choose a beta extension~$\kappa_{\Lambda}$ of~$\eta_\Lambda\otimes \K$, and~decompose~$\overline{P}_{\Lambda,i}\otimes \K\simeq \bigoplus_{s} \kappa_{\Lambda}\otimes \zeta_{\Lambda,i}^s$ where the~$\zeta_{\Lambda,i}^s$ are indecomposable~$\K$-representations of~$\J_{\Lambda}/\J_{\Lambda}^1$ of finite length.  So we have
\[P_{\Lambda,i}\otimes \K\simeq \bigoplus_s\ind_{\J_{\Lambda}}^{\G}( \kappa_{\Lambda}\otimes \zeta_{\Lambda,i}^s).\]
Let~$e$ be an idempotent of~$\mathfrak{Z}_\R(\mathfrak{t})$, the centre of~$\Rep_{\R}(\mathfrak{t})$.   Then~$e$ acts on each~$\K$-summand\[\ind_{\J_{\Lambda}}^{\G}( \kappa_{\Lambda}\otimes \zeta_{\Lambda,i}^s)\] by a central idempotent of~$\mathfrak{Z}_{\K}(\mathfrak{t})$, and hence as~$\ind_{\J_{\Lambda}}^{\G}( \kappa_{\Lambda}\otimes \zeta_{\Lambda,i}^s)$ is in a unique~$\K$-block,~$e$ acts by either zero or the identity on~$\ind_{\J_{\Lambda}}^{\G}( \kappa_{\Lambda}\otimes \zeta_{\Lambda,i}^s)$.

This means that the action of~$e$ on $P_{\Lambda,i} \otimes \K$ is induced by an idempotent of (the centre of) $\End_{\K[\J_{\Lambda}]}(\overline{P}_{\Lambda,i}\otimes \K)$.  Explicitly, we define the idempotent in~$\End_{\K[\J_{\Lambda}]}(\overline{P}_{\Lambda,i}\otimes \K)$ to preserve the decomposition~$\overline{P}_{\Lambda,i}\otimes \K\simeq \bigoplus_{s} \kappa_{\Lambda}\otimes \zeta_{\Lambda,i}^s$ and to act on the summand~$ \kappa_{\Lambda}\otimes \zeta_{\Lambda,i}^s$ by either zero or the identity as prescribed by the action of~$e$ on~$\ind_{\J_{\Lambda}}^{\G}(\kappa_{\Lambda}\otimes \zeta_{\Lambda,i}^s)$.  So $e$ lies in the intersection (in $\End_{\R[\G]}(P_{i,j}) \otimes \K$) of $\End_{\R[\G]}(P_{\Lambda,i})$ and~$\End_{\R[\J_{\Lambda}]}(\overline{P}_{\Lambda,i}) \otimes \K\simeq \End_{\K[\J_{\Lambda}]}(\overline{P}_{\Lambda,i}\otimes \K)$.  But, by the Mackey formula,~$\End_{\R[\J_{\Lambda}]}(\overline{P}_{\Lambda,i})$ is a direct summand of~$\End_{\R[\G]}(P_{\Lambda,i})$ (the summand supported on the double coset containing the identity in fact). Thus $e$ arises from a central idempotent of $\End_{\R[\J_{\Lambda}]}(\overline{P}_{\Lambda,i})$; and since~$\overline{P}_{\Lambda,i}$ is indecomposable, $e$ is zero or the identity.  Therefore~$P_{\Lambda,i}$ lies in a unique~$\R$-block.
\end{proof}

We deduce the following result (known previously for~$\GL_n(\F)$ with~$\R=\W(\Fl)$ by \cite{HelmForum}):
\begin{corollary}\label{cormodell}
Suppose that~$\R$ is a complete discrete valuation ring (and $\mathbb{Z}[1/p,\mu_{p^r}]$-algebra), with principal maximal ideal~$\mathfrak{m}=(\varpi_{\mathfrak{m}})$, and residue field~$\varkappa=\R/\mathfrak{m}$ of characteristic~$\ell\not\in \{0,p\}$.  Suppose, for all~$\Lambda\in\mathrm{Vert}_{\beta}$,~$\varkappa$ is sufficiently large for~$\J_{\Lambda}/\J_{\Lambda}^1$ (contains all roots of unity dividing the~$\ell$-regular part of~$|\J_{\Lambda}/\J_{\Lambda}^1|$).  Then reduction modulo $\mathfrak{m}$, i.e., the map of projective generators~$P\mapsto P\otimes_{\R}\varkappa$, induces a bijection between the~$\R$-blocks in~$\Rep_{\R}(\ft)$ and the $\varkappa$-blocks in~$\Rep_{\varkappa}(\ft)$. 
%
\end{corollary}
\begin{proof}
We consider the fine~$(\mathfrak{t},\R)$-graph~$\mathcal{G}_{P\otimes\R}$ and the fine~$(\mathfrak{t},\varkappa)$-graph~$\mathcal{G}_{P\otimes \varkappa}$.  The vertex sets of these two graphs are in natural bijection by reduction mod~$\mathfrak{m}$: indeed, the vertices of~$\mathcal{G}_{P\otimes\R}$ correspond to the distinct projective summands~$P_{\Lambda,i}$ of the inductions of the $\eta_\Lambda\otimes \R$, and the vertices of~$\mathcal{G}_{P\otimes \varkappa}$ correspond to the mod~$\ell$ reductions~$Q_{\Lambda,i}=P_{\Lambda,i}\otimes \varkappa$ (which identify with the distinct projective summands of the inductions of the~$\eta_\Lambda\otimes\varkappa$).

If there is a nonzero map from $Q_{\Lambda,i}$ to $Q_{\Lambda',i'}$ then we can precompose with the reduction mod $\ell$ to get a map $P_{\Lambda,i} \rightarrow Q_{\Lambda',i'}$.  Then using projectivity of $P_{\Lambda,i'}$ we can lift this to a map $P_{\Lambda,i} \rightarrow P_{\Lambda',i'}$.  Conversely, if there is a nonzero map $P_{\Lambda,i} \rightarrow P_{\Lambda',i'}$ then we can let $n$ be the largest integer such that $\varpi_{\mathfrak{m}}^n P_{\Lambda',i'}$ contains the image of this map.  Dividing by $\varpi_{\mathfrak{m}}^n$ then gives us a map whose reduction mod $\ell$ is a nonzero map $Q_{\Lambda,i} \rightarrow Q_{\Lambda',i'}$. 
\end{proof}
\begin{remark}
 Under the same hypotheses as the corollary, it is expected (cf.,~\cite{DHKMfiniteness}) that, for any reductive~$p$-adic group~$\G$, the natural map\[\mathfrak{Z}_{\R}(\G)_r\otimes \varkappa\rightarrow \mathfrak{Z}_{\varkappa}(\G)_r\] between depth $r$ centres is an isomorphism. Under our assumptions on~$\G$ (either~$r=0$ or~$\G$ is an inner form of a general linear or classical group with~$p\neq 2$), our results show the natural map at least induces a bijection between the primitive idempotents.
 \end{remark}

\begin{corollary}\label{ZlbarCorr}
\begin{enumerate}
\item 
There~$\W(\Fl)$-blocks in~$\Rep_{\W(\Fl)}(\G)$ and the~$\Zl$-blocks in~$\Rep_{\Zl}(\G)$ are in bijection by the natural map, $\mathfrak{Z}_{\W(\Fl)}(\G)\rightarrow \mathfrak{Z}_{\W(\Fl)}(\G)\otimes \Zl\rightarrow \mathfrak{Z}_{\Zl}(\G)$, on primitive central idempotents.
%
\item The~$\Zl$-blocks in~$\Rep_{\Zl}(\G)$ and the~$\Fl$-blocks in~$\Rep_{\Fl}(\G)$ are in bijection by the natural map, $\mathfrak{Z}_{\Zl}(\G)\rightarrow \mathfrak{Z}_{\Zl}(\G)\otimes \Fl\rightarrow \mathfrak{Z}_{\Fl}(\G)$, on primitive central idempotents.
\end{enumerate}
\end{corollary}

\begin{proof}
This follows from using the variant from Remark \ref{remarkotherdecomps} \ref{remarkotherdecompsii}, as for finite groups the~$\Zl$-blocks are defined over~$\W(\Fl)$.   Hence, applying the Corollary \ref{cormodell} with~$\R=\W(\Fl)$, we deduce that the map of projective generators~$P\mapsto P\otimes_{\Zl}\Fl$, defines a bijection between the~$\Zl$-blocks in~$\Rep_{\Zl}(\ft)$ and the~$\Fl$-blocks in~$\Rep_{\Fl}(\ft)$.  Applying this to every endo-parameter, as~$\W(\Fl)$ is sufficiently large for all endo-parameters, we obtain the statement in the introduction that reduction modulo~$\ell$, given by taking a finitely generated projective generator~$P$ to~$P\otimes\Fl$, defines a bijection between the blocks of~$\Rep_{\Zl}(\G)$ and the blocks of~$\Rep_{\Fl}(\G)$.\end{proof}

\subsection{Application 1: The block decomposition over~$\overline{\mathbb{Z}}[1/p]$}
Let~$\G$ be a connected reductive~$p$-adic group and~$\mathcal{F}$ a facet in the Bruhat--Tits building of~$\G$.  We have the following result of Dat and Lanard about the central idempotents of~$\overline{\mathbb{Z}}[1/p][\G_{\mathcal{F}}^+/\G_{\mathcal{F}}^1]$:

\begin{lemma}[{Dat--Lanard \cite{DatLanard}}]\label{DatLanardLemma}
\begin{enumerate}
\item Suppose~$\Gf$ is a finite group of Lie type\footnote{That is, the fixed points of a twisted Frobenius morphism acting on a connected reductive group defined over~$\overline{\mathbb{F}}_p$.} over~$\mathbb{F}_p$, then~$\overline{\mathbb{Z}}[1/p][\Gf]$ has no non-trivial central idempotents.
\item Suppose that~$p$ does not divide~$|\G_{\mathcal{F}}^+/\G_{\mathcal{F}}|$, then~$\overline{\mathbb{Z}}[1/p][\G_{\mathcal{F}}^+/\G_{\mathcal{F}}^1]$ has no non-trivial central idempotents.
\end{enumerate}
\end{lemma}

\begin{proof}
\begin{enumerate}
\item This is Dat--Lanard, \cite[Theorem 2.0.1]{DatLanard}.
\item The group~$\G_{\mathcal{F}}/\G_{\mathcal{F}}^1$ is a finite group of Lie type over~$\mathbb{F}_p$,~$\G_{\mathcal{F}}^+/\G_{\mathcal{F}}^1$ contains~$\G_{\mathcal{F}}/\G_{\mathcal{F}}^1$ as a normal subgroup with abelian quotient, and this part follows from Clifford Theory and the first part.  Similar considerations are used in \cite[\S 3.4]{DatLanard}, but we give the full argument for completeness:

For any finite group~$\H$, Dat and Lanard show that~$\overline{\mathbb{Z}}[1/p][\H]$ has no non-trivial central idempotents if and only if there is only one equivalence class of irreducible complex representations of~$\H$ under the relation that irreducible complex representations~$\rho,\rho'$ of~$\H$ are equivalent if they are \emph{connected by a (finite) chain of~$\ell$-block coincidences over~$\ell\neq p$} : that is, there is a finite sequence of primes~$(\ell_i)$ and a finite sequence of irreducible complex representations~$(\rho_i)$ of~$\H$ such that
   \[\rho\sim_{\ell_1} \rho_1 \sim_{\ell_2}\rho_2\sim_{\ell_3} \cdots \sim_{\ell_s} \rho' ,\]
where~$\rho_i\sim_{\ell_{i+1}}\rho_{i+1}$ if~$\rho_i$ and~$\rho_{i+1}$ are in the same~$\ell_{i+1}$-block.   As the~$\ell_i$-block relation is invariant under field automorphisms, we can consider each~$\sim_{\ell_i}$ as connecting $\overline{\mathbb{Q}_{\ell^i}}$-representations.  The~$\ell_i$-block decomposition on the irreducible (necessarily integral)~$\overline{\mathbb{Q}_{\ell^i}}$-representations is given by the transitive closure of the relationship of having a common constituent on reduction modulo~$\ell_i$.  Thus the equivalence relation ``connected by a chain of~$\ell$-block coincidences over~$\ell\neq p$'' on the irreducible complex representations is equivalent to the equivalence relation ``connected by a chain having common subquotients on reduction mod primes~$\ell\neq p$''.

Let~$\Gf^+=\G_{\mathcal{F}}^+/\G_{\mathcal{F}}^1$ and~$\Gf=\G_{\mathcal{F}}/\G_{\mathcal{F}}^1$.  For~$\rho$ an irreducible complex representation of~$\Gf$, set~$\I_{\Gf^+}(\rho)=\{g\in\Gf^+: \rho^g\simeq \rho\}$ the \emph{inertia subgroup} of~$\rho$.  Now given~$\pi,\pi'\in\Irr(\Gf^+)$, by Clifford Theory there exist~$\rho,\rho'\in\Irr(\Gf)$,~$\widetilde{\rho}\in\Irr(\I_{\Gf^+}(\rho),\rho)$, and~$\widetilde{\rho}'\in\Irr(\I_{\Gf^+}(\rho'),\rho')$ such that
\[\pi\simeq \Ind_{\I_{\Gf^+}(\rho)}^{\Gf^+}(\widetilde{\rho}),\qquad \pi'\simeq \Ind_{\I_{\Gf^+}(\rho')}^{\Gf^+}(\widetilde{\rho}'),\]
and, as~$\Gf^+/\Gf$ is abelian,
\begin{align*}
\Ind_{\Gf}^{ \I_{\Gf^+}(\rho)}(\rho)&\simeq \bigoplus_{\chi\in\Hom( \I_{\Gf^+}(\rho)/\Gf,\mathbb{C}^\times)}\widetilde{\rho}\otimes \chi,\quad\Ind_{\Gf}^{ \I_{\Gf^+}(\rho')}(\rho')\simeq \bigoplus_{\chi'\in\Hom( \I_{\Gf^+}(\rho')/\Gf,\mathbb{C}^\times)}\widetilde{\rho}'\otimes \chi',
\end{align*}
Applying Dat and Lanard's result to~$\Gf$ we have
\[\rho\sim_{\ell_1} \rho_1 \sim_{\ell_2}\rho_2\sim_{\ell_3} \cdots \sim_{\ell_s}\rho'\]
and working along the chain we can reduce to showing if~$\rho\sim_{\ell}\rho'$ then we can connect~$\pi$ and $\pi'$ by a sequence of equivalences.  So suppose~$\rho,\rho'$ contain a common constituent~$\overline{\rho}$ on reduction modulo~$\ell$.   Then
\begin{align*}
\Ind_{\Gf}^{\Gf^+}(\rho)&\simeq \bigoplus_{\chi\in\Hom( \I_{\Gf^+}(\rho)/\Gf,\Ql^\times)}\Ind_{\I_{\Gf^+}(\rho)}^{\Gf^+}(\widetilde{\rho}\otimes \chi)\\
\Ind_{\Gf}^{\Gf^+}(\rho')&\simeq \bigoplus_{\chi'\in\Hom( \I_{\Gf^+}(\rho')/\Gf,\Ql^\times)}\Ind_{\I_{\Gf^+}(\rho')}^{\Gf^+}(\widetilde{\rho}'\otimes \chi'),
\end{align*}
give the direct sums into irreducibles of~$\Ind_{\Gf}^{\Gf^+}(\rho)$ and~$\Ind_{\Gf}^{\Gf^+}(\rho')$, and both contain~$\Ind_{\Gf}^{\Gf^+}(\overline{\rho})$ on reduction modulo~$\ell$.  Hence (at least) one summand of~$\Ind_{\Gf}^{\Gf^+}(\rho)$ is linked mod~$\ell$ to one summand of~$\Ind_{\Gf}^{\Gf^+}(\rho')$.  But all summands of~$\Ind_{\Gf}^{\Gf^+}(\rho)$ are linked by congruences mod primes dividing~$|\I_{\Gf^+}(\rho)/\Gf|$ (as all characters~$\chi$ are linked to the trivial character), and similarly all summands of~$\Ind_{\Gf}^{\Gf^+}(\rho')$ are linked by congruences mod the primes dividing~$|\I_{\Gf^+}(\rho')/\Gf|$, and we are done.
\end{enumerate}
\end{proof}

For~$\G$ an inner form of a general linear group or an inner form of a classical~$p$-adic group (with~$p\neq 2$), it follows from Lemmas \ref{DatLanardLemma} and \ref{lemmakappafunctor}, that the vertices of the fine~$(\mathfrak{t},\overline{\mathbb{Z}}[1/p])$-graph are given by the~$\ind_{\J_{\Lambda}^1}^{\G}(\eta_i\otimes \overline{\mathbb{Z}}[1/p])$.  Moreover, in this case, it is straightforward to see that there is only one connected component of the fine ~$(\mathfrak{t},\overline{\mathbb{Z}}[1/p])$-graph without invoking the equality with the coarse graph:  By Mackey theory
\begin{align*}\Hom_{\overline{\mathbb{Z}}[1/p][\G]}&(\ind_{\J_{\Lambda}^1}^{\G}(\eta_\Lambda\otimes \overline{\mathbb{Z}}[1/p]),\ind_{\J_{\Lambda'}^1}^{\G}(\eta_{\Lambda'}\otimes \overline{\mathbb{Z}}[1/p]))\\&=\bigoplus\Hom_{\overline{\mathbb{Z}}[1/p][\J_{\Lambda}^1\cap(\J_{\Lambda'}^1)^g]}(\eta_{\Lambda}\otimes \overline{\mathbb{Z}}[1/p],\eta_{\Lambda'}^g\otimes \overline{\mathbb{Z}}[1/p])\neq 0,\end{align*}
as all the~$\eta_\Lambda$ in the projective~$P(\ft)$ intertwine.  Hence we obtain:

\begin{corollary}\label{EPCor}
Let~$\G$ be an inner form of a general linear group or an inner form of a classical~$p$-adic group (with~$p\neq 2$).  The endo-factor~$\Rep_{\overline{\mathbb{Z}}[1/p]}(\mathfrak{t})$ is indecomposable.  In other words, the decomposition by endo-parameter of Theorem~\ref{endosplit} is the~$\overline{\mathbb{Z}}[1/p]$-block decomposition. 
\end{corollary}

Our methods also allow us to consider the trivial (depth zero) endo-parameter~$\mathfrak{t}=\mathbf{0}_{\G}$ of any connected reductive~$p$-adic group~$\G$ (recovering some results of Dat and Lanard \cite{DatLanard}):  

\begin{corollary}[{\cite{DatLanard}}]
Let~$\G$ be a connected reductive $p$-adic group.  Suppose that for any maximal parahoric subgroup~$\G_{\Lambda}$, the quotient~$\G_{\Lambda}^+/\G_{\Lambda}$ is of order prime to~$p$.  Then~$\Rep_{\overline{\mathbb{Z}}[1/p]}(\mathbf{0}_{\G})$ is indecomposable.
\end{corollary}

\begin{remark}
In particular, this includes the case where~$\G$ is semisimple and simply connected (where~$\G_{\Lambda}^+=\G_{\Lambda}$, cf.~\cite[Lemma 7.7.8]{KP}) \cite[Corollary 3.3.3]{DatLanard}.  However, while for classical groups with $p$ odd the depth zero block over~$\overline{\mathbb{Z}}[1/p]$ is always indecomposable (in this case~$\G_{\Lambda}^+/\G_{\Lambda}$ is a $2$-group), it is not the case in general.  Dat and Lanard introduce a group which acts transitively on the primitive idempotents in the depth zero centre \cite[Corollary 3.4.2]{DatLanard}, and show that if~$\G$ is quasi-split and tamely ramified then~$\Rep_{\overline{\mathbb{Z}}[1/p]}(\mathbf{0}_{\G})$ is indecomposable \cite[Theorem 3.6.1]{DatLanard}.  
\end{remark}

\subsection{Application 2: Reduction of the block decomposition to depth zero}
We say that the family~$\{\kappa_{\Lambda}:\Lambda\in\mathrm{Vert}_{\beta}\}$ of beta extensions defined over~$\R$ has \emph{full intertwining} if, for all~$\Lambda,\Lambda'\in\mathrm{Vert}_{\beta}$, we have~$\G_{\beta}\subset \I_{\G}(\kappa_{\Lambda},\kappa_{\Lambda'})$.  In \cite[Conjecture 6.8]{technicalpaper}, we conjecture such a family always exists (note that, this intertwining statement is equivalent to the statement over~$\K$:~$\G_{\beta}\subset \I_{\G}(\kappa_{\Lambda}\otimes\K,\kappa_{\Lambda'}\otimes\K)$).

\begin{theorem}\label{reductiontodepthzeromaintheorem}
Suppose that there exists a family of beta extensions~$\Xi=\{\kappa_{\Lambda}:\Lambda\in\mathrm{Vert}_{\beta}\}$, defined over~$\R$, with full intertwining.  Then the map
\[f_{\Xi}:\ind_{\J_\Lambda}^{\G}(\kappa_\Lambda\otimes \Q_{\Lambda,i})\mapsto \ind^{\G_{\beta}}_{\G_{\beta,\Lambda}}(\Q_{\Lambda,i}),\]
defines a graph isomorphism between the fine~$(\ft,\R)$-graph with the fine~$(\mathbf{0}_{\G_{\beta}},\R)$-graph.   In particular, the~$\R$-blocks of~$\Rep_{\R}(\ft)$ are in natural bijection with the~$\R$-blocks in~$\Rep_{\R}(\mathbf{0}_{\G_{\beta}})$.
\end{theorem}

\begin{proof}
Suppose that we have an edge between~$\ind_{\J_{\Lambda}}^{\G}(\kappa_{\Lambda}\otimes \Q_{\Lambda,i})$ and~$\ind_{\J_{\Lambda'}}^{\G}(\kappa_{\Lambda'}\otimes\Q_{\Lambda',i'})$, and set~$\mathcal{P}= \Q_{\Lambda,i}$ and~$\mathcal{P}'=\Q_{\Lambda',i'}$, then in other words
\[\Hom_{\R[\G]}(\ind_{\J_{\Lambda}}^{\G}(\kappa_{\Lambda}\otimes \mathcal{P}),\ind_{\J_{\Lambda'}}^{\G}(\kappa_{\Lambda'}\otimes \mathcal{P}'))\neq 0.\]
This Hom-space embeds into
\begin{align*}
\Hom_{\R[\G]}&(\ind_{\J_{\Lambda}}^{\G}(\kappa_{\Lambda}\otimes \mathcal{P}),\ind_{\J_{\Lambda'}}^{\G}(\kappa_{\Lambda'}\otimes \mathcal{P}'))\otimes\K\\
&\simeq  \Hom_{\K[\G]}(\ind_{\J_{\Lambda}}^{\G}(\kappa_{\Lambda}\otimes \mathcal{P})\otimes\K,\ind_{\J_{\Lambda'}}^{\G}(\kappa_{\Lambda'}\otimes \mathcal{P}')\otimes\K).\end{align*}
In particular, this is non-zero if and only if, there exists~$g\in\G$ such that,
\[\Hom_{\K[\J_{\Lambda}^g\cap\J_{\Lambda'}]}(\kappa_{\Lambda,\K}^g\otimes (\mathcal{P}^g\otimes\K),\kappa_{\Lambda',\K}\otimes (\mathcal{P}'\otimes\K))\neq 0.\]
Restricting to~$\J^1$-groups, we see that we can assume that~$g\in\G_{\beta}$, and moreover for such an element~$\Hom_{\K[\J_{\Lambda}^g\cap\J_{\Lambda'}]}(\kappa_{\Lambda,\K}^g,\kappa_{\Lambda',\K})\simeq \K$, hence by \cite[Lemma 2.7]{RKSS} we have an isomorphism of~$\K$-vector spaces
\[\Hom_{\K[\J_{\Lambda}^g\cap\J_{\Lambda'}]}( \mathcal{P}^g\otimes\K, \mathcal{P}'\otimes\K)\simeq \Hom_{\K[\J_{\Lambda}^g\cap\J_{\Lambda'}]}(\kappa_{\Lambda,\K}^g\otimes (\mathcal{P}^g\otimes\K),\kappa_{\Lambda',\K}\otimes (\mathcal{P}'\otimes\K)).\]
Hence~$\Hom_{\K[\J_{\Lambda}^g\cap\J_{\Lambda'}]}( \mathcal{P}^g\otimes\K, \mathcal{P}'\otimes\K)\neq 0$ and as
\[\Hom_{\R[\J_{\Lambda}^g\cap\J_{\Lambda'}]}( \mathcal{P}^g, \mathcal{P}')\otimes \K\simeq \Hom_{\K[\J_{\Lambda}^g\cap\J_{\Lambda'}]}( \mathcal{P}^g\otimes \K, \mathcal{P}\otimes\K),\]
we have~$\Hom_{\R[\G_{\beta,\Lambda_i}^g\cap\G_{\beta,\Lambda_{i'}}]}( \mathcal{P}^g, \mathcal{P}')\neq 0$, and we have a non-zero Hom (by Mackey theory) and hence an edge between~$ \ind^{\G_{\beta}}_{\G_{\beta,\Lambda_i}}(\mathcal{P})$ and $ \ind^{\G_{\beta}}_{\G_{\beta,\Lambda_{i'}}}(\mathcal{P}')$.  The reverse direction follows similarly.
\end{proof}

At this point, it is tempting to make a conjecture, generalizing a result of Chinello \cite{Chinello}, and related to predictions of Dat \cite{DatFunctoriality}:

\begin{conjecture}\label{equivconjecture}
Let~$(\G,h)$ be an inner form of a general linear group or of a classical group with~$p\neq 2$, and~$\mathfrak{t}$ be an endo-parameter for~$\G$.  Let~$\R$ be an integral domain and~$\mathbb{Z}[1/p,\mu_{p^\infty}]$-algebra.  We have an equivalence of categories~$\Rep_{\R}(\ft)\simeq \Rep_{\R}(\mathbf{0}_{\G_{\beta}})$.\end{conjecture}

\begin{remark}
\begin{enumerate}
\item Notice the category~$\Rep_{\R}(\mathbf{0}_{\G_{\beta}})$ depends on a choice of a full semisimple element~$\beta$ for the endo-parameter~$\ft$ and embedding~$\varphi$.  We expect that by extending the tame parameter field (cf. \cite{BHeffective,AKMSS}) to the self-dual semisimple case, one can canonically associate to an endo-parameter a tamely ramified~$\F$-group~$\G_{\T}$ such that
\[\Rep_{\R}(\mathbf{0}_{\G_{\beta}})\simeq \Rep_{\R}(\mathbf{0}_{\G_{\T}}),\] 
for any choice of~$(\varphi,\beta)$ for~$\ft$, and so that there exists a reduction to a \emph{tamely ramified} depth zero situation.  This would then provide an analogue of the reduction to the tame case for Langlands parameters of \cite{DHKM1}.
\item For classical groups Heiermann \cite{Heiermann} has shown an arbitrary Bernstein block $\Rep_{\mathbb{C}}(\mathfrak{s})$ is equivalent to a unipotent block in a related group (in particular to a depth zero block).  Such a reduction to the unipotent block is also predicted for more general coefficient rings by Dat \cite{DatFunctoriality}.  Note that in the case of the conjecture~over~$\mathbb{Z}[1/p,\mu_{p^\infty}]$, we have already shown that the relevant depth zero category is indecomposable, so the full depth zero subcategory~$\Rep_{\mathbb{Z}[1/p,\mu_{p^\infty}]}(\mathbf{0}_{\G_{\beta}})$ is the unipotent block.  Recently, in the tame setting for~$\mathbb{C}$-representations, a reduction to depth zero has been constructed using Hecke algebra isomorphisms by Adler--Fintzen--Mishra--Ohara \cite{adler2024reductiondepthzerotame}.
\end{enumerate}
\end{remark}

\section{Interpretation in terms of Langlands parameters}\label{secLPs}

We explain how our results on blocks fit into the local Langlands in families conjecture.   For this section we suppose that~$\F$ is a~$p$-adic field to allow us to apply results on the local Langlands correspondence of Arthur, Mok, and Kaletha--M\'inguez--Shin--White.  We denote by~$\Irr(\G)$ the set of isomorphism classes of~$\mathbb{C}$-representations of~$\G$, and by~$\Phi(\mathscr{W}_\F,\G)$ the set of~$\widehat{\G}(\mathbb{C})$-conjugacy classes of Langlands parameters~$\rho:\mathscr{W}_\F\times \SL_2(\mathbb{C})\rightarrow \LG(\mathbb{C})$ for~$\G$.  If~$\G$ is symplectic, unitary, or odd split special orthogonal we denote by
\[\LL_\G:\Irr(\G)\rightarrow \Phi(\mathscr{W}_\F,\G)\]
the local Langlands correspondence of \cite{Arthur,Mok,KMSW}.  For simplicity we do not consider non-split or even orthogonal groups in this section.  

We let~$ \Phi(\mathscr{W}_\F,\G)^{ss}$ denote the set of~$\widehat{\G}(\mathbb{C})$-conjugacy classes of \emph{semisimple Langlands parameters} for~$\G$
; we have a semisimplification map~$\Phi(\mathscr{W}_\F,\G)\rightarrow \Phi(\mathscr{W}_\F,\G)^{ss}$ restricting via the embedding~$\varkappa:\W_\F\rightarrow \W_\F\times \SL_2(\mathbb{C})$ given by $w\mapsto \left(w,\left(\begin{smallmatrix}|w|^{1/2}&0 \\0&|w|^{-1/2}\end{smallmatrix}\right)\right)$.   Under the above hypotheses, by \cite[Section 7]{DHKM2}, the local Langlands correspondence for~$\G$ induces a semisimple local Langlands correspondence
\[\LL_\G^{ss}:\Cusp(\G)\rightarrow \Phi(\mathscr{W}_\F,\G)^{ss},\]
from the set of~$\G$-conjugacy classes of cuspidal supports of irreducible~$\mathbb{C}$-representations of~$\G$ to the set of semisimple Langlands parameters for~$\G$.

We let~$\Phi(\mathscr{P}_\F,\G)$ be the set of~$\widehat{\G}(\mathbb{C})$-conjugacy classes of wild inertial types for~$\G$, that is the~$\widehat{\G}(\mathbb{C})$-conjugacy classes of the restrictions to~$\mathscr{P}_\F$ (via the embedding~$\varkappa$) of the representatives of all elements of~$\Phi(\mathscr{W}_\F,\G)$.

For a~$\mathbb{Z}[1/p]$-algebra~$\R$, we let~$\mathfrak{R}_{\LG,\R}$ denote the universal~$\R$-algebra for Langlands parameters (on a fixed discretized Weil group~$\mathscr{W}_{\F}^0$) for~$\G$ constructed in \cite{DHKM1}; it carries an action of~$\widehat{\G}$ and we let~$\mathfrak{R}_{\LG,\R}^{\widehat{\G}}$ denote the subalgebra of~$\widehat{\G}$-invariant functions (the GIT-quotient, which is independent of the choice of discretization by \cite[Theorem 4.18]{DHKM1}).

Suppose now that~$\G$ is~$\F$-quasi-split.  We let~$\mathfrak{E}_{\G,\mathbb{Z}[\sqrt{q}^{-1}]}=\prod_{r\in \D(\G)}\mathfrak{E}_{\G,\mathbb{Z}[\sqrt{q}^{-1}],r}$ denote the integral model for the endomorphisms of a Gelfand--Graev representation for~$\G$ defined in \cite[Section 5]{DHKM2}, where~$\D(\G)$ denotes the set of depths for~$\G$ as in Lemma \ref{basiccentrelemma}.  We let~$\mathfrak{Z}_{\G,\R}^{\mathrm{ad}}$ denote the subring of~$\mathfrak{Z}_{\G,\R}$ of elements invariant under the automorphisms induced by the conjugation by elements of the~$\F$-points of the adjoint group (cf., \cite[Remark 5.16]{DHKM2}).  Then, by ibid., there is a canonical map 
\[\mathfrak{e}:\mathfrak{Z}_{\G,\mathbb{Z}[\sqrt{q}^{-1}]}^{\mathrm{ad}}\rightarrow  \mathfrak{E}_{\G,\mathbb{Z}[\sqrt{q}^{-1}]}.\]
Given a Whittaker datum~$(\U,\psi)$  over~$\R_0=\mathbb{Z}[\sqrt{q}^{-1},\mu_{p^\infty}]$ for~$\G$, for each~$r\in\D(\G)$, there exists a canonical isomorphism
\[\mathfrak{E}_{\G,\mathbb{Z}[\sqrt{q}^{-1}],r}\otimes \R_0\xrightarrow{\sim}\End_{\R_0[\G]}(\ind_{\U}^{\G}(\psi)_r).\]

\subsection{A corollary of local Langlands in families}\label{LLIFsection}
 The local Langlands in families conjecture \cite{KOberwolfach, DHKM2}, predicts a natural morphism
\[ \mathrm{LLIF}_{\G}:\mathfrak{R}_{\LG,\mathbb{Z}[\sqrt{q}^{-1}]}^{\widehat{\G}}\rightarrow \mathfrak{Z}_{\G,\mathbb{Z}[\sqrt{q}^{-1}]},\]
compatible with the semisimple local Langlands correspondence for classical groups, with image in~$\mathfrak{Z}_{\G,\mathbb{Z}[\sqrt{q}^{-1}]}^{\mathrm{ad}}$ and the property that, if~$\G$ is~$\F$-quasi-split then the induced map obtained by composing with~$\mathfrak{e}$:
\[ \mathrm{LLIF}_{\G}:\mathfrak{R}_{\LG,\mathbb{Z}[\sqrt{q}^{-1}]}^{\widehat{\G}}\xrightarrow{\sim} \mathfrak{E}_{\G,\mathbb{Z}[\sqrt{q}^{-1}]}\]
is an isomorphism.  The morphism is expected to be at least loosely compatible with the depth filtration: choose a filtration~$(\mathscr{P}_{\F}^e)_{e\in\mathbb{N}}$ of~$\mathscr{P}_\F$ by open normal subgroups of~$\mathscr{W}_\F$ and define~$\mathfrak{R}_{\LG,\mathbb{Z}[\sqrt{q}^{-1}],e}^{\widehat{\G}}$ by considering only parameters trivial on~$\mathscr{P}_{\F}^e$, then given any depth~$r\in\D(\G)$ it is expected that there exists an~$e(r)\in\mathbb{N}$ such that~$\mathrm{LLIF}_{\G}$ factors through~$\bigcup_{e\leqslant e(r)}\mathfrak{R}_{\LG,\mathbb{Z}[\sqrt{q}^{-1}],e}^{\widehat{\G}}$, and conversely given~$e\in\mathbb{N}$ it is expected that there exists~$r(e)\in\D(\G)$ such that~$\mathrm{LLIF}_{\G}$ has image in~$ \mathfrak{Z}_{\G,\mathbb{Z}[\sqrt{q}^{-1}],\leqslant r(e)}$.

In \cite{DHKM2}, such morphisms are constructed after inverting an integer which depends on~$\G$ (``the banal case'').

The formation of these rings~$\mathfrak{R}_{\LG,\mathbb{Z}[\sqrt{q}^{-1}]}^{\widehat{\G}}, \mathfrak{Z}_{\G,\mathbb{Z}[\sqrt{q}^{-1}]}, \mathfrak{E}_{\G,\mathbb{Z}[\sqrt{q}^{-1}]}$ in restricted depth/ramification are compatible with (at least) flat extensions, and hence the existence of~$\mathrm{LLIF}_{\G}$ loosely compatible with the depth filtration as above,  gives the existence of morphisms~$\mathrm{LLIF}_{\G,\R}:\mathfrak{R}_{\LG,\R}^{\widehat{\G}}\rightarrow \mathfrak{Z}_{\G,\R}$ and, in the quasi-split case,~$\mathrm{LLIF}_{\G,\R}:\mathfrak{R}_{\LG,\R}^{\widehat{\G}}\rightarrow \mathfrak{E}_{\G,\R}$, for~$\R/\mathbb{Z}[\sqrt{q}^{-1}]$ flat.  For the remainder of this section, we suppose~$\R$ is a flat~$\mathbb{Z}[\sqrt{q}^{-1}]$-algebra.

\begin{remark}\label{RemarkFS}For~$\ell\neq p$, (for any connected reductive $p$-adic group~$\G$) in a spectacular breakthrough Fargues--Scholze in \cite{FarguesScholze} constructed a natural map
\[\mathrm{FS}_{\G,\ell}:\mathrm{Exc}(\W_{\F},\widehat{\G})_{\mathbb{Z}_{\ell}[\sqrt{q}]}\rightarrow \mathfrak{Z}_{\G,\mathbb{Z}_{\ell}[\sqrt{q}]},\]
from the ``excursion algebra'' to the centre over~$\mathbb{Z}_{\ell}[\sqrt{q}]$.   Scholze \cite{scholze2025geometrizationlocallanglandscorrespondence} has proved that this construction is independent of~$\ell$ establishing \cite[Conjecture I.9.5]{FarguesScholze}, and it follows (cf.,~\cite[Lemma A.3]{DHKM2}) that the maps~$\mathrm{FS}_{\G,\ell}$ can be obtained by scalar extension from a unique morphism
\[\mathrm{FS}_{\G}:\mathrm{Exc}(\W_{\F},\widehat{\G})_{\mathbb{Z}[\sqrt{q}^{-1}]}\rightarrow \mathfrak{Z}_{\G,\mathbb{Z}[\sqrt{q}^{-1}]}.\]
It is expected that~$\mathrm{Exc}(\W_{\F},\widehat{\G})_{\mathbb{Z}[\sqrt{q}^{-1}]}$ and~$\mathfrak{R}_{\LG,\mathbb{Z}[\sqrt{q}^{-1}]}^{\widehat{\G}}$ are closely related (cf.~\cite[VIII.5]{FarguesScholze}).  In any case, they show~that there is a natural universal homeomorphism~$\mathrm{Exc}(\W_{\F},\widehat{\G})_{\mathbb{Z}[\sqrt{q}^{-1}]}\rightarrow \mathfrak{R}_{\LG,\mathbb{Z}[\sqrt{q}^{-1}]}^{\widehat{\G}}$, and in particular their primitive idempotents are in natural bijection. As in \cite[1.1]{DatLanard}, one can apply understanding of the primitive idempotents~$ \mathfrak{R}_{\LG,\R}^{\widehat{\G}}$ and of~$\mathfrak{Z}_{\G,\R}$, to understand properties of Fargues--Scholze's semisimple correspondence.
\end{remark}

As~$\widehat{\G}$ is connected, the primitive idempotents of~$\mathfrak{R}_{\LG,\R}^{\widehat{\G}}$ correspond the connected components of the full moduli space of Langlands parameters, which are studied in various cases in~\cite{DHKM1} and in \cite{cotner2024connected}.

\begin{definition}
Suppose~$\mathrm{LLIF}_{\G,\R}$ exists.  Then, via $\mathrm{LLIF}_{\G,\R}$, primitive idempotents of~$\mathfrak{R}_{\LG,\R}^{\widehat{\G}}$ correspond to sums of primitive idempotents of~$\mathfrak{Z}_{\G,\R}$.  
\begin{enumerate}
\item If an idempotent~$e\in\mathfrak{Z}_{\G,\R}$ is of the form $\mathrm{LLIF}_{\G,\R}(f)$ for an idempotent~$f\in\mathfrak{R}_{\LG,\R}^{\widehat{\G}}$ then we call it \emph{stable}.  In other words, it defines an idempotent of~the \emph{stable centre}
\[\mathfrak{Z}_{\G,\R}^{\st}:=\mathrm{LLIF}_{\G,\R}(\mathfrak{R}_{\LG,\R}^{\widehat{\G}}).\]
\item When it exists, we call the decomposition~$1=\sum e_i$ of~$1\in \mathfrak{Z}_{\G,\R}^{\st}$ into pairwise orthogonal primitive idempotents of~$\mathfrak{Z}_{\G,\R}^{\st}$, the \emph{stable $\R$-block decomposition} (again in general one will need some finiteness conditions on~$\R$ to guarantee existence).  We call the~$e_i$ \emph{ps-idempotents}.
\end{enumerate}
\end{definition}

Suppose~$\R$ is Noetherian, then given a ps-idempotent~$e\in\mathfrak{Z}_{\G,\R}^{\st}$ we can consider its decomposition~$e=\sum e_i$ into primitive orthogonal idempotents~$e_i\in\mathfrak{Z}_{\R}(\G)$.  By Corollary \ref{finitelymanyblocksofdepth}, we deduce:

\begin{lemma}
Suppose~$\R$ is a Noetherian domain,~$\mathrm{LLIF}_{\G,\R}$ exists, and is compatible with the depth filtration in the weak sense described above.  Then the ps-idempotents are finite sums of the primitive idempotents of~$\mathfrak{Z}_{\G,\R}$.
\end{lemma}

For quasi-split groups, we can also single out a distinguished idempotent:
\begin{lemma}
Suppose~$\G$  is quasi-split,~$\R$ is a Noetherian~${\mathbb{Z}}[\mu_{p^\infty},1/p]$-algebra, the morphism~$\mathrm{LLIF}_{\G,\R}$ exists, defines an isomorphism to~$\End_{\R[\G]}(\ind_{\U}^{\G}(\psi))$, and~$e\in\mathfrak{Z}_{\G,\R}^{\st}$ is a ps-idempotent.  In the decomposition $e=\sum e_i$ of a ps-idempotent into primitive idempotents, there is a unique idempotent $e_{\psi, gen}$ whose image under the natural map~$\Phi$ to~$\End_{\R[\G]}(\ind_{\U}^{\G}(\psi))$ is non-zero.  
\end{lemma}
Note that~$e_{\psi, gen}$ is the unique summand of~$e$ which supports~$\psi$-generic representations.

\begin{proof}
Let~$e'\in\mathfrak{R}_{\LG,\R}^{\widehat{\G}}$ be the unique primitive idempotent such that\[\mathrm{LLIF}_{\G,\R}(e')=e.\]  Then~$\Phi\circ \mathrm{LLIF}_{\G,\R}(e')=e''$ is a primitive idempotent of~$\End_{\R[\G]}(\ind_{\U}^{\G}(\psi))$ as~$\Phi\circ \mathrm{LLIF}_{\G,\R}$ is an isomorphism.  Moreover,~$e''=\Phi(\sum e_i)=\sum\Phi( e_i)$ is a decomposition of~$e''$ into orthogonal idempotents.  Hence, there is a unique~$i$ such that~$\Phi( e_i)\neq 0$ and for this~$i$ we have~$e''=\Phi(e_i)$.
 \end{proof}
 
 In terms of our description of the idempotents of~$\mathfrak{Z}_{\overline{\mathbb{Z}}[1/p]}(\G)$ in terms of endo-parameters for~$\G$, we let~$\mathcal{EP}(\G)$ denote the set of unrefined endo-parameters for~$\G$ of \cite[Remark 3.10]{technicalpaper}. 
 \begin{conjecture}\label{conjepstable}
We have a natural bijection (i.e., a unique bijection compatible with local Langlands)
\[ \left\{\text{primitive idempotents of }\mathfrak{Z}_{\G,\overline{\mathbb{Z}}[1/p]}^{\st}\right\} \leftrightarrow \mathcal{EP}(\G),\]
i.e.,~the ps-idempotents are given by endo-parameters forgetting their Witt data.  
\end{conjecture}

\subsection{The case of symplectic groups}
Let~$\G=\Sp_{2n}(\F)$.  Over~$\overline{\mathbb{Z}}[1/p]$ (using $p\neq 2$) the connected components of~$\mathfrak{R}_{\LG,\R}^{\widehat{\G}}$ correspond to wild inertial types by \cite{DHKM1}.  So we can reinterpret the Conjecture \ref{conjepstable} as saying that there is a unique bijection
 \[\Phi(\mathscr{P}_\F,\G) \leftrightarrow \mathcal{EP}(\G)\]
compatible with local Langlands.   In this special case, this becomes a mild extension of the ramification theorem of the fourth author, Blondel, and Henniart \cite{BHS}:

\begin{theorem}
There is a unique bijection~$\mathrm{LL}_{\G}^{\text{wild}}:\mathcal{EP}(\G)\rightarrow \Phi(\mathscr{P}_\F,\G)$ which is compatible with the local Langlands correspondence. 
\end{theorem}

\begin{proof}
Up to semisimplification, the map
\[\LL_\G:\Irr(\G)\rightarrow \Phi(\mathscr{W}_\F,\G)\]
is compatible with parabolic induction and induces a unique map
\[\LL_\G^{ss}:\Cusp(\G)\rightarrow \Phi(\mathscr{W}_\F,\G)^{ss},\]
from cuspidal supports to orbits of semisimple parameters \cite{DHKM2}.  The unrefined endo-parameter map factors through~$\Cusp(\G)$, and restriction to wild inertia through~$\Phi(\mathscr{W}_\F,\G)^{ss}$.  The statement for~$\LL_{\M}$ follows from \cite{BHS} for cuspidal representations of Levi subgroups~$\M$ of~$\G$, and parabolic induction (i.e., compatibility of~$\LL_\G$ with parabolic induction up to semisimplification \cite{DHKM2}) in general.
\end{proof}

We thus obtain the following description of our decomposition by unrefined endo-parameter:

\begin{corollary}\label{corgaloisdecomp}
Let~$\G=\Sp_{2n}(\F)$.  We have a decomposition of categories~\[\Rep_{\overline{\mathbb{Z}}[1/p]}(\G)=\prod_{\nu\in \Phi(\mathscr{P}_\F,\G)}\Rep_{\overline{\mathbb{Z}}[1/p]}(\nu)\]
where a smooth~$\overline{\mathbb{Z}}[1/p]$-representation lies in~$\Rep_{\overline{\mathbb{Z}}[1/p]}(\nu)$ if and only if it has unrefined endoparameter~$(\mathrm{LL}_{\G}^{\text{wild}})^{-1}(\nu)$.  Moreover, ~$\pi\in \Irr(\G)$ lies in~$\Rep_{\overline{\mathbb{Z}}[1/p]}(\nu)$ if and only if~$\mathrm{LL}(\pi)\vert_{\P_\F}\simeq \nu$.
\end{corollary}

\begin{remark}
For~$\mathfrak{t}\in \Phi(\mathscr{P}_\F,\G)$, to compute~$(\mathrm{LL}_{\G}^{\text{wild}})^{-1}( \mathfrak{t})$ explicitly, assuming the compatible picture we have, we can use any~$\pi\in\Irr(\G)$ with endoparameter~$\mathfrak{t}$ such that we can compute~$\LL(\pi)\mid_{\mathscr{P}_\F}$.  Using compatibility of~$\LL_\G$ with parabolic induction, for~$\mathfrak{t}$ with~$\M(\mathfrak{t})=\T$ (a maximal split torus in~$\Sp_{2n}(\F)$), we recover
\[(\mathrm{LL}_{\G}^{\text{wild}})^{-1}( \mathfrak{t})=(\phi_{\presuper{L}\T\hookrightarrow \LG}\circ\LL_{\T}(\chi))\mid_{{\mathscr{P}_\F}},\]
where~$\phi_{\presuper{L}\T\hookrightarrow \LG}$ is the (unique up to conjugacy) inclusion of~$L$-groups associated to~$\T\hookrightarrow \G$, and~$\chi\in\Irr(\T)$ is such that~$i_{\T,\B}^{\G}(\chi)$ has endo-parameter~$\mathfrak{t}$.  In particular, the special case when~$\chi$ is trivial, shows that if~$\pi\in\Irr(\G)$ has depth zero, then~$\LL_{\G}(\pi)$ is tame.
\end{remark}


\bibliographystyle{plain}
\bibliography{./Endosplitting}

\end{document}